\documentclass[11pt]{amsart}

\usepackage[utf8]{inputenc}
\usepackage[T1]{fontenc}
\usepackage{geometry,graphicx,float}
\usepackage{amsthm}
\usepackage{amsmath,amssymb}
\usepackage{xcolor,esint}
\definecolor{winered}{rgb}{0.5,0,0}
\usepackage[colorlinks=true]{hyperref}
\hypersetup{
    pdfauthor={O.~Junge,D.~Matthes,B.~Schmitzer},
    pdftitle={Entropic transfer operators},
	allcolors={winered}
}
\usepackage[section]{placeins}
\usepackage{caption}
\usepackage{subcaption}
\usepackage{pgf}
\usepackage{tikz}
\usetikzlibrary{arrows,automata,shapes}
\usepackage[deletedmarkup=xout]{changes}

\title{Entropic transfer operators}
\author{Oliver Junge$^\dagger$, Daniel Matthes$^\dagger$}
\address{$^\dagger$Department of Mathematics, TUM School of Computation, Information and Technology, Technical University of Munich, Germany}
\author{Bernhard Schmitzer$^\ddagger$}
\address{$^\ddagger$Institute of Computer Science, University of Göttingen, Germany}
\email{oliver.junge@tum.de, matthes@ma.tum.de, schmitzer@cs.uni-goettingen.de}

\date{\today}

\newcommand{\Z}{{\mathbb Z}}

\newcommand{\R}{{\mathbb R}}
\newcommand{\C}{{\mathbb C}}
\newcommand{\Rnn}{{\mathbb R}_{\ge0}}
\newcommand{\eps}{\varepsilon}
\newcommand{\cX}{\mathcal{X}}

\newcommand{\cY}{\mathcal{Y}}
\newcommand{\cA}{\mathcal{A}}

\newcommand{\setN}{\mathcal{N}}
\newcommand{\setM}{\mathcal{M}}
\newcommand{\bP}{\mathbf{P}}
\newcommand{\dd}{\;\mathrm{d}}
\newcommand{\dn}{\mathrm{d}}
\newcommand{\prb}{\mathcal{P}}
\newcommand{\opt}{\text{opt}}
\newcommand{\ent}{\mathbf{H}}

\newcommand{\id}{\operatorname{id}}
\newcommand{\COT}{\mathcal{C}}
\newcommand{\rightweaks}{\stackrel{\ast}{\rightarrow}}
\newcommand{\cub}{{\mathbf C}}
\newcommand{\dst}{\mathrm{d}}
\newcommand{\opn}[1]{\left\|#1\right\|_\text{op}}

\newtheorem{theorem}{Theorem}
\newtheorem{lemma}{Lemma}
\newtheorem{proposition}{Proposition}
\newtheorem{remark}{Remark}

\begin{document}

\begin{abstract}
We propose a new concept for the regularization and discretization of transfer and Koopman operators in dynamical systems.  Our approach is based on the entropically regularized optimal transport between two probability measures.  In particular, we use optimal transport plans in order to construct a finite-dimensional approximation of some transfer or Koopman operator which can be analysed computationally.  We prove that the spectrum of the discretized operator converges to the one of the regularized original operator, give a detailed analysis of the relation between the discretized and the original peripheral spectrum for a rotation map on the $n$-torus and provide code for three numerical experiments, including one based on the raw trajectory data of a small biomolecule from which its dominant conformations are recovered.  
\end{abstract}

\maketitle

\section{Introduction}
\subsection{Context and motivation}

A time-discrete dynamical system can be described by a state space $\cX$ and an evolution map $F:\cX \to \cX$ that specifies how points move from one time instance $t$ to the next: $x_{t+1}=F(x_t)$. In the case of an underlying (ordinary) differential equation, $F=F^{\tau}$ can be chosen to be the flow map for some fixed time step $\tau>0$. A classical goal in dynamical systems theory is then to characterize how the trajectory $(x_t)_t$ of some initial state $x_0$ behaves, e.g.~whether it is (almost) periodic or remains in some subset of $\cX$ for a long time.

However, in complicated systems (i.e.~chaotic or high-dimensional) typical individual trajectories can often not be described in simple geometric terms.  In chaotic systems, a trajectory exhibits \emph{sensitive dependence} on its initial point, and thus the trajectory of a single initial point is of little descriptive value.  In this case, it is more appropriate to describe the system in a statistical way: Consider an $\cX$-valued random variable $X$ with distribution $\mu$, then the distribution of $F(X)$ is given by the \emph{push forward}\footnote{with $F^{-1}(A)=\{x\in\cX:F(x)\in A\}$ denoting the \emph{preimage} of some (measurable) set $A$} $\mu\circ F^{-1}$ of  $\mu$ under the map $F$.  So instead of the nonlinear map $F$ on $\cX$ we use the linear map $F_\#:\mu\mapsto\mu\circ F^{-1}$ on probability measures over $\cX$, called the \emph{transfer} (or \emph{Frobenius-Perron}) \emph{operator} in order to describe the dynamics.  Relatedly, one may consider how some continuous \emph{observable} $\varphi:\cX\to\R$ is evolved through $F$, i.e.~$\varphi(x_{t+1})=\varphi(F(x_t))$, which yields the linear operator $K:\varphi \mapsto \varphi \circ F$, the \emph{Koopman operator}. This statistical point of view has been developed in ergodic theory, cf. \cite{Ko:31,vN:32,lasota2013chaos}.

Since these two operators are linear, tools from linear functional analysis can in principle be applied in order to analyse them. However, the infinite dimensionality of the underlying function spaces and the non-compactness of the operators make a numerical approximation non-trivial.
As a remedy, regularized versions of the operators can be considered which, e.g., result from (small) random perturbations of $F$ \cite{Ki:88, lasota2013chaos}.  In a suitable functional analytic setting, these regularized operators are compact, therefore have a discrete spectrum, and can be approximated numerically by a Galerkin method \cite{ding1999projection,DeJu99}. 

The resulting matrix gives a simple means for a forward prediction of the dynamics by mere matrix-vector multiplication \cite{lasota2013chaos,BuMoMe:12}.
It can also be used in order to construct a reduced predictive model of the system \cite{crommelin2006reconstruction,nuske2019coarse,klus2020data}.
Finally, the nodal domains of eigenvectors at eigenvalues of large magnitude can be used in order to decompose the state space into macroscopic (almost-)cyclic or almost-invariant (i.e.~metastable) components, yielding a coarse-grained model of the system  \cite{DeJu99}.

This approach works well if the (Hausdorff) dimension of $\cX$ is small, i.e.\ $\leq 3$.  Then, standard spaces can be used for the Galerkin approximation, e.g., piecewise constant functions, resulting in \emph{Ulam's method} \cite{Ulam60,li1976finite}.  In the case of a high dimensional $\cX$, but a low dimensional attractor/invariant set, set oriented methods \cite{DeHo97,DeJu99} can be employed  in order to restrict the computations to the invariant set.  Recently, approximation methods have been proposed which are suited for problems with dynamics on a high dimensional invariant set.  \emph{(Extended) dynamic mode decomposition} \cite{rowley2009spectral,williams2015data,williams2016kernel} is a Petrov-Galerkin approximation of the Koopman operator with different types of ansatz functions and point evaluations as test functions \cite{klus2016numerical}.

\subsection{Contribution and outline}

The contribution of this paper is the construction of an approximation method for the transfer and via its adjoint also for the Koopman operator which also works for dynamics on high dimensional invariant sets and only uses (trajectory) data. It combines a regularization of the transfer operator through an entropically regularized optimal transport plan \cite{Villani-OptimalTransport-09,Santambrogio-OTAM,PeyreCuturiCompOT} with a discretization through the restriction to point masses.  It bears similarities with the construction using Gaussian kernels in \cite{KoltaiRenger2018,froyland2021spectral}. In contrast to this and (extended) dynamic mode decomposition, though, our approach inherits important structural properties of the original operator: Our continuous and our discrete regularized operators are Markov, and they preserve the invariant measure (in the discrete setting: the uniform measure on the sample points). 

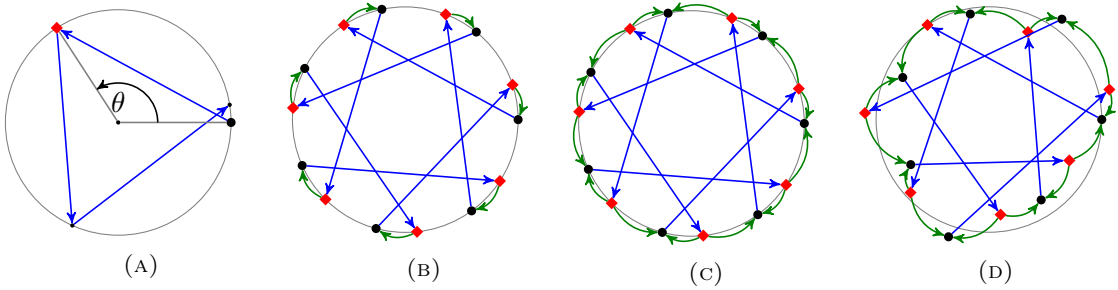
\begin{figure}[ht]
\centering
\begin{subfigure}[c]{0.24\textwidth}
    \begin{tikzpicture}[
    	scale=0.75, ->, 
    	>=stealth', 
    	auto, 
    	point/.style = {draw, circle,  fill = black, inner sep = 1.0pt},
    	imgpoint/.style = {draw=red, diamond,  fill = red, inner sep = 1.0pt},
    	dot/.style = {draw, circle,  fill = black, inner sep = 0.1pt},
    	semithick]
    
    	\def\rad{2}
    
    	\draw[gray,thin] (0,0) circle (\rad); 
    	\node (0) at +(0:0) [dot] {};
    	\node (1) at +(0:\rad) [point] {};
    	\node (2) at +(123:\rad) [imgpoint] {};
    	\node (3) at +(246:\rad) [dot] {};
    	\node (4) at +(369:\rad) [dot] {};
    	\path (1) edge [blue] (2)
	          (2) edge [blue] (3)
	          (3) edge [blue] (4);
	    \draw[gray,-] (0) -- (2);      
	    \draw[gray,-] (0) -- (1); 
	    \draw (0,0) +(0:0.7) arc (0:123:.7);  
	    \node [above] at (0) {$\theta$};   
     
    \end{tikzpicture}
    \caption{}
    \end{subfigure}
\begin{subfigure}[c]{0.24\textwidth}
    \begin{tikzpicture}[
    	scale=0.75, ->, 
    	>=stealth', 
    	auto, 
    	point/.style = {draw, circle,  fill = black, inner sep = 1.0pt},
    	imgpoint/.style = {draw=red, diamond,  fill = red, inner sep = 1.0pt},
    	dot/.style = {draw, circle,  fill = black, inner sep = 0.1pt},
    	semithick]
    
    	\def\rad{2}
    
   	    \draw[gray,thin] (0,0) circle (\rad); 
    	\foreach \i in {0,...,6}
		{
    		\node[point] (\i) at (\i*51:\rad) {};
	   		\node[imgpoint] (\i+7) at (\i*51+123:\rad) {};
		}
   		\path (0) edge [blue] (0+7);
   		\path (0+7) edge [green!50!black, bend left] (2);
   		\path (1) edge [blue] (1+7);
   		\path (1+7) edge [green!50!black, bend left] (3);
   		\path (2) edge [blue] (2+7);
   		\path (2+7) edge [green!50!black, bend left] (4);
   		\path (3) edge [blue] (3+7);
   		\path (3+7) edge [green!50!black, bend left] (5);
   		\path (4) edge [blue] (4+7);
   		\path (4+7) edge [green!50!black, bend left] (6);
   		\path (5) edge [blue] (5+7);
   		\path (5+7) edge [green!50!black, bend left] (0);
   		\path (6) edge [blue] (6+7);
   		\path (6+7) edge [green!50!black, bend left] (1);

    \end{tikzpicture}
    \caption{}
\end{subfigure}
\begin{subfigure}[c]{0.24\textwidth}
    \begin{tikzpicture}[
    	scale=0.75, ->, 
    	>=stealth', 
    	auto, 
    	point/.style = {draw, circle,  fill = black, inner sep = 1.0pt},
    	imgpoint/.style = {draw=red, diamond,  fill = red, inner sep = 1.0pt},
    	dot/.style = {draw, circle,  fill = black, inner sep = 0.1pt},
    	semithick]
    
    	\def\rad{2}
    
   	    \draw[gray,thin] (0,0) circle (\rad); 
    	\foreach \i in {0,...,6}
		{
    		\node[point] (\i) at (\i*51:\rad) {};
	   		\node[imgpoint] (\i+7) at (\i*51+123:\rad) {};
		}

   		\path (0) edge [blue] (0+7);
   		\path (1) edge [blue] (1+7);
   		\path (2) edge [blue] (2+7);
   		\path (3) edge [blue] (3+7);
   		\path (4) edge [blue] (4+7);
   		\path (5) edge [blue] (5+7);
   		\path (6) edge [blue] (6+7);

   		\path (0+7) edge [green!50!black, bend left] (2);
   		\path (0+7) edge [green!50!black, bend right] (3);
   		\path (1+7) edge [green!50!black, bend left] (3);
   		\path (1+7) edge [green!50!black, bend right] (4);
   		\path (2+7) edge [green!50!black, bend left] (4);
   		\path (2+7) edge [green!50!black, bend right] (5);
   		\path (3+7) edge [green!50!black, bend left] (5);
   		\path (3+7) edge [green!50!black, bend right] (6);
   		\path (4+7) edge [green!50!black, bend left] (6);
   		\path (4+7) edge [green!50!black, bend right] (0);
   		\path (5+7) edge [green!50!black, bend left] (0);
   		\path (5+7) edge [green!50!black, bend right] (1);
   		\path (6+7) edge [green!50!black, bend left] (1);
   		\path (6+7) edge [green!50!black, bend right] (2);
    \end{tikzpicture}
    \caption{}
\end{subfigure}
\hfill
\begin{subfigure}[c]{0.24\textwidth}
    \quad 
    \begin{tikzpicture}[
    	scale=0.75, ->, 
    	>=stealth', 
    	auto, 
    	point/.style = {draw, circle,  fill = black, inner sep = 1.0pt},
    	imgpoint/.style = {draw=red, diamond,  fill = red, inner sep = 1.0pt},
    	dot/.style = {draw, circle,  fill = black, inner sep = 0.1pt},
    	semithick]
    
    	\def\rad{2}
    
   	    \draw[gray,thin] (0,0) circle (\rad); 
    	\node[point] (0) at +(0:\rad) {};
    	\node[point] (1) at +(51+3:\rad+0.2) {};
    	\node[point] (2) at +(102-2:\rad-0.1) {};
    	\node[point] (3) at +(153+1:\rad-0.3) {};
    	\node[point] (4) at +(204+6:\rad-0.4) {};
    	\node[point] (5) at +(255-4:\rad+0.2) {};
    	\node[point] (6) at +(306-3:\rad-0.3) {};

    	\node[imgpoint] (7) at +(0+123:\rad) {};
    	\node[imgpoint] (8) at +(51+3+123:\rad+0.2) {};
    	\node[imgpoint] (9) at +(102-2+123:\rad-0.1) {};
    	\node[imgpoint] (10) at +(153+1+123:\rad-0.3) {};
    	\node[imgpoint] (11) at +(204+6+123:\rad-0.4) {};
    	\node[imgpoint] (12) at +(255-4+123:\rad+0.2) {};
    	\node[imgpoint] (13) at +(306-3+123:\rad-0.3) {};

   		\path (0) edge [blue] (7);
   		\path (1) edge [blue] (8);
   		\path (2) edge [blue] (9);
   		\path (3) edge [blue] (10);
   		\path (4) edge [blue] (11);
   		\path (5) edge [blue] (12);
   		\path (6) edge [blue] (13);

    	\path (7) edge [green!50!black, bend left] (2);
    	\path (7) edge [green!50!black, bend right] (3);
    	\path (8) edge [green!50!black, bend left] (3);
    	\path (8) edge [green!50!black, bend right] (4);
    	\path (9) edge [green!50!black, bend left] (4);
    	\path (9) edge [green!50!black, bend right] (5);
    	\path (10) edge [green!50!black, bend left] (5);
    	\path (10) edge [green!50!black, bend right] (6);
    	\path (11) edge [green!50!black, bend left] (6);
    	\path (11) edge [green!50!black, bend right] (0);
    	\path (12) edge [green!50!black, bend left] (0);
    	\path (12) edge [green!50!black, bend right] (1);
    	\path (13) edge [green!50!black, bend left] (1);
    	\path (13) edge [green!50!black, bend right] (2);

    \end{tikzpicture}
    \caption{}
    \end{subfigure}
\caption{(A) The unit circle $\cX=S^1$ with the shift map $F: \varphi \mapsto \varphi + \theta$ (blue) with $\theta \approx 2\pi/3$ is a simple dynamical system with an approximate 3-cycle.
(B) A discretization $\cX^N$ with 7 equidistant points (black) composed with a deterministic map (green) of $F(\cX^N)$ (red) back to $\cX^N$ yields a single cycle of length 7 that does not reflect the approximate 3-cycle.
(C) A composition with an entropic optimal transport map (green) yields a Markov chain on $\cX^N$  which inherits the almost 3-cycle (this fact is not deducible from this visualization).
(D) This still holds when $\cX^N$ does not exactly lie on the attractor/invariant set and when it is distributed at random.}
\label{fig:Intro}
\end{figure}

The approach presented here approximates the deterministic map  $F$ on $\cX$ by a stochastic model on a discrete set $\cX^N$, which is a finite collection of points from $\cX$.  By continuity of the  map $F$, it is reasonable to expect that macroscopic features are still observed by ``probing'' the dynamics only in $\cX^N$ (if this set represents $\cX$ sufficiently well).  However, it is not enough to simply consider the deterministic map $F$ restricted to $\cX^N$ since $\cX^N$ is usually not invariant under $F$.  A deterministic identification of points in $F(\cX^N)$ and $\cX^N$ bears the danger that the resulting dynamics is dominated by combinatorial discretization artifacts (see Figure \ref{fig:Intro} for an illustration and Section \ref{sec:TorusSpectrumDiscussion} for discussion).

Instead, we propose a stochastic identification of $F(\cX^N)$ and $\cX^N$ by means of an optimal transport plan with entropic regularization between distributions on $F(\cX^N)$ and $\cX^N$.  Under this plan, the masses of the image points in $F(\cX^N)$ are distributed such that each point in $\cX^N$ receives precisely a unit amount of mass.
The optimal entropic transport plan reflects spatial proximity, i.e., the mass of a point in $F(\cX^N)$ goes preferably to points in $\cX^N$ that are nearby, but it also introduces some stochastic blur on a length scale that depends on some regularization parameter $\eps$.

Our approximation of the transfer operator has several attractive properties. First, as already mentioned, it is a Markovian approximation, and it preserves the invariant measure. Second, the blur introduced by the entropic regularization masks dynamic phenomena below the length scale of $\eps^{1/2}$, and in particular can be used to remove discretization artefacts. By varying $\eps$, dynamical phenomena on specific scales can be studied. The blur further leads to compactness of the approximated transfer operator, and thus to a separation of its spectrum into a few ``large'' eigenvalues of modulus close to one, and the remaining ``small'' eigenvalues close to zero; the large eigenvalues relate to macroscopic dynamical features like cyclic or metastable behaviour, cf.\ \cite{DeJu99}. Finally, the Sinkhorn algorithm provides a tool to numerically solve the underlying regularized optimal transport problems efficiently.

In \cite{OTCoherentSet2021}, an approach for the detection of coherent sets in non-autonomous dynamical systems has been proposed which also uses the concept of optimal transport.  There, it is assumed that two measurements of $\mu$ and $F_{\#}\mu$ are given, $F$ is unknown and $\mu$ is not assumed to be an invariant measure of $F$. The transfer operator $F_\#$ is estimated by the optimal transport between $\mu$ and $F_{\#}\mu$, and approximately coherent sets are subsequently estimated. Entropic smoothing provides a simple numerical algorithm, compactness of $F_\#$ and validity of the clustering step.
In contrast, we assume that an approximation $\mu^N$ of $\mu$ and the behaviour of $F$ on $\mu^N$ are known where now $\mu$ is assumed to be an invariant measure of $F$.  Again, entropic smoothing provides a simple numerical algorithm and compactness of the transfer operator. 

We recall basic properties on transfer operators and (entropic) optimal transport in Sections \ref{sec:transfer_operators} and \ref{sec:ot}.
In Section \ref{sec:Method}, we introduce the regularized version $T^\eps$ of the original transfer operator $T=F_\#$ and the discrete transfer operator $T^{N,\eps}$ and establish the convergence of (an extension of) $T^{N,\eps}$ to $T^{\eps}$ in norm. 
As a proof of concept, in Section \ref{sec:shift map} we analytically study the application to the shift map on the $d$-torus to provide intuition for the interplay of the length scales associated to eigenfunctions of $T$, discretization index $N$, and entropic regularization $\eps$.
Several numerical examples are given in Section \ref{sec:Examples}.

\subsection{General hypotheses and notations}
Throughout this manuscript, $(\cX,\dst)$ is a compact metric space, and $F:\cX\to\cX$ is a continuous map. We shall occasionally assume that~$\cX$ is a compact subset of $\R^d$, but only when explicitly mentioned. Unless stated otherwise, the cost function $c:\cX\times\cX\to\Rnn$ for optimal transportation is the canonical one, $c(x,y)=\dst(x,y)^2$. We denote the set of probability measures on $X$ by $\prb(X)$ and use the notation $h\mu$ in order to denote a measure which is absolutely continuous with density $h$ with respect to the measure $\mu$.  Weak convergence of a sequence $(\mu^N)_N$ of measures to a limiting measure $\mu$ is denoted by $\mu^N \rightweaks \mu$.

\section{Background on transfer operators}
\label{sec:transfer_operators}

\subsection{Push-forward and invariant measures}
Choose some points $x_1,\ldots,x_N\in \cX$ independently according to some probability measure $\mu$ on $\cX$, i.e.\ such that the probability $\bP(x_j\in B)$ to find $x_j$ in some Borel set $B$ is $\mu(B)$.  Then
\begin{align*}
    \bP(F(x_j)\in A) = \bP(x_j\in F^{-1}(A)) = \mu(F^{-1}(A)), 
\end{align*}
i.e.\ the probability to find some image point $F(x_j)$ in some Borel set $A$ is $\mu(F^{-1}(A))$ where we denote by $F^{-1}(A)=\{x\in \cX\mid F(x)\in A\}$ the \emph{preimage} of the set $A\subset \cX$. Phrasing this differently: If some points $x_1,\ldots,x_N$ are distributed according to the probability measure $\mu$ then their image points $F(x_1),\ldots,F(x_N)$ are distributed according to the probability measure
\begin{equation}\label{eq:push_fwd}
F_\#\mu := \mu(F^{-1}(\cdot)),
\end{equation}
called the \emph{push-forward measure} of $\mu$ (under $F$).
The mapping $F_\#:\mu\mapsto F_\#\mu$ defines a linear operator on the space of bounded signed measures, called \emph{transfer operator} (or \emph{Frobenius-Perron operator}).  Clearly, if $\mu$ is a linear combination of point measures at some points $x_1,\ldots,x_N\in\cX$, i.e.\ $\mu=\sum_j a_j\delta_{x_j}$, $a_j\in\R$, then
\begin{equation}\label{eq:point_push_fwd}
F_\# \mu = \sum_j a_j\delta_{x_j}(F^{-1}(\cdot)) = \sum_j a_j\delta_{F(x_j)}.
\end{equation}
Additionally, from (\ref{eq:push_fwd}) we directly get that $F_\# \mu \geq 0$ and $\int \dd F_\# \mu = \int \dd \mu$ for any non-negative measure $\mu$, i.e.\ $F_\#$ is a \emph{Markov operator}. 

A probability measure which is a fixed point of $F_\#$ is called an \emph{invariant measure}. Invariant measures provide information about, e.g., recurrent behaviour of points:

\begin{theorem}[Poincaré]
Let $\mu$ be an invariant measure for $F$. For some Borel set $A\subset \cX$ let $A_0$ be the subset of points from $A$ which return infinitely often to $A$ under iteration with~$F$. Then $\mu(A_0)=\mu(A)$.
\end{theorem}

\subsection{Ulam's method}

Invariant measures can be approximated by eigenvectors of some matrix approximation of $F_\#$. A particularly popular approximation results from \emph{Ulam's method} \cite{Ulam60}:
Fix a non-negative reference measure $m$ (e.g.~the Lebesgue measure or a measure concentrated on some invariant set $\cA \subset \cX$, e.g.~some attractor) and an $m$-essentially disjoint covering $A_1,\ldots,A_n$ of the support of $\mu$. Instead of solving the eigenproblem $F_\# \mu = \mu$ over all non-negative measures and for all measurable sets, we parametrize
$$\mu= \sum_{j=1}^n h_j \frac{m(A_j \cap \cdot)}{m(A_j)}
\quad \text{for non-negative coefficients } h_1,\ldots,h_n,$$
and only enforce the condition $F_\# \mu = \mu$ on the sets $A_j$,
$$h_j = \mu(A_j) = (F_\# \mu)(A_j) = \sum_{k=1}^n P_{jk} \cdot h_k
\quad \text{with coefficients } P_{jk} := \frac{m(A_k \cap F^{-1}(A_j))}{m(A_k)},$$
resulting in the discrete eigenproblem $Ph = h$ on $\R^{n}$.
$P$ is a Markov matrix, i.e.\ its columns sum to 1 and its entries give the conditional probability to map from element $A_k$ into $A_j$ of the covering. 

A simple way to implement Ulam's method is to choose a suitable set of sample points in each $A_j$ (e.g.\ on a grid or randomly) and to compute the location of their images in order to approximate the transition probabilities $P_{jk}$ \cite{Hsu81,Hunt98}. Convergence (as the number $n$ of covering sets goes to infinity) is slow \cite{Mu01} since the ansatz is of low regularity.

If one covers the entire state space $\cX$ by the $A_j$, the numerical effort for constructing the covering scales exponentially in the dimension of $\cX$.  In the case that the support of the invariant measure under consideration is low dimensional, this can be avoided by computing a covering of the support first \cite{DeHo97,DeJu99}.   

\subsection{Other eigenvalues}

As mentioned, invariant measures are fixed points of $F_\#$, i.e.\ eigenmeasures at the eigenvalue $1$. But other eigenpairs might also be of interest. Roots of unity reveal (macroscopic) cycling behaviour, while real eigenvalues close to $1$ can be used to detect almost invariant sets, i.e.\ metastable dynamics \cite{DeJu99}: Let $F_\# \mu=\mu$. If $F_\# \nu = \lambda \nu$, $\nu=h\mu$, $h \in L^1(\mu)$, for some real $\lambda < 1$, $\lambda\approx 1$, then the sets $A^-:=\{ h < 0 \}$ and $A^+:=\{ h > 0\}$ have the property that the ``internal'' transition probability
\[
p(A^\pm) := \frac{\nu(A^\pm \cap F^{-1}(A^\pm))}{\nu(A^\pm)}
\]
to stay in $A^-$ resp.\ $A^+$ under one iteration of $F$ is close to $1$.  We will demonstrate this kind of macroscopic behaviour in the experiments in Section~\ref{sec:experiments}.
This can immediately be applied to other eigenpairs of the matrix $P$ from Ulam's method.
In the general case this spectral analysis may be more convenient on densities, see below.

\subsection{Push-forward for densities}
\label{sec:PshfwdDens}

Sometimes it is more convenient to restrict oneself to measures that are absolutely continuous with respect to some reference measure $\mu$. The push-forward under $F$ induces a linear map $T : L^1(\mu) \to L^1(F_\# \mu)$
where for $h\in L^1(\mu)$, $T h$ is characterized by
\begin{align} 
\label{eq:Top}
    \int_{\cX}T h(y) \; \varphi(y)  \dd\big(F_\# \mu)(y) = \int_{\cX} h(x) \; \varphi(F(x))  \dd\mu(x),
\end{align}
for continuous test functions $\varphi : \cX \to \R$, or, more concisely,
\begin{align}
    \label{eq:dfnT} 
    T h:=\frac{\dn(F_\# (h\mu))}{\dn(F_\# \mu)}. 
\end{align}
If $F$ is a homeomorphism this simplifies to $T h=h\circ F^{-1}$ and becomes independent of $\mu$.
For any convex function $J$ one can show with Jensen's inequality \cite[Lemma 3.15]{LinHK2021} that
\begin{align}
	\label{eq:Jensen}
    \int_{\cX} J\circ T h\,\dd (F_\# \mu) \le \int_{\cX}J\circ h\dd\mu,
\end{align}
hence $T h$ indeed lies in $L^1(F_\# \mu)$ and in fact the appropriate restrictions of $T$ are continuous linear operators from $L^p(\mu)$ to $L^p(F_\#\mu)$ for any $p\geq 1$.

If $\mu$ happens to be invariant under $F$, i.e., $F_\# \mu=\mu$, then $T$ is actually a continuous linear self-mapping on each $L^p(\mu)$, and in particular on the Hilbert space $L^2(\mu)$. 
The definition of $T$ above, however, makes sense even if $\mu$ is not invariant. 
The necessity to consider transfer operators with respect to non-invariant measures
arises e.g.~when $\mu$ is an approximation to an invariant measure by point masses,
in which case $F_\#\mu$ is not only different from $\mu$, but typically not even 
absolutely continuous with respect to $\mu$, making spectral analysis impossible.
The method we introduce below modifies $T$ by means of optimal transport in such
a way that it is always a self-mapping on $L^p(\mu)$, thus making spectral analysis
feasible again.

\subsection{The Koopman operator}

As alluded to in the introduction, there is a dual view on the description of dynamics via operators originally conceived in \cite{Ko:31}, which has received a lot of interest in the past two decades, see \cite{Me:05,rowley2009spectral,BuMoMe:12,klus2016numerical} and references therein: given some continuous scalar field $\varphi:\cX\to\R$, often called \emph{observable}, define the \emph{Koopman operator} $K:C(\cX)\to C(\cX)$ by
\[
K\varphi := \varphi \circ F.
\]
The characterization \eqref{eq:Top} shows that $K$ is the formal adjoint to the transfer operator. Indeed, for $p=2$ we have
\begin{equation}
    \label{eq:duality}
\langle Th,\varphi \rangle_{L^2(F_\#\mu)}
= \langle h, K\varphi \rangle_{L^2(\mu)}
\end{equation}
for $h\in L^2(\mu)$, $\varphi\in C(\cX)$. By duality, spectral properties of one operator directly carry over to the adjoint one. In particular, the spectra of $T$ and $K$ are identical.


\section{Background on plain and entropic optimal transport}
\label{sec:ot}
Extensive introductions to optimal transport can be found in \cite{Villani-OptimalTransport-09,Santambrogio-OTAM}, computational aspects are treated in \cite{PeyreCuturiCompOT}. Here we briefly summarize the ingredients required for our method.
\subsection{The original optimal transport problem}
The basic problem in the mathematical theory of optimal transport (OT) is the following: given two mass distributions, represented by probability measures $\mu,\nu\in\prb(\cX)$, and a function $c:\cX\times\cX\to\R$ that assigns the cost $c(x,y)$ to the transfer of one unit of mass from $x\in\cX$ to $y\in\cX$, find the most cost efficient way to ``re-distribute'' $\mu$ into $\nu$. The first mathematically sound definition of to ``re-distribute'' was due to Monge \cite{MongeOT1781} in terms of maps $\Phi:\cX\to\cX$ that transport $\mu$ to $\nu$, i.e.
\begin{align}
    \label{eq:mapsto}
    \Phi_\#\mu=\nu.
\end{align}
In that class of transport maps, one seeks to minimize the total cost associated with $\Phi$, given by
\begin{align}
    \label{eq:monge}
    \int_{\cX}c(x,\Phi(x))\dd\mu(x). 
\end{align}
This formulation is very intuitive, since $\Phi$ assigns to each original mass position $x$ a target position $\Phi(x)$. An optimal map $\Phi$ need not exist in general. A prototypical obstacle is that if $\mu$ charges some single point $x$ with positive mass (which is the situation of interest here), then this mass may need to be split into several parts going to different target points $y$ charged by $\nu$. 

\subsection{Kantorovich relaxation and Wasserstein distances}
\label{sec:OTKantorovich}
Additional flexibility is provided by Kantorovich's relaxation. It is formulated with so-called transport plans, which are probability measures $\gamma\in\prb(\cX\times\cX)$ with marginals $\mu$ and $\nu$ respectively. We denote the set of transport plans by
\begin{align}
    \label{eq:constraint}
    \Gamma(\mu,\nu) := \left\{ \gamma \in \prb(\cX \times \cY) \,|\, \pi^1_\#\gamma=\mu,\, \pi^2_\#\gamma=\nu \right\},
\end{align}
where $\pi^i : \cX \times \cX \to \cX$ denotes the map $(x_1,x_2) \mapsto x_i$ for $i=1,2$.
The constraints correspond to the condition \eqref{eq:mapsto}.
For measurable $A, B \subset \cX$, one interprets $\gamma(A \times B)$ as the mass that is transported from $A$ to $B$.
Accordingly, the minimization problem \eqref{eq:monge} becomes
\begin{align}
    \label{eq:kantorovich}
    \COT(\mu,\nu) := \inf_{\gamma \in \Gamma(\mu,\nu)} \int_{\cX\times\cX}c(x,y)\dd\gamma(x,y).
\end{align}
Unlike for \eqref{eq:monge}, admissible $\gamma$ always exist, e.g.~the product measure $\mu \otimes \nu \in \Gamma(\mu,\nu)$.
Existence of an optimal plan $\gamma_\opt$ follows from mild assumptions, e.g.~when $c$ is continuous \cite{Villani-OptimalTransport-09}.
For $\cX \subset \R^d$, $c(x,y)=\dst(x,y)^2=\|x-y\|^2$, Brenier's fundamental theorem \cite{MonotoneRerrangement-91} says that $\gamma_\opt$ is actually not far from being a map. Namely, its support lies in the graph of the subdifferential of a convex function.  Efficient numerical solvers for the optimal transport problem all use the sparsity of the support of $\gamma_\opt$.

In our case, i.e.\ when $c(x,y) := \dst(x,y)^2$, the map $(\mu,\nu) \mapsto W_2(\mu,\nu):=\sqrt{\COT(\mu,\nu)}$ is called the 2-Wasserstein distance. It is a metric on $\prb(\cX)$ that metrizes weak* convergence. For non-compact metric $\cX$ additional conditions on the moments of $\mu$, $\nu$ apply. The construction works analogously for arbitrary exponents $p \in [1,\infty)$.

\subsection{Entropic regularization}

The basic optimal transport problem \eqref{eq:kantorovich} is convex but not strictly convex. That ``flatness'' is often an obstacle, both in the analysis and for an efficient numerical treatment. A way to overcome it is \emph{entropic regularization}. It amounts to augmenting the variational functional in \eqref{eq:kantorovich} by a regularizing term: For fixed $\eps>$, let
\begin{align}
    \label{eq:entropic}
    \COT^\eps(\mu,\nu) := \inf_{\gamma \in \Gamma(\mu,\nu)} \int_{\cX\times\cX}c(x,y)\dd\gamma(x,y) + \eps\ent(\gamma\mid\theta),
\end{align}
where $\ent(\cdot\mid\theta)$ is the probabilistic entropy with respect to some non-negative reference measure $\theta$ on $\cX\times\cX$,
\[ 
\ent(\gamma\mid\theta) := \int_{\cX\times\cX} \left(\log\frac{\dn\gamma}{\dn\theta}\right)\dd\gamma
\quad \text{if } \gamma \ll \theta,\,\gamma \geq 0 
\quad\text{and} \quad\ent(\gamma\mid\theta)=\infty \text{ otherwise.} 
\]
The appropriate choice of $\theta$ depends on the context at hand; in this article we shall always assume that 
\[ \theta = \mu\otimes\nu \]
when solving for $\COT^\eps(\mu,\nu)$.
For that problem, the unique minimizer $\gamma^\eps_\opt$ can be written in the form
\begin{align}
    \label{eq:entropicsolution}
    \gamma^\eps_\opt = g^\eps\theta, \quad\text{where}\quad
    g^\eps(x,y) = \exp\left(\frac{-c(x,y)+\alpha(x)+\beta(y)}{\eps}\right),
\end{align}
for some Lagrange multipliers $\alpha,\beta:\cX\to\R$ that realize the marginal constraints~\eqref{eq:constraint}. The multipliers $\alpha$~and $\beta$ are unique up to a constant additive shift.
The marginal constraints imply that $g^\eps$ is bistochastic in the following sense:
\begin{equation}
\label{eq:gBistochastic}
\int_{\cX} g^\eps(x,y')\dd \nu(y') = 1, \quad
\int_{\cX} g^\eps(x',y)\dd \mu(x') = 1 \quad
\text{for $\mu \otimes \nu$-almost all $(x,y)$.}
\end{equation}
This implies the following condition on $\alpha$:
\begin{equation}
\label{eq:EntropicOptimality}
\alpha(x) = -\eps \cdot \log 
\int_{\cX} \exp\left(\frac{-c(x,y)+\beta(y)}{\eps}\right) \dd \nu(y) 
\end{equation}
for $\mu$-almost every $x$ and the analogous condition for $\beta$.
These conditions holding $\mu$-almost everywhere for $\alpha$ and $\nu$-almost everywhere for $\beta$ are sufficient and necessary for optimality of $\gamma^\eps$ as given in~\eqref{eq:entropicsolution}.
After discretization, these multipliers can be calculated very efficiently via the Sinkhorn algorithm, see below.
Another consequence is the loss of sparsity: while minimizers of the unregularized problem tend to be sparse (possibly even concentrated on a graph, see \cite{MonotoneRerrangement-91}) the support of $\gamma^\eps_\opt$, minimizing the regularized problem, is that of $\theta=\mu \otimes \nu$. When a sparse, deterministic transport is sought, the diffusivity would be considered a nuisance.
But it can also be helpful in reducing discretization artifacts, which is beneficial in our application, the minimizing $\gamma^\eps_\opt$ is always unique (unlike for $\eps=0$), and regularization improves the regularity of the map $\COT^\eps$, which becomes differentiable with respect to $\mu$ and $\nu$, see for instance \cite{FeydyDissertation}.

\subsection{Discretization and Sinkhorn algorithm}
\label{sec:Sinkhorn}
When $\mu$ and $\nu$ are supported on finite subsets $\setM=\{x_1,\ldots,x_M\}$ and $\setN=\{y_1,\ldots,y_N\}$ of $\cX$, then $\mu$ and $\nu$ are identified with tuples $(\mu_1,\ldots,\mu_M)$ and $(\nu_1,\ldots,\nu_N)$ of non-negative numbers that sum up to one, and plans $\gamma$ can be identified with non-negative matrices $(\gamma_{ij})$. The admissibility condition in \eqref{eq:constraint} becomes
\begin{align}
    \label{eq:d-constraint}
    \sum_{j=1}^N\gamma_{ij}=\mu_i \text{ for all } i, \qquad \sum_{i=1}^M\gamma_{ij} = \nu_j \text{ for all } j.
\end{align}
The reference measure $\theta$ can also be represented by a discrete matrix of coefficients.
With the short-hand notation $c_{ij}=c(x_i,y_j)$, the entropic problem \eqref{eq:entropic} becomes
\begin{align}
\label{eq:d-entropic}
\COT^\eps(\mu,\nu) = \inf_{\gamma \in \Gamma(\mu,\nu)} \sum_{i,j} \gamma_{ij}\left(c_{ij} + \eps\log(\gamma_{ij}/\theta_{ij})\right),
\end{align}
with $\theta_{ij}=\mu_i\,\nu_j$.
In analogy to \eqref{eq:entropicsolution}, its solution is given by
\begin{align}
	\label{eq:d-entropicsolution}
    \gamma_{ij}^\eps = \exp\left(\frac{-c_{ij}+\alpha_i +\beta_j}{\eps}\right) \cdot \theta_{ij}
\end{align}
with Lagrange multipliers $(\alpha_i)$ and $(\beta_j)$ that satisfy the constraint \eqref{eq:d-constraint}. As above, $(\alpha_i)$, $(\beta_j)$ are unique up to a constant additive shift.

Problem \eqref{eq:d-entropic} can be solved numerically via the Sinkhorn algorithm \cite{Cuturi13} which consists of alternatingly adjusting $(\alpha_i)$ and $(\beta_j)$ such that the constraints $\pi^1_\# \gamma=\mu$ and $\pi^2_\# \gamma=\nu$ are satisfied, leading to the discretized version of \eqref{eq:EntropicOptimality}. That is, for an initial $(\beta^{(0)}_j)$, e.g.~all zeros, one sets for $\ell=1,2,3,\ldots$,
\begin{equation}
\label{eq:Sinkhorn}
    \begin{split}
	\exp(\alpha^{(\ell)}_i/\eps) & = \left[\sum_{j=1}^N \exp\left(\frac{-c_{ij}+\beta^{(\ell-1)}_j}{\eps}\right) \cdot \nu_j \right]^{-1}, \\
	\exp(\beta^{(\ell)}_i/\eps) & = \left[\sum_{i=1}^M \exp\left(\frac{-c_{ij}+\alpha^{(\ell)}_i}{\eps}\right) \cdot \mu_i \right]^{-1}.
    \end{split}
\end{equation}

The Sinkhorn algorithm has reached great popularity due to its ``lightspeed'' performance \cite{Cuturi13}, at least on problems with a moderate number of points and a sufficiently large regularization $\eps>0$. More algorithmic details for small $\eps$ and larger pointclouds are explored in \cite{SchmitzerScaling2019,FeydyDissertation}. There are also efficient GPU implementations that work in high ambient dimensions and without memory issues, such as the KeOps and geomloss libraries \cite{KeOps,feydy2019interpolating}.
In our numerical experiments, solving the regularized transport problems has always been faster than the eigen-decomposition of the discrete Markov matrices and thus the former does not introduce a computational bottleneck into the method.

\section{Approximation method}
\label{sec:Method}

We now introduce our method for the approximation and spectral analysis of the transfer operator introduced in Section~\ref{sec:transfer_operators}.
First, we observe that an optimal transport plan also induces a transfer operator in Section \ref{sec:OTTransfer}. Then we introduce the entropically smoothed version $T^\eps: L^p(\mu) \to L^p(\mu)$ of the original operator $T$ in Section \ref{sec:TEps} and subsequently its approximate version $T^{N,\eps} : L^p(\mu^N) \to L^p(\mu^N)$ in Section \ref{sec:TNEps}.
Finally, we introduce an extension $\hat{T}^{N,\eps} : L^p(\mu) \to L^p(\mu)$ of $T^{N,\eps}$ and show convergence of the extension to $T^{\eps}$ in $L^2$-operator norm as $N \to \infty$ in Section \ref{sec:Convergence}. This implies convergence of the spectrum of $\hat{T}^{N,\eps}$ to that of $T^{\eps}$ (see Section \ref{sec:ConvergenceSpectra}). The relation between the spectra of $T^{N,\eps}$ and its extension $\hat{T}^{N,\eps}$ is elaborated upon in Section \ref{sec:SpectrumRelation}.

\subsection{Transfer operators induced by transport plans}

\label{sec:OTTransfer}
A transport plan $\gamma \in \Gamma(\mu,\nu)$, computed with or without entropic regularization, or possibly also non-optimal, induces a linear map $G : L^p(\mu) \to L^p(\nu)$ characterized by
\begin{equation}
\label{eq:GCouplingPairing}
	\int_{\cX} \varphi(y)\,\big(G h\big)(y) \dd \nu(y)
	= \int_{\cX \times \cX} \varphi(y)\,h(x) \dd \gamma(x,y).
\end{equation}
With entropic regularization, when $\gamma=g(\mu \otimes \nu)$ is of the form \eqref{eq:entropicsolution}, one obtains
\[\big(G h\big)(y) = \int_{\cX} h(x)\,g(x,y) \dd\mu(x).\]
Taking into account \eqref{eq:gBistochastic} and with Jensen's inequality, similar to \eqref{eq:Jensen} this yields
\[ \int_{\cX} J\big(G h(y)\big)\dd \nu(y) \le \int_{\cX\times\cX} J(h(x))\dd \gamma(x,y)
 = \int_{\cX} J(h(x))\dd \mu(x), \]
and therefore $G$ indeed maps $L^p(\mu) \to L^p(\nu)$. A similar representation and the same conclusion can be drawn for general $\gamma \in \Gamma(\mu,\nu)$ by using disintegration \cite[Theorem 5.3.1]{AmbrosioGradientFlows2008}.

\subsection{Entropic regularization of the transfer operator}
\label{sec:TEps}

For some $\mu$ let $T$ be the transfer operator $L^p(\mu) \to L^p(F_\# \mu)$ as introduced in Section \ref{sec:PshfwdDens}. We do not assume $F_\# \mu=\mu$.
Let $\gamma^{\eps}$ be the optimal $\eps$-entropic transport plan from $F_\# \mu$ to $\mu$, i.e.~the minimizer corresponding to \eqref{eq:entropic}, and let $g^{\eps}$ be its density, i.e.~$\gamma^{\eps}=g^{\eps} (F_\# \mu\otimes \mu)$, see \eqref{eq:entropicsolution}.
In this case the marginal conditions \eqref{eq:gBistochastic} become
\begin{align}
\label{eq:marginalRe}
\int_{\cX} g^{\eps}(x',x) \dd F_\# \mu(x') & = 1, &
\int_{\cX} g^{\eps}(y,y') \dd \mu(y') & = 1.
\end{align}
for $\mu$-almost all $x$ and $F_\#\mu$-almost all $y$.

As discussed in Section \ref{sec:OTTransfer}, $\gamma^{\eps}$ induces a map $L^p(F_\# \mu) \to L^p(\mu)$ which we will denote by $G^{\eps}$. It is given by
\begin{equation}
\label{eq:Gentropic}
(G^{\eps}h)(y) = \int_{\cX} h(x) g^{\eps}(x,y) \dd F_\# \mu(x).
\end{equation}

Now, set $T^{\eps} := G^{\eps} T$ as the composition of the two maps, which takes $L^p(\mu)$ into itself. This is illustrated in the left panel of Figure \ref{fig:Operators}. Using \eqref{eq:dfnT}, we obtain the representation
\begin{align}
	\label{eq:Tentropic}
	(T^{\eps} h)(y) = (G^{\eps} T h)(y) 
	& = \int_{\cX} (T h)(x)\,g^{\eps}(x,y) \dd F_\# \mu(x) \\
	& = \int_{\cX} g^{\eps}(x,y) \dd F_\# (h\mu)(x) \nonumber \\
    & = \int_{\cX} h(x) \; g^{\eps}(F(x),y)  \dd \mu(x), \nonumber
\end{align}
i.e.\ we obtain an integral operator with a smoothing kernel as also used in \cite{DeJu99} and later in \cite{Fr:13} in the context of characterizing coherent sets in non-autonomous systems.
Using \eqref{eq:marginalRe} one finds that the measure $g^{\eps}(F(\cdot),\cdot) (\mu \otimes \mu)$ lies in $\Gamma(\mu,\mu)$ and by the form of the density $g^{\eps}(F(\cdot),\cdot)$,
\begin{align}
	\label{eq:gFdens}
	g^{\eps}(F(x),y) = \exp\left(\frac{-c(F(x),y)+\alpha(F(x))+\beta(y)}{\eps}\right),
\end{align}
we deduce that it is the optimal $\eps$-entropic coupling between $\mu$ and itself with respect to the cost function $c(F(\cdot),\cdot)$ and $\theta=\mu\otimes \mu$, see \eqref{eq:entropicsolution}.
We will denote this density by
\begin{align}
	\label{eq:tdens}
	t^{\eps}(x,y) := g^{\eps}(F(x),y).
\end{align}

\subsection*{The regularized Koopman operator}

We can use duality in order to derive a regularized version of the Koopman operator from \eqref{eq:Tentropic}. The adjoint ${G^{\eps*}}:L^2(\mu)\to L^2(F_\#\mu)$ of $G^\eps:L^2(F_\#\mu)\to L^2(\mu)$ is
\[
({G^{\eps*}}f)(x) = \int_\cX f(y)\;g^\eps(x,y) \dd\mu(y).
\]
From \eqref{eq:duality}, we get for $h\in L^2(\mu),\varphi\in C(\cX)$:
\begin{align*}
\langle T^\eps h,\varphi \rangle_{L^2(\mu)}
& = \langle G^\eps Th,\varphi \rangle_{L^2(\mu)} 
 = \langle h, K{G^{\eps*}}\varphi \rangle_{L^2(\mu)},
\end{align*}
which leads us to define the \emph{entropically regularized Koopman operator} to be $K^\eps:C(\cX)\to C(\cX)$,
\begin{align*}
K^\eps & := K\circ {G^{\eps*}},
&
(K^\eps f)(x) & = \int_\cX f(y) \; g^\eps(F(x),y) \dd \mu(y).
\end{align*}

\subsection*{\texorpdfstring{Consistency for $\eps\to 0$}{Consistency for eps to 0}} We clarify the consistency of the representation of $T^\eps : L^p(\mu) \to L^p(\mu)$ in \eqref{eq:Tentropic} with the classical transfer operator $T : L^p(\mu) \to L^p(F_\# \mu)$ from \eqref{eq:Top} in the limit $\eps\downarrow0$. We assume for simplicity that the unregularized optimal transport $\gamma^0$ from $F_\#\mu$ to $\mu$ is given by an invertible transport map $\Psi:\cX\to\cX$, i.e.~$\gamma^0=(\id,\Psi)_\#F_\#\mu$. In that case, as $\eps\downarrow0$, the entropic optimal plan $\gamma^{\eps}$ converges weakly to $\gamma^0$, see for instance \cite{Carlier-EntropyJKO-2015}. Now let $\varphi\in C(\cX)$ be a test function, and assume that $h\in L^p(F_\#\mu)$ is continuous. The representation of $T^\eps$ from \eqref{eq:Tentropic} and the definition of $T$ from \eqref{eq:dfnT} imply:
\begin{align*}
    \int_{\cX}\varphi(y)\,\big(T^{\eps}h\big)(y)\dd\mu(y) 
    &= \int_{\cX\times\cX}\varphi(y)\,\big(T h\big)(x)\,g^{\eps}(x,y)\dd\big(F_\#\mu\otimes\mu\big)(x,y) \\
    &= \int_{\cX\times\cX}\varphi(y)\,\big(T h\big)(x)\dd \gamma^{\eps}(x,y) \\
    &\stackrel{\eps\to0}{\longrightarrow} \int_{\cX\times\cX}\varphi(y)\,\big(T h\big)(x)\,\dd \gamma^0(x,y) \\
    &=\int_{\cX}\varphi\circ\Psi(x)\,\big(Th\big)(x)\dd F_\#\mu(x),
\end{align*}
which in view of $\Psi_\#F_\#\mu=\mu$ implies that $\mu$-a.e.~the limit of $T^\eps h$ is $(Th)\circ\Psi^{-1}$. That is, the limit of the modified transfer operators $T^{\eps}$ coincides, at least formally, with the unregularized transfer operator $T$, composed with the optimal transport map from $\mu$ to $F_\#\mu$. In particular, if $\mu$ is invariant under $F$, then the optimal map $\Psi$ is the identity, and $T^\eps$ converges to $T$. In that case, for small positive $\eps>0$, the operator $T^{\eps}$ should be thought of as a regularized version of $T$, where $\big(T^{\eps}h\big)(x)$ is a weighted average of the values $h(y)$ for $y$ close to $F^{-1}(x)$.

\subsection{Approximate entropic transfer operator}
\label{sec:TNEps}
Let now $(\mu^N)_N$ be an approximating sequence of $\mu$, i.e.~$\mu^N \rightweaks \mu$ as $N \to \infty$. 
In our applications we consider $\mu^N$ that are concentrated on a finite number $N$ of points $x^N_1,\ldots,x^N_N\in\cX$,
\begin{equation}
\label{eq:MuDiscrete}
\mu^N := \sum_{k=1}^N m^N_k\delta_{x^N_k} \quad \text{with}\quad m^N_k \geq 0 \text{ and } \sum_{k=1}^Nm^N_k=1.
\end{equation}
The push forward $F_\# \mu^N$ of $\mu^N$ is then given by $\sum_{k=1}^N m^N_k \delta_{F(x^N_k)}$.
But the construction in this section works for general $\mu^N \in \prb(\cX)$.
The measure $\mu^N$ induces the discrete transfer operator $T^{N} : L^p(\mu^N) \to L^p(F_\# \mu^N)$ via \eqref{eq:Top}.
We can then perform the same construction as in the previous section for the measure $\mu^N$, introducing an operator $G^{N,\eps} : L^p(F_\# \mu^N) \to L^p(\mu^N)$ via the entropic optimal transport plan $\gamma^{N,\eps}$ from $F_\# \mu^N$ to $\mu^N$ and composing $G^{N,\eps}$ with $T^{N}$ to obtain the operator $T^{N,\eps}$ from $L^p(\mu^N)$ onto itself. This is part of the illustration in the right panel of Figure \ref{fig:Operators} .
Analogously to \eqref{eq:Tentropic}, \eqref{eq:tdens} we find that
\begin{align}
\label{eq:TNepsAction}
	\big(T^{N,\eps} h\big)(y) = \int_{\cX} h(x)\,t^{N,\eps}(x,y) \dd \mu^N(x)
\end{align}
$\mu^N$-almost everywhere, where $t^{N,\eps}$ is the density of the optimal $\eps$-entropic transport plan between $\mu^N$ and itself for the cost $c(F(\cdot),\cdot)$.
For the discrete case \eqref{eq:MuDiscrete}, this becomes
\begin{align}
	\label{eq:AlmostGamma}
	\big(T^{N,\eps} h\big)(x^N_k) = \sum_{j=1}^N h(x^N_j)\,t^{N,\eps}(x^N_j,x^N_k)\,m^N_j.
\end{align}
In this case, $t^{N,\eps}$ can be identified with a bistochastic $N \times N$ matrix $\big(t^{N,\eps}(x^N_j,x^N_k)\big)_{j,k=1}^N$ where bistochastic has to be understood relative to the weights $(m^N_k)$, i.e.,
\begin{align*}
    \sum_{j=1}^N m^N_j\, t^{N,\eps}(x^N_j,x^N_k) =1, 
    \quad
    \sum_{k=1}^N m^N_k\,t^{N,\eps}(x^N_j,x^N_k) =1. 
\end{align*}
This matrix can be approximated numerically efficiently by the Sinkhorn algorithm, see Section \ref{sec:Sinkhorn}.

\subsection{Approximate entropic Koopman operator}

The same discretization strategy applied to the Koopman operator yields
\begin{align}
	\big(K^{N,\eps} f\big)(x) = \int_{\cX} f(y)\,t^{N,\eps}(x,y) \dd \mu^N(y)
\end{align}
or, more explicitely,
\begin{align}
	\big(K^{N,\eps} f\big)(x^N_k) = \sum_{j=1}^N f(x^N_j)\,t^{N,\eps}(x^N_k,x^N_j)\,m^N_j,
\end{align}
which is the transpose of $\big(t^{N,\eps}(x^N_j,x^N_k)\big)_{j,k=1}^N$ as a representation of the discretized regularized Koopman operator.

\subsection{\texorpdfstring{Convergence as $N \to \infty$}{Convergence as N to Infty}}

\label{sec:Convergence}
\begin{figure}[ht]
\centering
\includegraphics[]{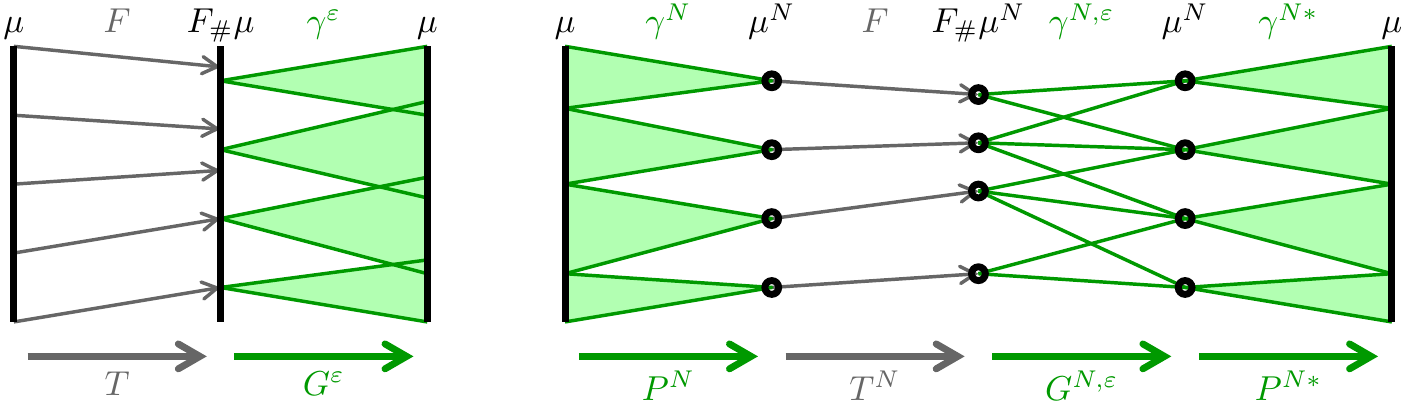}
\caption{Schematic of the construction of the maps $T^\eps$ (left) and $\hat{T}^{N,\eps}$ (right). In this figure $\gamma^\eps$ is a non-deterministic (entropic) continuous to continuous transport plan, $\gamma^N$ is a deterministic unregularized continuous to discrete transport plan and $\gamma^{N,\eps}$ is a non-deterministic discrete to discrete transport plan.}
\label{fig:Operators}
\end{figure}
For a better comparison of the transfer operators $T^{\eps}$ and $T^{N,\eps}$ related to a measure $\mu$ and its approximation $\mu^N$, we will now extend $T^{N,\eps}$ to the space $L^p(\mu)$.
Let $\gamma^N$ be the optimal (non-entropic) transport plan from  $\mu$ and $\mu^N$. As discussed in Section \ref{sec:OTTransfer}, $\gamma^N$ induces an operator $P^N: L^p(\mu) \to L^p(\mu^N)$. We now set
\begin{align*}
	\hat{T}^{N,\eps} := P^{N\ast}\,T^{N,\eps}\,P^{N},
\end{align*}
where $P^{N\ast} : L^p(\mu^N) \to L^p(\mu)$ is the adjoint of $P^N$ in the case $p=2$, which is equal to the transfer operator induced by the `transpose' $\gamma^{N*}:=(\pi^2,\pi^1)_\# \gamma^N$ of the transport plan $\gamma^N$.
The structure of the operator $\hat{T}^{N,\eps}$ is illustrated in the right panel of Figure \ref{fig:Operators}.
One finds with \eqref{eq:GCouplingPairing} and \eqref{eq:TNepsAction}
\begin{align*}
	\int_{\cX} (\hat{T}^{N,\eps}h\big)(y)\,\varphi(y) \dd\mu(y)
	& = \int_{\cX^2} \varphi(y)\,\big(T^{N,\eps}\,P^N h\big)(w)\,\dd \gamma^{N*}(w,y) \\
	& = \int_{\cX^3} \varphi(y)\,t^{N,\eps}(v,w)\, \big(P^N h\big)(v)\,\dd \gamma^{N*}(w,y) \dd \mu^N(v) \\
	& = \int_{\cX^4} \varphi(y)\,t^{N,\eps}(v,w)\,h(x) \dd \gamma^N(x,v) \dd \gamma^N(y,w)
\end{align*}
for test functions $\varphi \in C(\cX)$. Equivalently we can write
\begin{align}
\label{eq:THatExplicit}
	(\hat{T}^{N,\eps}h\big)(y) & = \int_{\cX} h(x)\,\hat{t}^{N,\eps}(x,y) \dd \mu(x) \\
	\text{with} \quad \hat{t}^{N,\eps}(x,y) & := \int_{\cX \times \cX} t^{N,\eps}(v,w) \dd \gamma^N_x(v) \dd \gamma^N_y(w)
	\nonumber
\end{align}
where $(\gamma^N_x)_x$ is the disintegration of $\gamma^N$ with respect to its first marginal, i.e.
\begin{align*}
	\int_{\cX \times \cX} \varphi(x,y)\dd \gamma^N(x,y) = \int_{\cX} \left[ \int_{\cX} \varphi(x,y)\dd \gamma^N_x(y) \right] \dd \mu(x)
\end{align*}
for $\varphi \in C(X \times Y)$.
Note that since $\mu$-almost all $(\gamma^N_x)_x$ are probability measures and since $t^{N,\eps} \in L^\infty(\mu^N \otimes \mu^N)$, one finds that $\hat{t}^{N,\eps} \in L^p(\mu \otimes \mu)$ for any $p \in [1,\infty]$.

\begin{proposition}
\label{prop:Conv}
If $\mu^N \rightweaks \mu$ as $N \to \infty$ then $\|\hat{t}^{N,\eps}-t^{\eps}\|_{L^2(\mu \otimes \mu)} \to 0$ and $\hat{T}^{N,\eps} \to T^{\eps}$ in the $L^2(\mu)\to L^2(\mu)$ operator norm.
\end{proposition}

\begin{proof}
In order to estimate the difference of $\hat{T}^{N,\eps}$ and $T^{\eps}$ in the operator norm, let $h\in L^2(\mu)$ be given and estimate as follows:
\begin{align*}
    \|\hat{T}^{N,\eps}h-T^{\eps}h\|_{L^2(\mu)}^2
    &= \int_{\cX} \left|\int_{\cX}\big(\hat{t}^{N,\eps}(y,x)-t^{\eps}(y,x)\big)h(y)\dd\mu(y)\right|^2\dd\mu(x) \\
    &\le \int_{\cX}\left[\int_{\cX}\big(\hat{t}^{N,\eps}(y,x)-t^{\eps}(y,x)\big)^2\dd\mu(y)\int_{\cX}h(y)^2\dd\mu(y)\right]\dd\mu(x) \\
    &= \|\hat{t}^{N,\eps}-t^{\eps}\|_{L^2(\mu \otimes \mu)}^2\|h\|_{L^2(\mu)}^2.
\end{align*}
That is, we have
\begin{align*}
    \|T^{N,\eps}-T^{\eps}\|_{L^2(\mu)\to L^2(\mu)}
    &\le \|\hat{t}^{N,\eps}-t^{\eps}\|_{L^2(\mu \otimes \mu)}.
\end{align*}
Therefore we now proceed to prove that the latter tends to zero.
For each $N$, let $\alpha^N$ and $\beta^N$ be the optimal scaling factors, as specified in \eqref{eq:d-entropicsolution}, for the $\eps$-entropic optimal transport between $\mu^N$ and itself for the cost function $\hat{c} := c(F(\cdot),\cdot)$ and $\theta=\mu^N \otimes \mu^N$. By the optimality condition for $\alpha^N$ \eqref{eq:d-constraint} one has $\mu^N(x)$-almost everywhere
\begin{align}
	\label{eq:AlphaCondition}
	\alpha^N(x) = -\eps \cdot \log\left(\int_{\cX} \exp\left(\frac{-\hat{c}(x,y)+\beta^N(y)}{\eps}\right) \dd \mu^N(y) \right),
\end{align}
cf.~\eqref{eq:EntropicOptimality}. Note that the right-hand side can be evaluated for any $x \in \cX$, not just $\mu^N$-almost everywhere, and thus we can extend $\alpha^N$ to a function $\cX \to \R$.

By compactness of $\cX$ and continuity of $F$, $\hat{c}$ is uniformly continuous on $\cX \times \cX$.
This means there exists a modulus of continuity $\omega:\R_+ \to \R_+$ with $\lim_{w\downarrow0}\omega(w)=0$, such that
\begin{align*}
|\hat{c}(x,y)-\hat{c}(x',y')| \leq \omega(\sqrt{\dst(x,x')^2 + \dst(y,y')^2}).
\end{align*}
One quickly verifies that $\alpha^N$ inherits the modulus of continuity from $\hat{c}$ by estimates of the form
\begin{align*}
	\alpha^N(x') 
	& \leq -\eps \cdot \log\int_{\cX}\exp\left(\frac{-\hat{c}(x,y)-\omega(\dst(x,x'))+\beta^N(y)}{\eps}\right) \dd \mu^N(y) \\
	& = \alpha^N(x) + \omega(\dst(x,x')).
\end{align*}
In the same way we extend the scaling factor $\beta^N$ to $\cX$, also inheriting the modulus of continuity from $\hat{c}$.

Therefore, by fixing the additive shift invariance, e.g.~by fixing $\alpha^N(x_0)=0$ for some fixed $x_0 \in \cX$, the sequences $(\alpha^N)$ and $(\beta^N)$ are uniformly bounded and equicontinuous. By the Arzelà–Ascoli theorem, there exists a uniformly convergent subsequence with limits $\alpha$ and $\beta$. A posteriori, we have uniform convergences $\alpha^N\to\alpha$ and $\beta^N\to\beta$, not just for a subsequence, but for the entire sequence. The reason is that for any chosen subsequences of $(\alpha^N)$ and $(\beta^N)$, the Arzelà–Ascoli theorem still applies and provides uniformly convergent subsubsequences with the same limits $\alpha$ and $\beta$: indeed, by going to the weak-*-limit $\mu$ with $\mu^N$ and to the uniform limits $\alpha$ and $\beta$ with $\alpha^N$ and $\beta^N$ in \eqref{eq:AlphaCondition} and its counterpart for the second marginal, we find that $\alpha$ and $\beta$ are optimal scaling factors for the limit problem between $\mu$ and itself. Therefore, $\alpha$ and $\beta$ are unique, up to global additive constants, and the latter are eliminated by the condition $\alpha(x_0) := \lim_{N \to \infty} \alpha^N(x_0)=0$. 

The extensions of $\alpha^N$ and $\beta^N$ to functions on $\cX$ carries over to 
\[
t^{N,\eps}(x,y)=\exp\left(\frac{-\hat{c}(x,y)+\alpha^{N}(x)+\beta^{N}(y)}{\eps}\right),
\]
i.e.\ we can evaluate $t^{N,\eps}$ on $\cX\times \cX$.
Consequently, the sequence $(t^{N,\eps})_N$ converges uniformly to $t^{\eps}$.
We can now estimate
\begin{align*}
\|\hat{t}^{N,\eps}-t^{\eps}\|_{L^2(\mu \otimes \mu)} & \leq
\|\hat{t}^{N,\eps}-t^{N,\eps}\|_{L^2(\mu \otimes \mu)} + \|t^{N,\eps}-t^{\eps}\|_{L^2(\mu \otimes \mu)}.
\end{align*}
The second term tends to zero as $N \to \infty$ by the uniform convergence of $t^{N,\eps} \to t^{\eps}$ on $\cX \times \cX$.
For the first term we estimate, by using the representation \eqref{eq:THatExplicit} and Jensen's inequality, 
\begin{align}
    \nonumber
	\|\hat{t}^{N,\eps}-t^{N,\eps}\|_{L^2}^2
	& = \int_{\cX \times \cX}  \left| \int_{\cX \times \cX} [t^{N,\eps}(v,w)-t^{N,\eps}(x,y)]
	\dd \gamma^N_x(v) \dd \gamma^N_y(w) \right|^2 \dd \mu(x) \dd \mu(y) \\
	\label{eq:tdiff}
	& \leq \int_{\cX^4} \left| t^{N,\eps}(v,w)-t^{N,\eps}(x,y) \right|^2 \dd \gamma^N(x,v) \dd \gamma^N(y,w) .
\end{align}
Since $\mu^N \rightweaks \mu$ by hypothesis, the plans $\gamma^N$ converge weakly-* to the optimal plan $\gamma$ for the transport of $\mu$ to itself, which is concentrated on the diagonal of $\cX$, see Section \ref{sec:OTKantorovich}. On the other hand, uniform convergence of $t^{N,\eps}$ to $t^\eps$ on $\cX\times\cX$ implies \begin{align*}
    \left| t^{N,\eps}(v,w)-t^{N,\eps}(x,y) \right|^2 \to \left| t^{\eps}(v,w)-t^{\eps}(x,y) \right|^2
\end{align*}
uniformly with respect to $(x,y,v,w)\in\cX^4$. Since the limit expression vanishes for $v=x$ and $w=y$, the bound in \eqref{eq:tdiff} tends to zero, implying the convergence of $\hat t^{N,\eps}-t^\eps$ to zero in $L^2(\mu\otimes\mu)$.
\end{proof}

\subsection{Compactness and convergence of spectrum}
\label{sec:ConvergenceSpectra}

\begin{proposition}
For fixed $\eps > 0$, the operator $T^\eps:L^2(\mu)\to L^2(\mu)$ is compact.
\end{proposition}

\begin{proof}
As shown in \eqref{eq:Tentropic}, $T^\eps$ is given by
\begin{align*}
	(T^{\eps} h)(y) & = \int_{\cX} g^{\eps}(F(x),y) \; h(x) \dd \mu(x).
\end{align*}
Note  that $g^\eps$ is bounded and thus square integrable on the compact set $\cX\times\cX$.  Therefore, $T^\eps$ is compact \cite[X.2]{Yosida:80}. 
\end{proof}

Compactness of $T^\eps$ together with convergence of $\hat T^{N,\eps}$ to $T^\eps$ in the $L^2$-operator norm yields that the eigenpairs of $\hat T^{N,\eps}$ converge to eigenpairs of $T^\eps$ as $N\to\infty$:

\begin{lemma}\cite[Lemma 2.2]{bramble1973rate}    
Let $\lambda^\eps$ be a nonzero eigenvalue of $T^\eps$ with algebraic multiplicity $m$ and let $\Gamma$ be a circle centered at $\lambda^\eps$ which lies in the resolvent set of $T^\eps$ and contains no other points of the spectrum of $T^\eps$. Then, there is an $N_0>0$ such that for all $N>N_0$, there are exactly $m$ eigenvalues (counting algebraic multiplicities) of $\hat T^{N,\eps}$ inside $\Gamma$ and no eigenvalue on $\Gamma$. 
\end{lemma}

\begin{theorem}\cite[Theorem 5]{Osborn75}
Let $\hat\lambda^{N,\eps}$ be an eigenvalue of $\hat T^{N,\eps}$ such that $\hat\lambda^{N,\eps}\to\lambda^\eps$ for $N\to\infty$.  For each $N$, let $\hat\varphi^{N,\eps}$ be an eigenvector of $\hat T^{N,\eps}$ at $\hat\lambda^{N,\eps}$.  Then there is a generalized eigenvector $\varphi^\eps_N$ of $T^\eps$ at $\lambda^\eps$ and a constant $C>0$ such that
\[
\|\hat\varphi^{N,\eps}-\varphi^\eps_N\|_{L^2} \leq C \|\hat T^{N,\eps}-T^\eps\|_{L^2}.
\]
\end{theorem}

Note that the construction of $T^{N,\eps}$ and $\hat{T}^{N,\eps}$ can also be performed in the  case $\eps=0$.  In that case, however, $T^0=T$ is in general not compact and the convergence statement is not applicable. In special settings, $T=T^0$  can be shown to be quasi compact \cite{LasotaYorke:74} and there are convergence results, e.g. for Ulam's method for expanding intervals maps~\cite{li1976finite}.

\subsection{\texorpdfstring{Relation between the spectra of $T^{N,\eps}$ and $\hat{T}^{N,\eps}$}{Relation between the spectra of TNeps and hatTNeps}}

\label{sec:SpectrumRelation}

Recall the optimal transport plan $\gamma^N$ from $\mu$ to $\mu^N$ and the induced operator $P^N : L^p(\mu) \to L^p(\mu^N)$ as introduced in Section \ref{sec:Convergence}. Consider the case where $\gamma^N_x = \delta_{\Phi(x)}$ for $\mu$-a.e.~$x$, that is, for almost all $x$, all the mass of $\mu$ at $x$ is transported to $\Phi(x)$ for some measurable $\Phi : \cX \to \cX$. This map $\Phi$ then solves the Monge problem \eqref{eq:monge}. It exists, for instance, when $\cX \subset \R^d$ and $\mu$ is dominated by the Lebesgue measure.
In this case $(P^{N\ast}h)(x)=h(\Phi(x))$ $\mu$-a.e., i.e.~the function $P^{N\ast} h$ is piecewise constant on preimages of points under $\Phi$ (these preimages may be individual points, in which case piecewise constancy is irrelevant). Conversely, for a function $h \in L^2(\mu)$ that is piecewise constant on preimages of single points under $\Phi$, we find $(P^Nh)(x)=h(\Phi^{-1}(x))$. The operator $P^N$ is zero on the orthogonal complement of the space of such functions.
Therefore, for each eigenpair $(h,\lambda) \in L^2(\mu^N) \otimes \C$ of $T^{N,\eps}$, $(h \circ \Phi,\lambda)$ is an eigenpair of $\hat{T}^{N,\eps}$ and all other eigenvalues of $\hat{T}^{N,\eps}$ are zero.
In this case, when $\mu^N$ is supported on a finite number of points, the spectrum of $\hat{T}^{N,\eps}$ (which converges to that of $T^\eps$), can be studied via the discrete matrix representation of $T^{N,\eps}$.

It might happen however that $\gamma^N$ is not induced by a map $\Phi$, in which case the direct correspondence between the spectra of $T^{N,\eps}$ and $\hat{T}^{N,\eps}$ fails. When $\cX \subset \R^d$ and $\mu$ is atomless this can be remedied as follows: By \cite[Theorem 1.32]{Santambrogio-OTAM} for each $N$ there exists a transport plan $\tilde{\gamma}^N$ that is induced by a map $\Phi^N$ such that $\int_{\cX^2} c(x,y) \dd \tilde{\gamma}^N(x,y)=\int_{\cX} c(x,\Phi^N(x))\dd \mu(x) \to 0$ as $N \to \infty$. Let $\tilde{P}^N$ be the corresponding transfer operator, which can be used to define an alternative extension $\tilde{T}^{N,\eps}$ of $T^{N,\eps}$. As argued above, these two operators will have the same non-zero eigenvalues and the corresponding eigenvectors are connected via $\Phi^N$. By the same arguments as in Proposition \ref{prop:Conv} one will find $\tilde{T}^{N,\eps} \to T^{\eps}$ in operator norm. Hence, the spectrum of $T^{N,\eps}$ still converges to that of $T^{\eps}$.

\section{Proof of concept --- analysis of the shift map on the torus}
\label{sec:shift map}
Below, we use the following short-hand notation for vectors $v=(v_1,\ldots,v_d)\in\C^d$:
\[ [v]^2= v^Tv = \sum_{\alpha=1}^d v_\alpha^2. \]
If $v\in\R^d\subset\C^d$, then $[v]^2=|v|^2$, but in general, $[v]^2$ is a complex number.

\subsection{\texorpdfstring{The shift map on the $d$-torus}{The shift map on the d-torus}}
On the $d$-dimensional torus $\cX=\R^d/\Z^d$, we consider the shift map $F(x)=x+\theta$ with a fixed vector $\theta\in(-1/2,+1/2)^d$. The uniform measure $\mu$ on $\cX$ is clearly invariant under $F$. 
\begin{remark}
    If $\theta$ happens to be \emph{irrational}, then $\mu$ is the \emph{unique} invariant $F$-measure, and $F$ is quasi-periodic on $\cX$. For the discussion below, these properties are not important.
\end{remark}
The transfer operator is given by
\[(T h)(x) = h(x-\theta).\]
A complete basis of eigenfunctions is formed by $\varphi_k(x)=e^{2\pi ik\cdot x}$ for $k\in\Z^d$, with respective eigenvalues
\begin{align}
    \label{eq:origlambda}
    \lambda_k=e^{-2\pi ik\cdot\theta}. 
\end{align}
Indeed,
\begin{align*}
    \big(T \varphi_k\big)(x)
    = \varphi_k(x-\theta) 
    = e^{2\pi ik\cdot(x-\theta)} 
    = e^{-2\pi ik\cdot\theta}\varphi_k(x).
\end{align*}
For the distance of points $x,x'\in\cX$, we use
\[\dst(x,x') := \min_{k\in\Z^d}|\hat x'-\hat x+k|,\]
where $\hat x,\hat x'\in\R^d$ are any representatives of $x,x'$, and $|\cdot|$ is the Euclidean distance on $\R^d$.
Equivalently, define the fundamental domain \[\cub:=(-1/2,+1/2]^d. \]
Choose representatives $\hat x,\hat x'\in\R^d$ of $x,x'\in\cX$ such that $\hat x'\in\hat x+\cub$. Then $\dst(x,x') = |\hat x-\hat x'|$.

\subsection{Spectrum of the regularized transfer operator}
In the example under consideration, the kernel $t^{\eps}$ can be calculated almost explicitly, due to the homogeneity and periodicity of the base space. It follows that the scaling factors $\alpha$ and $\beta$ are constant in space\footnote{Because $\alpha$ and $\beta$ are (up to constant shifts) the unique maximizers of a dual problem of \eqref{eq:entropic}, and the latter is invariant under translations, $\alpha$ and $\beta$ must be constant.}, and thus by \eqref{eq:gFdens} and \eqref{eq:tdens}:
\begin{align}
	\label{eq:tTorus}
    t^{\eps}(x,y) = g^{\eps}(F(x),y)) = \frac1{Z_\eps}e^{-\dst^2(x+\theta,y)/\eps},
    \quad
    Z_\eps := \eps^{d/2}\int_{\eps^{-1/2}\cub}e^{-|\zeta|^2}\dd\zeta.
\end{align}
Provided $y\in x+\theta+\cub$, this can be written more explicitly:
\begin{align*}
    t^{\eps}(x,y) = \frac1{Z_\eps}e^{-|x+\theta-y|^2/\eps}.
\end{align*}
\begin{proposition}\label{prop:eigenvalue asymptotics}
    For any $\eps>0$, a complete system of eigenfunctions of $T^{\eps}$ is given by the $\varphi_k$. The respective eigenvalues $\lambda^\eps_k$ satisfy
    \begin{align}
        \label{eq:ev1}
        \big|\lambda^\eps_k - e^{-\pi^2\eps|k|^2}\lambda_k\big| \le 2^{d+1}e^{-\frac1{8\eps}},
    \end{align}
    uniformly for $0<\eps<1/(8(d+2)\ln 2)$. Above, the $\lambda_k$ are given by \eqref{eq:origlambda}.
\end{proposition}
This result shows the interplay of two length scales. Based on \eqref{eq:tTorus} we see that $T^\eps$ acts locally like a translation by $\theta$ (given by $F$ or $T$) and a subsequent convolution with a Gaussian kernel of width $\sqrt{\eps}$, which is the length scale of the blur introduced by entropic optimal transport. On the other side, $1/|k|$ is the periodicity length scale of the eigenfunction $\varphi_k$. Proposition \ref{prop:eigenvalue asymptotics} implies that for small $\eps>0$, the corresponding eigenvalues of $T^\eps$ and of $T$ are similar for eigenfunctions with length scale above the blur scale, i.e.~$\sqrt{\eps}\ll 1/|k|$.
\begin{proof}
    Clearly, each $\varphi_k$ is an eigenfunction since
    \begin{align*}
        \big(T^{\eps}\varphi_k\big)(y)
        & = \int_{\cX}e^{2\pi ik\cdot x}\,t^{\eps}(x,y)\dd x \\
        &= \frac1{Z_\eps}\int_{y-\theta+\cub}e^{-|x-y+\theta|^2/\eps+2\pi ik\cdot x}\dd x \\
        &= \left(\frac1{Z_\eps}\int_{y-\theta+\cub}\exp
        \left(-\frac1{\eps}|x-y+\theta|^2+2\pi ik\cdot(x-y+\theta)\right)
        \dd x\right)e^{2\pi ik\cdot(y-\theta)} \\
        &= \left(\frac1{Z_\eps}\int_\cub\exp\left(-\left[\frac z{\sqrt{\eps}}-\sqrt{\eps}\pi ik\right]^2\right) e^{-\eps\pi^2|k|^2}\dd z\right)\lambda_k\varphi_k(y) \\
        &= \left(\frac{\sqrt{\eps^d}}{Z_\eps}\int_{\eps^{-1/2}\cub} e^{-[\zeta-i\sqrt\eps\pi k]^2}\dd\zeta\right)e^{-\eps\pi^2|k|^2}\lambda_k\varphi_k(y).
    \end{align*}
    To prove the claimed asymptotics, we start by observing that for each $\zeta,\tau\in\R^d$ with $\zeta\in\R^d\setminus\eps^{-1/2}\cub$:
    \begin{align}
    	\label{eq:ZetaTau}
        \big|e^{-[\zeta-i\tau]^2}\big|
        = e^{|\tau|^2-|\zeta|^2}
        \le e^{|\tau|^2}e^{-\frac1{8\eps}}e^{-\frac12|\zeta|^2},
    \end{align}
    since $|\zeta|^2\ge\frac1{4\eps}$.
    Using that, independently of $\tau$,
    \begin{align}
        \label{eq:complexmagic}
        \int_{\R^d}e^{-[\zeta-i\tau]^2}\dd\zeta = \pi^{d/2},
    \end{align}
    we thus obtain the following bound:
    \begin{equation}\label{eq:iii}
        \begin{split}
            \left|\pi^{d/2}-\int_{\eps^{-1/2}\cub} e^{-[\zeta-i\tau]^2}\dd \zeta\right|
            &=\left|\int_{\R^d\setminus\eps^{-1/2}\cub} e^{-[\zeta-i\tau]^2}\dd \zeta\right| \\
            &\le e^{|\tau|^2}e^{-\tfrac{1}{8\eps}}\int_{\R^d} e^{-\frac12[\zeta]^2}\dd\zeta
            = (2\pi)^{d/2}e^{|\tau|^2}e^{-\frac1{8\eps}}.
        \end{split}
    \end{equation}
    Apply this with $\tau=0$ to obtain
    \begin{align*}
        Z_\eps \ge (\pi\eps)^{d/2}-(2\pi\eps)^{d/2}e^{-\frac1{8\eps}} = (\pi\eps)^{d/2}\big(1-2^{d/2}e^{-\frac1{8\eps}}\big),
    \end{align*}
    and in particular we have $Z_\eps\ge(\pi\eps/2)^{d/2}>0$ for all $\eps>0$ in the range specified in the claim.
    Apply \eqref{eq:iii} again, now with $\tau=\sqrt\eps\pi k$, to obtain
    \begin{align*}
        \left|(\pi\eps)^{d/2}-{\eps}^{d/2}\int_{\eps^{-1/2}\cub} e^{-[\zeta-i\sqrt\eps\pi k]^2}\dd\zeta\right|
        &\le (2\pi\eps)^{d/2}e^{\eps\pi^2|k|^2}e^{-\frac1{8\eps}}.
    \end{align*}
    Divide this inequality by $Z_\eps$. Using that $Z_\eps^{-1}\le(\pi\eps/2)^{-d/2}$, and using \eqref{eq:iii}, it follows that for $\tau=0$,
    \begin{align*}
        \left|1-\frac{(\pi\eps)^{d/2}}{Z_\eps}\right| \le \frac{(2\pi\eps)^{d/2}e^{-1/(8\eps)}}{Z_\eps} \le 2^{d} e^{-\frac1{8\eps}},
    \end{align*}
    on the other hand, and finally that $e^{\pi^2\eps|k|}\ge1$, we conclude that
    \begin{align*}
        \left|1-\frac{\eps^{d/2}}{Z_\eps}\int_{\eps^{-1/2}\cub} e^{-[\zeta-i\sqrt\eps\pi k]^2}\dd\zeta\right|
        \le 2^{d+1}e^{\eps\pi^2|k|^2}e^{-\frac1{8\eps}}.
    \end{align*}
    Since $|\lambda_k|=1$, this yields the claim \eqref{eq:ev1}.
\end{proof}

\subsection{Spectrum of the discretized regularized transfer operator}
\label{subsec:spectrum_discretized}
For the discretization, we consider a regular lattice with $N=n^d$ points $x^N_j$ in the fundamental domain $\cub$. Specifically, let $J_N\subset\Z^d$ be the set of indices $j=(j_1,\ldots,j_d)$ such that
\[x^N_j=(j_1/n,\ldots,j_d/n)\in\cub.\]
All points $x^N_j$ are given equal weight $m^N_j=1/N$, thus 
\[\mu^N = \frac1N\sum_{j\in\Z_n^d}\delta_{x^N_j}.\]
Based on \eqref{eq:AlmostGamma} we introduce the matrix $\Gamma^{N,\eps}$ via
\begin{align*}
\Gamma^{N,\eps}_{\ell,m} := t^{N,\eps}(x^N_m,x^N_\ell)\,m^N_m,
\end{align*}
such that the matrix-vector product with $\Gamma^{N,\eps}$ corresponds to the application of $T^{N,\eps}$, i.e.,
\[
\big(\Gamma^{N,\eps} v\big)_\ell =  \big(T^{N,\eps} h\big)(x_\ell^N),
\]
where $v\in\R^{\Z_n^d}$ with $v_k=h(x^N_k)$.
Spectral analysis of $T^{N,\eps}$ then is equivalent to spectral analysis of $\Gamma^{N,\eps}$.
As before, the high symmetry of the problem implies that the weight factors $\alpha$ and $\beta$ in the representation \eqref{eq:gFdens} are actually global constants, and thus
\begin{align*}
    \Gamma^{N,\eps}_{\ell,m} = \frac1{N\,Z^{N,\eps}}\exp\left(-\frac1\eps\dst(x^N_m+\theta,x^N_\ell)^2\right),
\end{align*}
with a normalization constant 
\begin{align}
    \label{eq:Zxmp}
    Z^{N,\eps} = 
    \frac1N\sum_{m\in J_N} \exp\left(-\frac1\eps\dst(x^N_m+\theta,0)^2\right).
\end{align}
In order to see that $Z^{N,\eps}$ has this simple form, recall that the distance function $\dst$ is $d$-fold periodic in both entries, and that the points $(x^N_\ell)_{\ell\in J_N}$ are part of a $d$-fold periodic lattice. Therefore, the sum with respect to $m\in J_N$ of the exponential values depending on $\dst(x^N-m+\theta,x^N_\ell)$ is actually independent of the chosen index $\ell$, and in particular, one may choose $\ell=0$.
\begin{proposition}\label{prop:eigenvalue asymptotics2}
    A complete system of eigenvectors for $\Gamma^{N,\eps}$ is given by $v^N_k$ for $k\in J_N$, where
    \[ \big(v^N_k\big)_j = e^{2\pi ik\cdot j/n}.\]
    The associated eigenvalues $\lambda^{N,\eps}_k$ satisfy
    \begin{align}
        \label{eq:ev2}
        \big|\lambda^{N,\eps}_k-e^{-\pi^2\eps|k|^2}\lambda_k\big|
        &\le \frac{K_d}{n^2\eps}\big(1+\eps|k|^2+n^2\eps e^{-1/(8\eps)}\big) , \\
        \label{eq:ev3}
        \big|\lambda^{N,\eps}_k|&\le \frac{K_d}{\eps|k|^2}\big(1+n^2\eps\,e^{-1/(32\eps)}\big),
    \end{align}
    provided that $\eps\le\bar\eps_d$ and $(n^2\eps)^{-1}\le\bar h_d$, with positive constants $K_d$, $\bar\eps_d$ and $\bar h_d$ that depend only on the dimension $d$. Above, the $\lambda_k$ are given by \eqref{eq:origlambda}.
\end{proposition}
Compared to Proposition \ref{prop:eigenvalue asymptotics} we see here the effect of a third length scale $1/n$, associated with the discretization. Proposition \ref{prop:eigenvalue asymptotics2} implies that for small $\eps$, eigenvalues of $T^{N,\eps}$ are similar to those of $T$, as long as the entropic blur is larger than the discretization scale, but smaller than the scale of the eigenfunction, i.e. 
\begin{align}
\label{eq:smallness_assumption}
	\frac{1}{n} \ll \sqrt{\eps} \ll \frac{1}{|k|}.
\end{align}
For the torus we will refine this discussion in Section \ref{sec:TorusSpectrumDiscussion}.
More generally, we expect that this observed relation between the three length scales also provides a good intuition for other dynamical systems $(\cX,F,\mu)$.
\begin{proof}
    To prove that $v^N_k$ for a given $k\in J_N$ is an eigenvector, and to calculate its eigenvalue $\lambda^{N,\eps}_k$, we need to evaluate the product $\Gamma^{N,\eps}v^N_k$ in every component $m\in J_N$:
    \begin{align}
    \label{eq:eigvalcalc}
        \big[\Gamma^{N,\eps}v^N_k\big]_m
        = \frac1{N\,Z^{N,\eps}}\sum_{\ell\in J_N}e^{-\frac1\eps\dst(\ell/n+\theta,m/n)^2}e^{2\pi ik\cdot\ell/n}.
    \end{align}
    By $n$-periodicity of both exponents with respect to each component of $\ell$, one can shift the domain of summation in $\Z^d$. In particular, we may replace the summation over $\ell\in J_N$ by 
    \begin{align*}
        \ell\in J_{N,\theta}+m\quad\text{where}\quad 
        J_{N,\theta} := \left\{ j\in\Z^d \,\middle|\,\frac{j}n+\theta\in\cub\right\}.
    \end{align*}
    Notice that $J_{N,\theta}$ and thus also $J_{N,\theta}+m$ are indeed translates of $J_N$. By definition of $\dst$, we can rewrite \eqref{eq:eigvalcalc} as
    \begin{align*}
        \big[\Gamma^{N,\eps}v^N_k\big]_m
        &= \frac1{N\,Z^{N,\eps}}\sum_{\ell\in J_{N,\theta}+m} e^{-\frac1\eps[(\ell-m)/n+\theta]^2}e^{2\pi ik\cdot\ell/n} \\
        &= \frac1{N\,Z^{N,\eps}}\left(\sum_{j\in J_{N,\theta}}e^{-\frac1\eps[j/n+\theta]^2}e^{2\pi ik\cdot j/n}\right)[v^N_k]_m,
    \end{align*}
    which proves the eigenvector property of $v^N_k$ with corresponding eigenvalue
    \begin{align}
        \label{eq:lambdarep}
        \lambda^{N,\eps}_k = \frac1{N\,Z^{N,\eps}}\sum_{j\in J_{N,\theta}}e^{-\frac1\eps[j/n+\theta]^2}e^{2\pi ik\cdot j/n}.
    \end{align}
    We start by proving \eqref{eq:ev2}. Directly from \eqref{eq:lambdarep}, recalling that $\lambda_k=e^{-2\pi i k\cdot\theta}$ by \eqref{eq:origlambda}, we obtain
    \begin{align*}
        \lambda^{N,\eps}_k 
        &= \frac{\lambda_k}{N\,Z^{N,\eps}} \sum_{j\in J_{N,\theta}} \exp\left(-\frac1\eps\left[\frac jn+\theta\right]^2+2\pi i k\cdot\left(\frac jn+\theta\right)\right) \\
        &= \frac{\lambda_k}{N\,Z^{N,\eps}} \sum_{j\in J_{N,\theta}} \exp\left(-\left[\frac1{\eps^{1/2}}\left(\frac jn+\theta\right)-\eps^{1/2}\pi ik\right]^2\right) e^{-\eps\pi^2|k|^2}.
    \end{align*}
    For brevity, let 
    \begin{align}
        \label{eq:xigrid}
        \xi_j = \frac{j+n\theta}{n\eps^{1/2}}
    \end{align}
    in the following; then
    \begin{align}
        \label{eq:tech000}
        \lambda^{N,\eps}_k = \frac{\lambda_k e^{-\eps\pi^2|k|^2}}{N\,Z_{N,\eps}}\sum_{j\in J_{N,\theta}}e^{-[\xi_j-\eps^{1/2}\pi k i]^2},
        \quad
        Z^{N,\eps} = \frac1N\sum_{j\in J_{N,\theta}}e^{-[\xi_j]^2}.
    \end{align}
    To further simplify these expressions, we estimate the errors induced by first replacing the sum over $j\in J_{N,\theta}$ above by a sum over all $j\in\Z^d$, and then by replacing that sum by an integral.
    
    For the passage from $J_{N,\theta}$ to $\Z^d$, we combine two elementary inequalities. First, analogous to \eqref{eq:ZetaTau} for $j\in\Z^d\setminus J_{N,\theta}$, and for any $\tau\in\R^d$,
    \begin{align}
        \label{eq:111}
        |e^{-[\xi_j-i\tau]^2}| = e^{|\tau|^2-|\xi_j|^2}
        \le e^{|\tau|^2}e^{-\frac1{8\eps}}e^{-\frac12|\xi_j|^2},
    \end{align}
    since $|\xi_j|^2\ge\frac1{4\eps}$.
    And second, for any $a,b\in\R$ with $0<b\le1$,
    \begin{align}
        \label{eq:seriesest}
        b\sum_{j\in\Z}e^{-\frac12(bj-a)^2} \le \Theta,
    \end{align}
    with a uniform constant $\Theta$; for an easy proof of \eqref{eq:seriesest}, observe that
    \begin{align*}
        (bx-a)^2\le 2(bj-a)^2+\frac12b^2 \quad\text{for all $x\in[j-1/2,j+1/2]$},
    \end{align*}
    and so
    \begin{align*}
        b\sum_{j\in\Z}e^{-\frac12(bj-a)^2} 
        \le b\int_{\R} e^{-\frac14(bx-a)^2+\frac18b^2} \dd x
        = e^{\frac18b^2}\int_{\R}e^{-\frac14 y^2}\dd x
        \le 2\sqrt\pi\,e^{\frac18}. 
    \end{align*}
    Now, Combining \eqref{eq:111} and \eqref{eq:seriesest} with $b:=(n^2\eps)^{-1/2}$ --- note that $b\le1$ for $n^2\eps\ge 1$ --- we obtain that
    \begin{align}
        \nonumber
        \frac1{N}\left|\sum_{j\in\Z^d} e^{-[\xi_j-i\tau]^2} - \sum_{j\in J_{N,\theta}} e^{-[\xi_j-i\tau]^2}\right|
        &\le \eps^{d/2}e^{|\tau|^2}e^{-\frac1{8\eps}}\prod_{\alpha=1}^d\left(\frac1{n\eps^{1/2}}\sum_{j_\alpha\in\Z}e^{-\frac12(\frac{j_\alpha}{n\eps^{1/2}}+\theta_\alpha)^2}\right) \\
        \label{eq:tech001}
        &\le \eps^{d/2}\Theta^de^{|\tau|^2}e^{-\frac1{8\eps}}.
    \end{align}
    For the passage from the sum to the integral, we use that \eqref{eq:xigrid} with $j\in\Z^d$ arbitrary defines a $d$-dimensional uniform lattice in $\R^d$ of mesh size $(n^2\eps)^{-1/2}$. We associate each $j\in\Z^d$ to the half-open cube
    \begin{align*}
        B_j:=\xi_j+(n^2\eps)^{-1/2}\cub.
    \end{align*}
    The $B_j$ form a tesselation of $\R^d$.         
    
    Consider a $C^2$-function $h:\R^d\to\C$, i.e., both real and imaginary part of $h$ are $C^2$. A second order Taylor expansion yields for arbitrary $\zeta\in B_j$ that
    \begin{align*}
        \big|h(\zeta)-\big[h(\xi_j)+\nabla h(\xi_j)\cdot(\zeta-\xi_j)\big]\big| \le \frac12\sup_{z\in B_j}\opn{\nabla^2h(z)}\,|\zeta-\xi_j|^2,
    \end{align*}
    where $\nabla h(\xi_j)\in\C^d$ is $h$'s complex valued gradient at $\xi_j$, and $\opn{\nabla^2h(z)}$ is the Euclidean operator norm of the complex valued Hessian at $z$, i.e.,
    \begin{align*}
        \opn{\nabla^2h(z)} = \sup\left\{ \big|v\cdot \nabla^2h(z)v\big|\,\middle|\,v\in\R^d,\,|v|\le1\right\}. 
    \end{align*}
    Therefore,
    \begin{align*}
        \left| \fint_{B_j} h(\zeta)\dd\zeta - h(\xi_j) \right| 
        &\le \left|\nabla h(\xi_j)\cdot \fint_{B_j}(\zeta-\xi_j)\dd\zeta \right|
        + \frac12\sup_{z\in B_j}\opn{\nabla^2h(z)}\fint_{B_j}|\zeta-\xi_j|^2\dd\zeta \\
        &= \frac d2\fint_{-\frac12(n^2\eps)^{-1/2}}^{\frac12(n^2\eps)^{-1/2}}w^2\dd w\ \sup_{z\in B_j}\opn{\nabla^2h(z)} \\
        &= \frac{d}{24n^2\eps}\sup_{z\in B_j}\opn{\nabla^2 h(z)}.
    \end{align*}
    Specifically for $h(z)=e^{-[z-i\tau]^2}$, we have
    \begin{align*}
        \nabla^2h(z) = h(z)\big((4zz^T-4\tau\tau^T-2\id)-4i(\tau z^T+z\tau^T)\big),
    \end{align*}
    and therefore, 
    \begin{align*}
        \opn{\nabla^2h(z)} &\le \big(\opn{4zz^T-4\tau\tau^T-2\id}+ 4\opn{\tau z^T+z\tau^T}\big)|h(z)| \\
        &\le \big(2+4|z|^2+4|\tau|^2+8|\tau||z|\big) \big|e^{-\sum_k(z_k^2-\tau_k^2-2iz_k\tau_k)}\big| \\
        &\le \big(2+8|z|^2+8|\tau|^2\big) e^{|\tau|^2-|z|^2} \\
        &\le 8(1+|\tau|^2)e^{|\tau|^2}(1+|z|^2)e^{-\frac12|z|^2} e^{-\frac12|z|^2} \\
        &\le 16(1+|\tau|^2)e^{|\tau|^2}e^{-\frac12|z|^2},
    \end{align*}
    where we have used the elementary inequality $(1+|z|^2)e^{-|z|/2}\le2$ to pass from the fourth line to the last. 
    Hence the difference between the integral $\int_{\R^d}h(\xi)\dn\xi$ and the approximation by the sum $(n^2\eps)^{-d/2}\sum_{j\in\Z^d}h(\xi_j)$ can be estimated as follows:
    \begin{align}
        \nonumber
        \left|\int_{\R^d}e^{-[\xi-i\tau]^2}\dd\zeta - (n^2\eps)^{-d/2}\sum_{j\in\Z^d}e^{-[\xi_j-i\tau]^2}\right|
        &\le \frac{16d(1+|\tau|^2)e^{|\tau|^2}}{24n^2\eps}\,(n^2\eps)^{-d/2}\sum_{j\in\Z^d}e^{-\frac12|\xi_j|^2} \\
        \label{eq:tech002}
        &\le\frac{2d\Theta^d}3 \cdot\frac{(1+|\tau|^2)e^{|\tau|^2}}{n^2\eps}.
    \end{align}
    Now recall that by \eqref{eq:complexmagic}, the value of the integral above is $\pi^{d/2}$, independently of $\tau$. We combine the estimates \eqref{eq:tech001} and \eqref{eq:tech002} and apply them to the representations of $\lambda^{N,\eps}_k$ and of $Z^{N,\eps}$ in \eqref{eq:tech000}, with $\tau=0$ and $\tau=\sqrt{\eps}\pi k$, respectively. This yields:
    \begin{align}
        \label{eq:techhelp002}
        \left|\pi^{d/2}-\eps^{-d/2}Z^{N,\eps}\right| &\le \Theta^d\left(e^{-\frac1{8\eps}}+\frac{2d}{3n^2\eps}\right), \\
        \label{eq:techhelp003}
        \left|\pi^{d/2}\lambda_k- e^{\eps\pi^2|k|^2}\eps^{-d/2}Z^{N,\eps}\lambda^{N,\eps}_k\right| &\le \Theta^d e^{\eps\pi^2|k|^2}\left(e^{-\frac1{8\eps}}+\frac{2d(1+\eps\pi^2|k|^2)}{3n^2\eps}\right).
    \end{align}
    Finally, we assume that $\eps>0$ is sufficiently small and $n\eps^{-1/2}$ is sufficiently large so that
    \begin{align*}
        \Theta^d\left(e^{-\frac1{8\eps}}+\frac{2d}{3n^2\eps}\right) \le \frac{\pi^{d/2}}2
    \end{align*}
    and consequently
    \begin{align}
        \label{eq:techhelp004}
        \eps^{-d/2}Z^{N,\eps} \ge \frac{\pi^{d/2}}2.
    \end{align}
    The claim \eqref{eq:ev2} is now proven by combining \eqref{eq:techhelp002}, \eqref{eq:techhelp003} and \eqref{eq:techhelp004}:
    \begin{align*}
        \big|\lambda^{N,\eps}_k - e^{-\pi^2\eps|k|^2}\lambda_k\big|
        &=\frac{\eps^{d/2}e^{-\pi^2\eps|k|^2}}{Z^{N,\eps}}
        \big(\big|e^{\pi^2\eps|k|^2}\eps^{-d/2}Z^{N,\eps}\lambda^{N,\eps}_k-\eps^{-d/2}Z^{N,\eps}\lambda_k\big|\big) \\
        &\le\frac{\eps^{d/2}e^{-\pi^2\eps|k|^2}}{Z_{N,\eps}}
        \big(\big|e^{\pi^2\eps|k|^2}\eps^{-d/2}Z_{N,\eps}\lambda^{N,\eps}_k-\lambda_k\pi^{d/2}\big| + \big|\pi^{d/2}-\eps^{-d/2}Z^{N,\eps}\big|\big) \\
        &\le\frac{2(1+e^{-\pi^2\eps|k|^2})\Theta^d}{\pi^{d/2}}\left(e^{-\frac1{8\eps}}+\frac{2d(1+\eps\pi^2|k|^2)}{3n^2\eps}\right).
    \end{align*}
    We turn to the proof of \eqref{eq:ev3}. Assume $k\in J_N$ is not zero. Without loss of generality, let $k_1$ be a component with largest modulus, i.e., $|k|^2\le d k_1^2$. We write $k=(k_1,k^*)$ and decompose $j\in J_{N,\theta}$ accordingly as $j=(j_1,j^*)\in J_{N,\theta}^1\times J_{N,\theta}^*$,
    as well as $\theta=(\theta_1,\theta^*)$. We then have, using that $[j/n+\theta]^2=(j_1/n+\theta_1)^2+[j^*/n+\theta^*]$ and $k\cdot j=k_1j_1+k^*\cdot j^*$,
    \begin{equation}
        \label{eq:tech004}
        \begin{split}
        \big|N\,Z^{N,\eps}\lambda^{N,\eps}_k\big|
        & = \left|\sum_{(j_1,j^*)\in J_{N,\theta}} e^{-\frac1\eps[j^*/n+\theta^*]^2} e^{2\pi i k^*\cdot j^*/n}e^{-\frac1\eps(j_1/n+\theta_1)^2}e^{2\pi i k_1j_1/n}\right| \\
        & \le \sum_{j^*\in J_{N,\theta}^*}e^{-\frac1\eps[j^*/n+\theta^*]^2}I,
        \end{split}
    \end{equation}
    where
    \begin{align*}
        \ I := \left|\sum_{j_1\in J_{N,\theta}^1}e^{-\frac1\eps(j_1/n+\theta_1)^2}e^{2\pi i k_1j_1/n}\right|.
    \end{align*}
    To estimate $I$ further, we perform a summation by parts: for any complex numbers $a_\ell,b_\ell$, one has that
    \begin{align}
    \label{eq:sumbyparts}
    \begin{split}
        \sum_{\ell=\ell_0}^{\ell_1} b_\ell(a_{\ell+1}-2a_\ell+a_{\ell-1})
        &= \sum_{\ell=\ell_0}^{\ell_1}(b_{\ell+1}-2b_\ell+b_{\ell-1})a_\ell \\
        &\quad + a_{\ell_1+1}b_{\ell_1}-a_{\ell_1}b_{\ell_1+1} + a_{\ell_0-1}b_{\ell_0}-a_{\ell_0}b_{\ell_0-1}.
    \end{split}
    \end{align}
    We substitute
    \begin{align*}
        e^{2\pi i k_1j_1/n}
        = -\frac{e^{2\pi i k_1(j_1+1)/n}-2e^{2\pi i k_1 j_1/n}+e^{2\pi i k_1(j_1-1)/n}}{2(1-\cos(2\pi k_1/n))},
    \end{align*}
    and use \eqref{eq:sumbyparts}:
    \begin{equation}
        \label{eq:tech005}
    \begin{split}
        &2(1-\cos(2\pi k_1/n))I\\
        &= \left|\sum_{j_1\in J_{N,\theta}^1}e^{-\frac1\eps(j_1/n+\theta_1)^2}\big(e^{2\pi i k_1(j_1+1)/n}-2e^{2\pi i k_1 j_1/n}+e^{2\pi i k_1(j_1-1)/n}\big)\right|\\
        &= \left|\sum_{j_1\in J_{N,\theta}^1}\big(e^{-\frac1\eps((j_1+1)/n+\theta_1)^2}-2e^{-\frac1\eps(j_1/n+\theta_1)^2}+e^{-\frac1\eps((j_1-1)/n+\theta_1)^2}\big)e^{2\pi i k_1 j_1/n}+B\right| \\
        &= \left|\frac1{n^2}\sum_{j_1\in J}\left(4\left(\frac{\nu_{j_1}+\theta_1}\eps\right)^2-\frac2{\eps}\right)e^{-\frac1\eps(\nu_{j_1}+\theta_1)^2} + B\right|\\
        &\le \frac1{n^2\eps}\sum_{j_1\in J_{N,\theta}^1}\left(4\left(\frac{\nu_{j_1}+\theta_1}{\sqrt{\eps}}\right)^2+2\right)\exp\left(-\left(\frac{\nu_{j_1}+\theta_1}{\sqrt{\eps}}\right)^2\right)+ |B|.
    \end{split}
    \end{equation}
    Here $\nu_{j_1}$ is a suitable intermediate value between $(j_1-1)/n$ and $(j_1+1)/n$, provided by Taylor's theorem,
    and $B$ summarizes the ``boundary terms''. According to \eqref{eq:sumbyparts}, $B$ is the sum of four products of the form 
    \[ \pm e^{-\frac1\eps(j_1^+/n+\theta_1)^2}e^{2\pi ik_1 j_1^-/n}, \]
    where $|j_1^+-j_1^-|=1$,
    and $|j_1^+/n+\theta_1|\ge1/2-1/n\ge1/4$. Thus,
    \[ |B|\le 4e^{-1/(16\eps)}. \]
    Next, decompose $J_{N,\theta}^1=J_+\cup J_-$, where $J_+$ and $J_-$ contain the indices $j_1$ with $j_1/n+\theta>0$ and with $j_1/n+\theta\le0$, respectively. That is, $J_+$ and $J_-$ are the upper and lower halves, respectively, of $J_{N,\theta}^1$. Fix $j_1\in J_+$ in the following. Assuming that $n\sqrt\eps>4$, and recalling that $|\nu_{j_1}-j_1/n|\le1/n$, it follows that
    \begin{align*}
        \left|\frac{\nu_{j_1}-j_1/n}{\sqrt\eps}\right| < \frac14 \quad \text{and so} \quad
        \eta_{j_1}:=\frac{\nu_{j_1}+\theta_1}{\sqrt\eps} > -\frac14.
    \end{align*}
    Let $n^*\in\Z$ be an index such that
    \begin{align}
        \label{eq:nrange}
        \frac34 n\sqrt{\eps} < n^* < n\sqrt\eps,
    \end{align}
    and define $\sigma_{j_1}:=(\nu_{j_1}-j_1/n+n^*/n)/\sqrt{\eps}\in[1/2,5/4]$. An elementary inequality, obtainable e.g. from $e^\eta\ge(1+\eta/2)^2$, provides on grounds of $\eta_{j_1}>-1/4$ and $1/2<\sigma_{j_1}<5/4$ that
    \begin{align*}
        (4\eta_{j_1}^2+2)e^{-\eta_{j_1}^2}\le 16e^{-(\eta_{j_1}-\sigma_{j_1})^2}.
    \end{align*}
    Substituting the expressions for $\eta_{j_1}$ and $\sigma_{j_1}$ yields
    \begin{align*}
        \left(4\left(\frac{\nu_{j_1}+\theta_1}{\sqrt{\eps}}\right)^2+2\right)\exp\left(-\left(\frac{\nu_{j_1}+\theta_1}{\sqrt{\eps}}\right)^2\right)
        \le 16 e^{-\frac1\eps\big((j_1-n^*)/n+\theta_1\big)^2}.
    \end{align*}
    Now since $n^*<n/2$ for $\eps<1/4$ by \eqref{eq:nrange}, we have that $j_1-n^*\in J_{N,\theta}^1$, and so, by a simple index shift,
    \begin{align*}
        \sum_{j_1\in J_+}\left(4\left(\frac{\nu_{j_1}+\theta_1}{\sqrt{\eps}}\right)^2+2\right)\exp\left(-\left(\frac{\nu_{j_1}+\theta_1}{\sqrt{\eps}}\right)^2\right)
        \le 16\sum_{j_1\in J_{N,\theta}^1}e^{-\frac1\eps(j_1/n+\theta_1)^2}.
    \end{align*}
    The reasoning for $J_-$ in place in $J_+$ in analogous. Recalling \eqref{eq:tech005}, we conclude that
    \begin{align*}
        I \le \frac 1{2n^2\eps(1-\cos(2\pi k_1/n))}\left[32\sum_{j_1\in J_{N,\theta}^1} e^{-\frac1\eps(j_1/n+\theta_1)^2} + 4n^2\eps e^{-1/(16\eps)}\right].
    \end{align*}
    The denominator is now estimated using that $\cos t\le 1-\frac18t^2$ for all $|t|\le\pi$. Further, the expression $e^{-1/(32\eps)}$ is clearly less or equal than the largest term in the sum if $n>4$. This leads to the following estimate:
    \begin{align*}
        I \le \frac4{\eps|k|^2}\big(32+4n^2\eps e^{-1/(32\eps)}\big) \sum_{j_1\in J_{N,\theta}^1} e^{-\frac1\eps(j_1/n+\theta_1)^2} .
    \end{align*}
    Eventually, substitution into \eqref{eq:tech004} provides
    \begin{align*}
        \big|N\,Z^{N,\eps}\lambda^{N,\eps}_k\big|
        &\le \frac{128+16n^2\eps e^{-1/(32\eps)}}{\eps|k_1|^2}\sum_{j^*\in J_{N,\theta}^*}e^{-\frac1\eps[j^*/n+\theta']^2}\sum_{j_1\in J_{N,\theta}^1}e^{-\frac1\eps(j_1/n+\theta_1)^2} \\
        &\le \frac{128 d N\,Z_N}{\eps|k|^2}\big(1+n^2\eps e^{-1/(32\eps)}\big),
    \end{align*}
    which in turn shows \eqref{eq:ev3}.
\end{proof}

\subsection{Consequence: accumulation of eigenvalues}
\label{sec:TorusSpectrumDiscussion}
In this section, we draw several conclusions from Propositions \ref{prop:eigenvalue asymptotics} and \ref{prop:eigenvalue asymptotics2} on the relation between the shift $\theta$, the length scales $\eps^{1/2}$ and $n^{-1}$, and the accumulation of eigenvalues $\lambda^{N,\eps}_k$ on or near the unit circle. The goal is to give a theoretically founded explanation for some of the results from numerical experiments that are presented in Section \ref{sec:ExamplesShift} further below and to provide some intuition for more general systems.

We start by discussing the smallness assumption $(n^2\eps)^{-1}\le\bar h_d$ of Proposition \ref{prop:eigenvalue asymptotics2}, i.e. $n^{-1}\ll \sqrt\eps$ on the left-hand side of~\eqref{eq:smallness_assumption}. This is far more than a technical hypothesis: we claim that, in general, the discretized dynamics does need a blur that stretches over several cells of the discretization in order to produce meaningful results. 

Let us consider the case $n^{-1} \geq \sqrt\eps$, i.e.~$n^2\eps \leq 1$.  Writing in the representation \eqref{eq:lambdarep} of $\lambda^{N,\eps}_k$ the first exponent as $[j+n\theta]^2/(n^2\eps)$, the entire sum reduces essentially to only one term, namely the one for the index $j=\bar m\in J_{N,\theta}$ closest to $-n\theta$. Consequently, $\lambda^{N,\eps}_k\approx e^{2\pi i k\cdot\bar m/n}$ for all $k$. That is, a large portion of the $N$ eigenvalues lies very close to the unit circle, in groups with uniform angular distance. The spectrum thus contains little information about the dynamics, and might change significantly with the discretization.

For illustration, consider the one-dimensional case $d=1$ with $\theta=1/11$: the spectrum of the pure transfer operator $T$ consists of eleven equally spaced groups of eigenvalues on the unit circle. For discretizations with $n=10$, $n=100$, $n=1000$ etc. points, the limiting spectrum always consists of $n$ groups of points (since the integer $\bar m$ nearest to $-n/11$ never has two or five as prime factor). Hence the coarsest discretization with $n=10$ produces the qualitatively best approximation and the additional structures that appear in the finer discretizations are only misleading.

Consider the case $n^{-1}\ll\sqrt\eps$ next. Under this hypothesis, the
estimate \eqref{eq:ev2} implies a good agreement between the genuine eigenvalues $\lambda^{N,\eps}_k$ and their $N$-independent approximations
\begin{align*}
    \Lambda^\eps_k:=e^{-\eps\pi^2|k|^2}e^{-2\pi i k\cdot\theta},
\end{align*}
at least for sufficiently small indices $k$, that is $|k|^{-1}\ll\sqrt\eps$. Actually, for larger $|k|$, the respective eigenvalues $\lambda^{N,\eps}_k$ lie close to the origin due to \eqref{eq:ev3}, and provide no significant information about the dynamics. Since we are interested in the spectrum close to the unit circle, it is sufficient to study the $\Lambda^\eps_k$ with small $k$ in the sense above.

For ease of presentation, let $d=1$, so that $\theta\in(-1/2,1/2)$ is a real number. Assume that $p,q\in\Z$ are coprime integers with $q>0$ such that
\begin{align}
    \label{eq:nmbr}
    \delta:=\theta-\frac pq\quad\text{satisfies}\quad c:=q^2|\delta|\in(0,1).
\end{align}
Notice that by Dirichlet's approximation theorem, there are infinitely many such pairs $(p,q)$ for each irrational $\theta$. The smaller the value of $c$, the more pronounced is the accumulation phenomenon described below.

Define the index set $L_q:=\{k\in\Z:-q<2k\le q\}$. The $q$ points $\Lambda^\eps_k$ for $k\in L_q$ have almost uniform angular distance of $2\pi/q$. Actually,
\begin{align*}
    \frac{\Lambda^\eps_k}{|\Lambda^\eps_k|} 
    = e^{-2\pi i k(p/q+\delta)}
    = e^{-2\pi i kp/q}e^{i\phi},
\end{align*}
with $|\phi|=|2\pi k\delta|\le \pi q\delta < \pi/q$ by \eqref{eq:nmbr}, so that the variation of the individual angular positions is less than half of the expected angular distance $2\pi/q$. An analogous argument applies also to the shifted index sets $L_q+mq$ with an arbitrary $m\in\Z$; the corresponding $\Lambda^\eps_k$ form a group of $q$ points with approximately uniform angular distance of $2\pi/q$. Notice that the smaller the value of $c\in(0,1)$, the more uniform the angles between the elements of $\{\Lambda^\eps_k : k\in L_q+mq\}$, and the better the alignment of two different such groups with similar $m$.

Consider the question of whether these groups of approximate eigenvalues are distinguishable in the spectrum of $T^{N,\eps}$. If $q\gg\eps^{-1/2}$, then estimate \eqref{eq:ev2} fails, and the positions of the approximations $\Lambda^\eps_k$ gives little information on the positions of the genuine eigenvalues $\lambda^{N,\eps}_k$. If, on the other hand, $q\ll\eps^{-1/2}$, then the modulus $|\Lambda^\eps_k|=e^{-\pi \eps k^2}$ is approximately one for $k$'s from several different groups $L_q+mq$. Thus, there will be significantly more than $q$ points close to the unit circle, and the groups will become more difficult to distinguish.

The most interesting regime is thus if $q$ is of the order of $\eps^{-1/2}$. Then the moduli of the $q$ points $\Lambda^\eps_k$ in one group $k\in L_q+mq$ are still comparable, but the difference between corresponding points in two consecutive groups with $k\in L_q+mq$ and with $k\in L_q+(m+1)q$, respectively, are notable --- there is at least a factor of $e^{-\pi \eps q^2}$. In particular, there is one distinguished group of eigenvalues very close to the unit circle, namely for $k\in L_q$, and that group is accompanied by pairs of groups of $\Lambda^{N,\eps}_k$ with $k\in L_q\pm mq$ for $m=1,2,3,\ldots$ that are moving successively closer to the origin as $m$ increases. 

For a specific example, let $n=1000$, $\theta=1/\pi$. The best three rational approximations $p/q$ of $\theta$ --- in the sense of smallest $c$ in \eqref{eq:nmbr} --- with $q<n$ are given by
\begin{equation*}
    (p,q) \in \{(1,3),(7,22),(113,355)\},
\end{equation*}
with corresponding values for $c$:
\[ c_{(1,3)} = 0.135\ldots,\ c_{(7,22)} = 0.062\ldots,\ c_{(113,355)} = 0.003\ldots \]
The constant $c$ attains exceptionally small values in this example; this is due to the number theoretic properties of $\pi$.

The corresponding approximate 355--, 22--, and 3--cycles should be visible approximately up to $\eps \lesssim (\pi \cdot 355)^{-2} \approx 8 \cdot 10^{-7}$, $\eps \lesssim 2 \cdot 10^{-4}$ and $\eps \lesssim 10^{-2}$ respectively. For $\eps \lesssim 1/n^2 = 10^{-6}$ we expect visible discretization artifacts in the spectrum, i.e.~the 355--cycle will already be hidden by discretization. The other two should appear as $\eps$ increases towards 1. This intuition will be confirmed in the numerical examples, see Figure \ref{fig:circle_map_irrational}.

\clearpage

\section{Numerical experiments}
\label{sec:experiments}
We perform three numerical experiments in order to demonstrate that our method
\begin{enumerate}
    \item indeed captures the spectrum of the regularized transfer operator correctly in the case of the shift map on the torus (section \ref{sec:ExamplesShift});
    \item correctly and robustly captures relevant macroscopic dynamical features in a well known low-dimensional example (section \ref{sec:ExamplesLorenz});
    \item is readily applicable to problems on high-dimensioanl state spaces (section \ref{sec:ExamplesAlanineDipeptide}).
\end{enumerate}
The following numerical experiments have been carried out in the Python programming language, partially using the KeOps package \cite{KeOps} which allows for an efficient and stable implementation of the Sinkhorn algorithm. The corresponding scripts are available at \href{https://github.com/gaioguy/EntropicTransferOperators}{\texttt{https://github.com/gaioguy/EntropicTransferOperators}}.

\label{sec:Examples}
\subsection{The shift map on the circle}
\label{sec:ExamplesShift}
As a first numerical test we consider the shift map from Section \ref{sec:shift map} on the circle $\cX=S^1=\R/\Z$.  We examine the case of rational $\theta=\frac13$ first, i.e.\ such that the circle decomposes into three subsets which are cyclically permuted by $F$. Correspondingly, we expect the spectrum of $\Gamma^{N,\eps}$ to exhibit the third roots of unity.  Indeed, the spectra shown in Figure~\ref{fig:circle_map_rational} clearly show these, except in the case where a regular lattic discretization is used and $\eps=10^{-6}$. In fact, in this case $\Gamma^{N,\eps}$ turns out to be a permutation matrix. This is in agreement with the discussion of Section \ref{sec:TorusSpectrumDiscussion} as $n=1000$ and $3$ are coprime.

\begin{figure}[ht]
    \centering
    \includegraphics[trim=28mm 7mm 35mm 12mm, clip, width=0.21\textwidth]{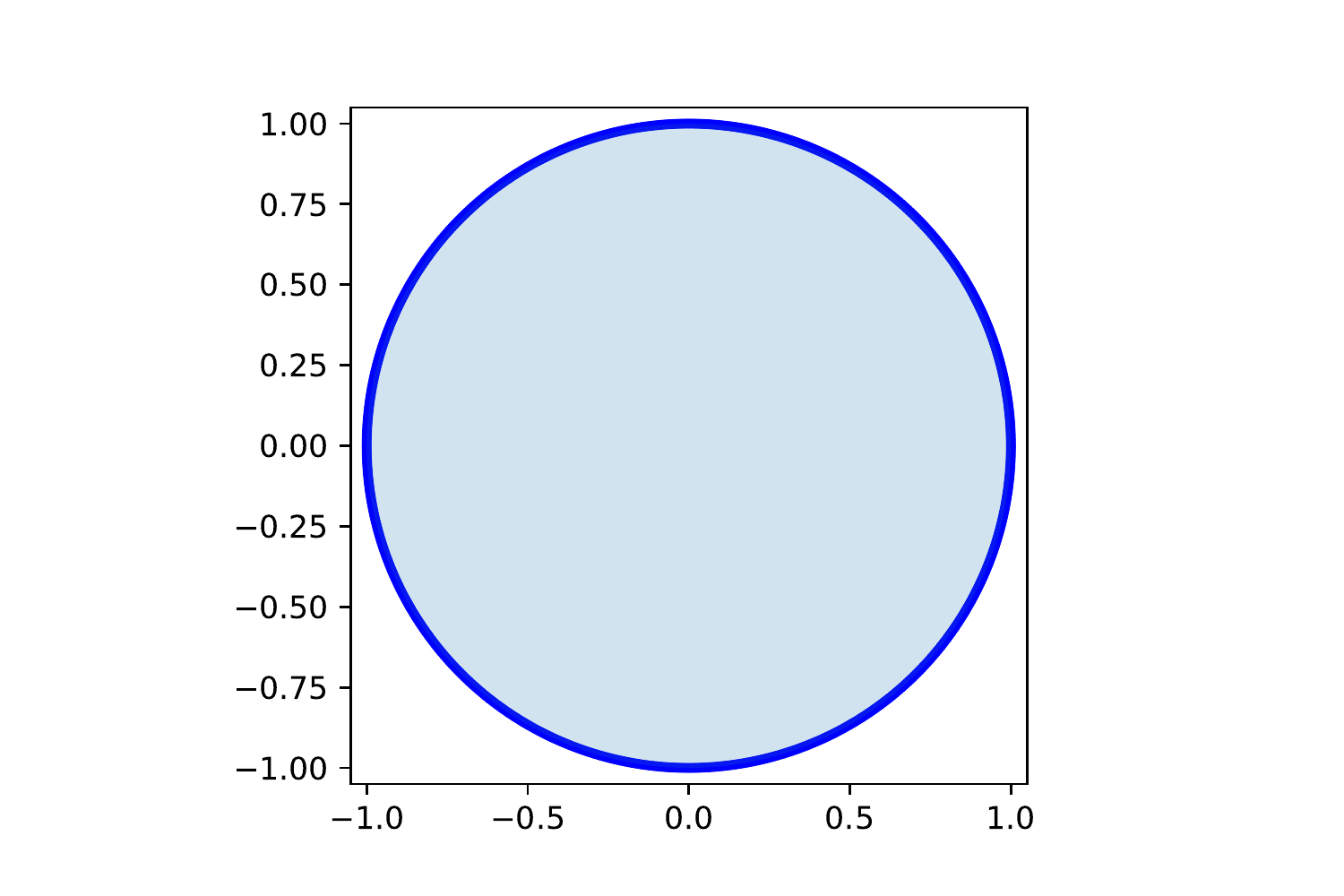}
    \includegraphics[trim=39mm 7mm 35mm 12mm, clip, width=0.185\textwidth]{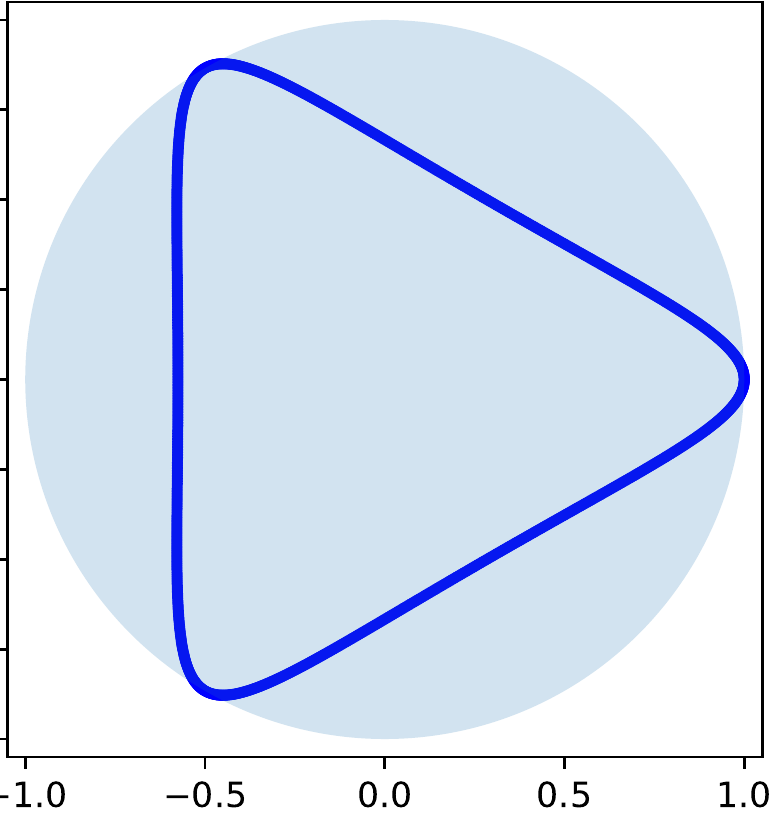}
    \includegraphics[trim=39mm 7mm 35mm 12mm, clip, width=0.185\textwidth]{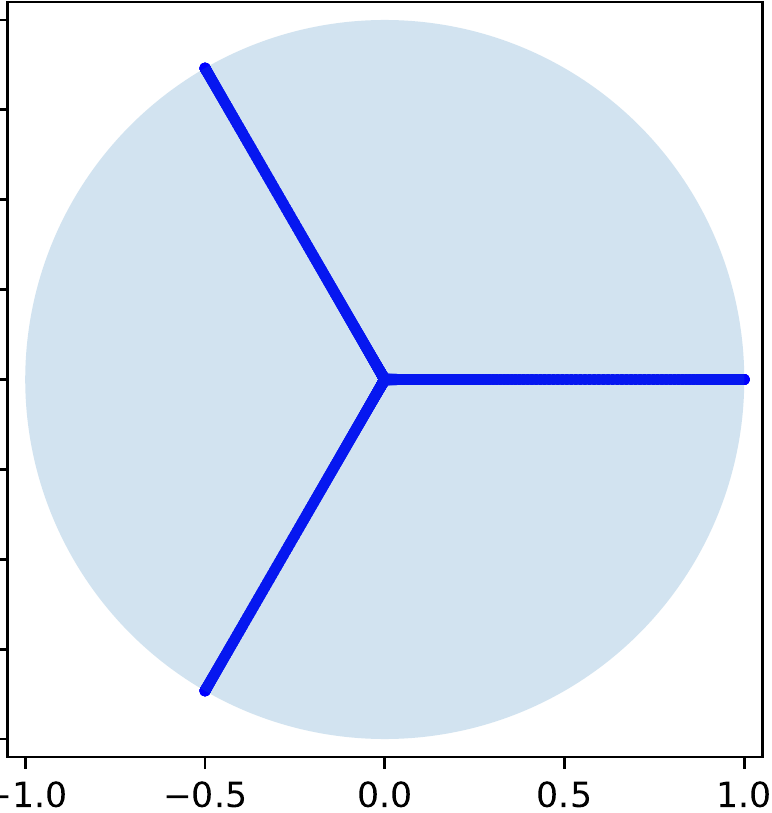}
    \includegraphics[trim=39mm 7mm 35mm 12mm, clip, width=0.185\textwidth]{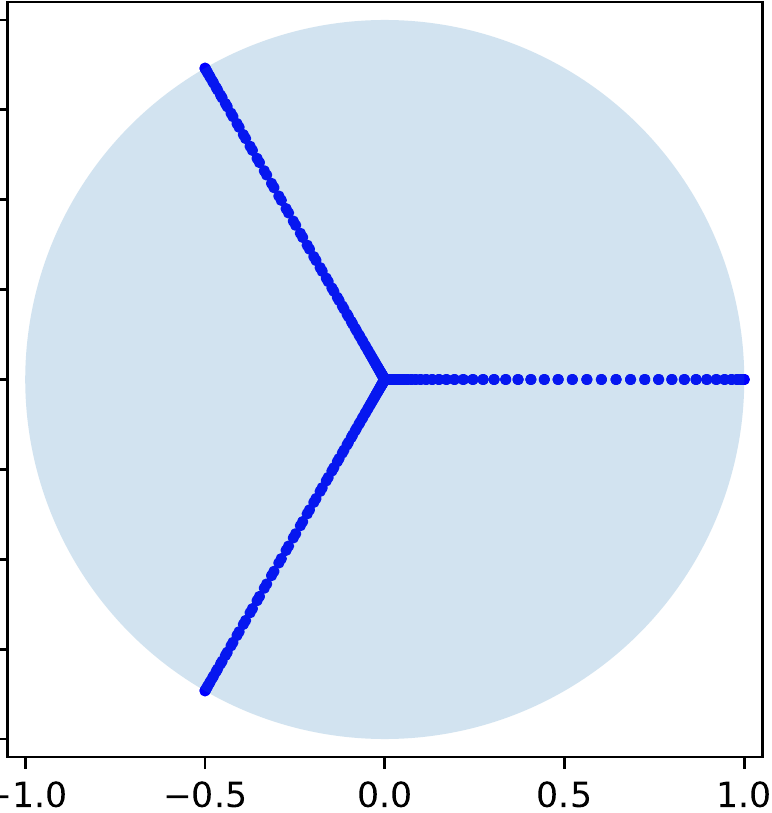}
    \includegraphics[trim=39mm 7mm 35mm 12mm, clip, width=0.185\textwidth]{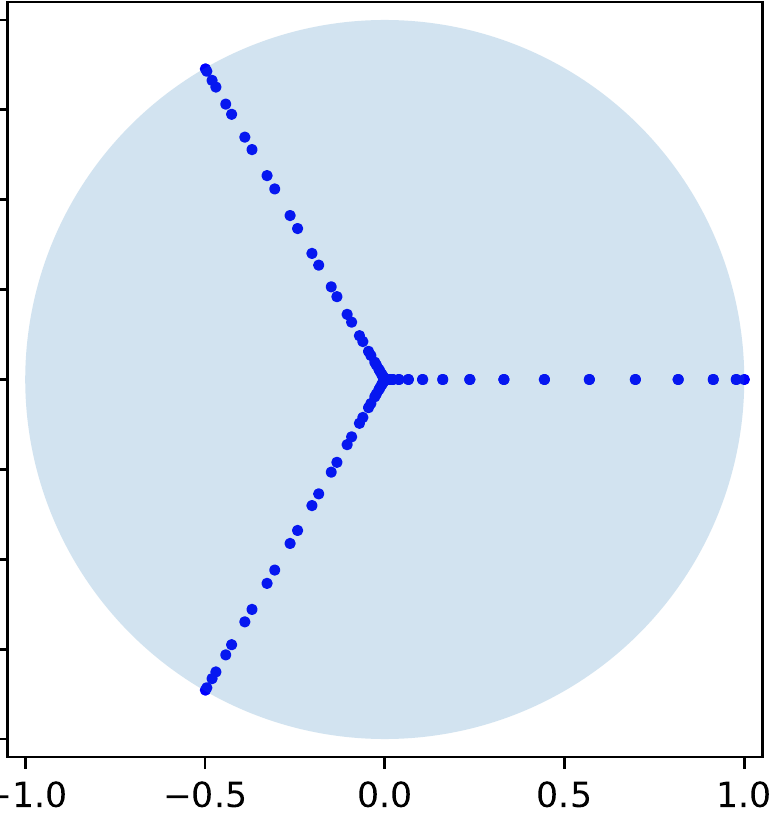}
    \\[5mm]  
    \includegraphics[trim=14mm 3mm 17mm 3mm, clip, width=0.21\textwidth]{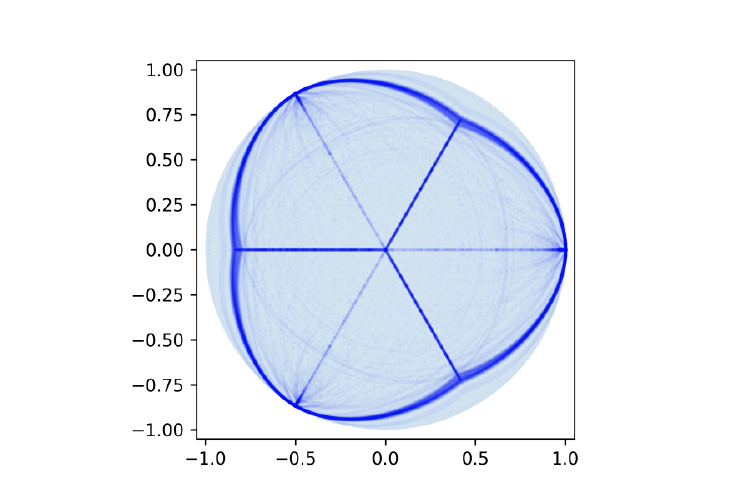}
    \includegraphics[trim=19.5mm 3mm 17mm 3mm, clip, width=0.185\textwidth]{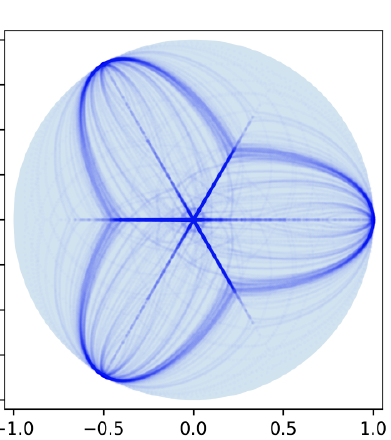}
    \includegraphics[trim=19.5mm 3mm 17mm 3mm, clip, width=0.185\textwidth]{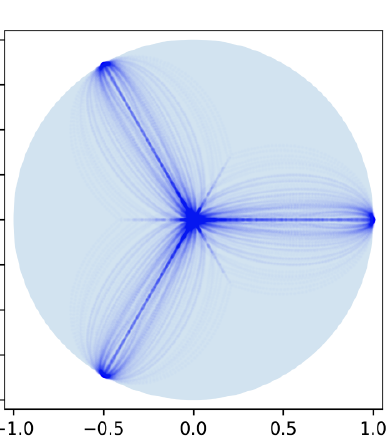}
    \includegraphics[trim=19.5mm 3mm 17mm 3mm, clip, width=0.185\textwidth]{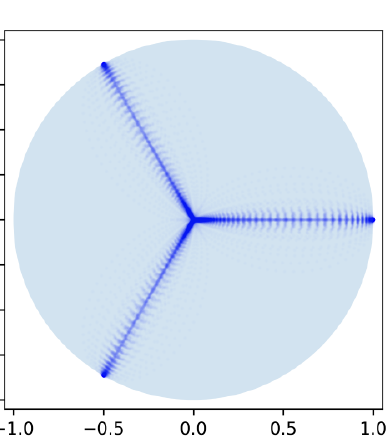}
    \includegraphics[trim=19.5mm 3mm 17mm 3mm, clip, width=0.185\textwidth]{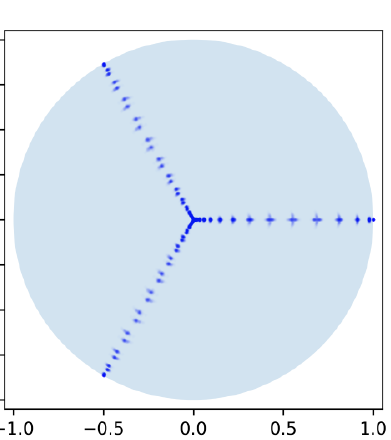}
    \label{fig:circle_map_rational-random}
    \caption{Spectra of $\Gamma^{N,\eps}$ for the circle shift map with $\theta=\frac13$ using $N=1000$ points, for $\eps\in\{10^{-6}, 10^{-5}, 10^{-4}, 10^{-3}, 10^{-2}\}$ (left to right) and two different choices of the discretization: Upper row: points on a regular lattice; lower row: points chosen randomly from a uniform distribution. In the lower row, each plot shows the spectra for 100 realisations of the discretization points. }
    \label{fig:circle_map_rational}
\end{figure}

For irrational $\theta=\frac1\pi\approx\frac13$ the spectra that we obtain experimentally are shown in Figure~\ref{fig:circle_map_irrational}. As expected, with increasing $\eps$ they are in excellent agreement with the asymptotics given by Proposition \ref{prop:eigenvalue asymptotics}.
The map does not exhibit an exact macroscopic 3-cycle.
However, as discussed in Section \ref{sec:TorusSpectrumDiscussion}, it exhibits approximate $q$-cycles for various $q$, in particular for $q=22$ and $q=3$ and these are visible in the corresponding regimes of $\eps$.

In Section \ref{sec:shift map}, the discretization of $\cX$ was a regular lattice as this allowed an analytic  treatment of the system by leveraging the high level of symmetry. The convergence results of Section \ref{sec:Method} do not require such regularity. This is confirmed by Figures \ref{fig:circle_map_rational} and \ref{fig:circle_map_irrational}, where histograms of the spectra of $\Gamma^{N,\eps}$ are shown for 100 realizations of a random choice of the underlying point cloud, i.e.\ each point $x_j$ is chosen independently from a uniform distribution on $S^1$. 

To summarize: The numerical results on the shift map are in excellent aggreement with the theorems on the spectrum in section \ref{sec:shift map}.

\begin{figure}[ht] 

\begin{subfigure}[b]{\textwidth}
    \centering
    \includegraphics[trim=28mm 7mm 35mm 12mm, clip, width=0.21\textwidth]{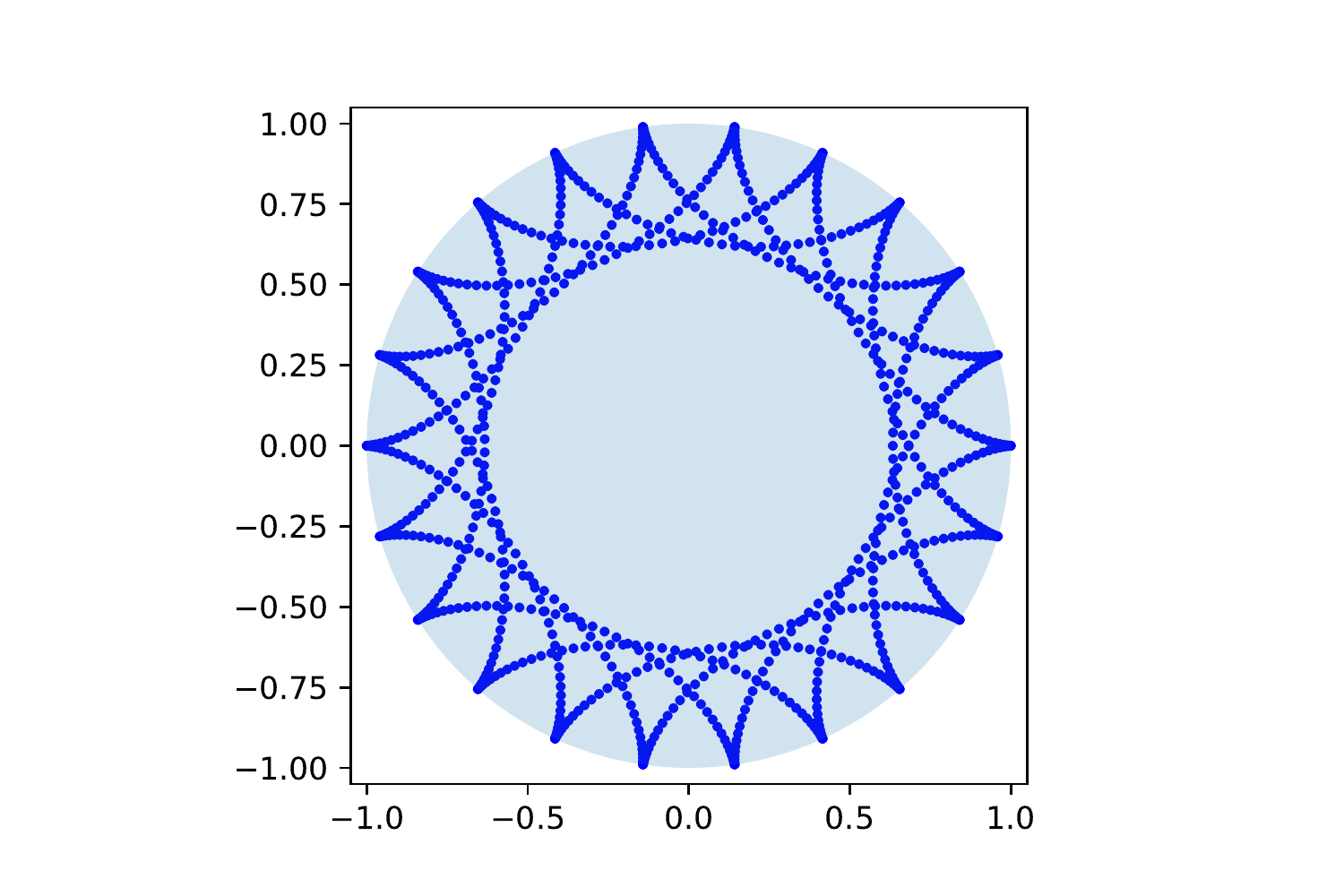}
    \includegraphics[trim=39mm 7mm 35mm 12mm, clip, width=0.185\textwidth]{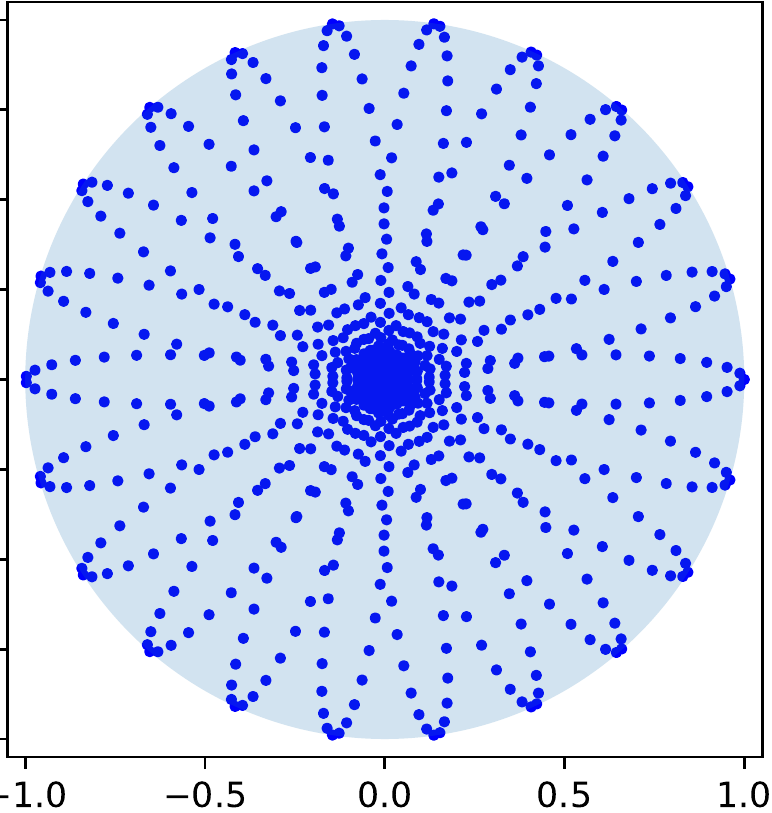}
    \includegraphics[trim=39mm 7mm 35mm 12mm, clip, width=0.185\textwidth]{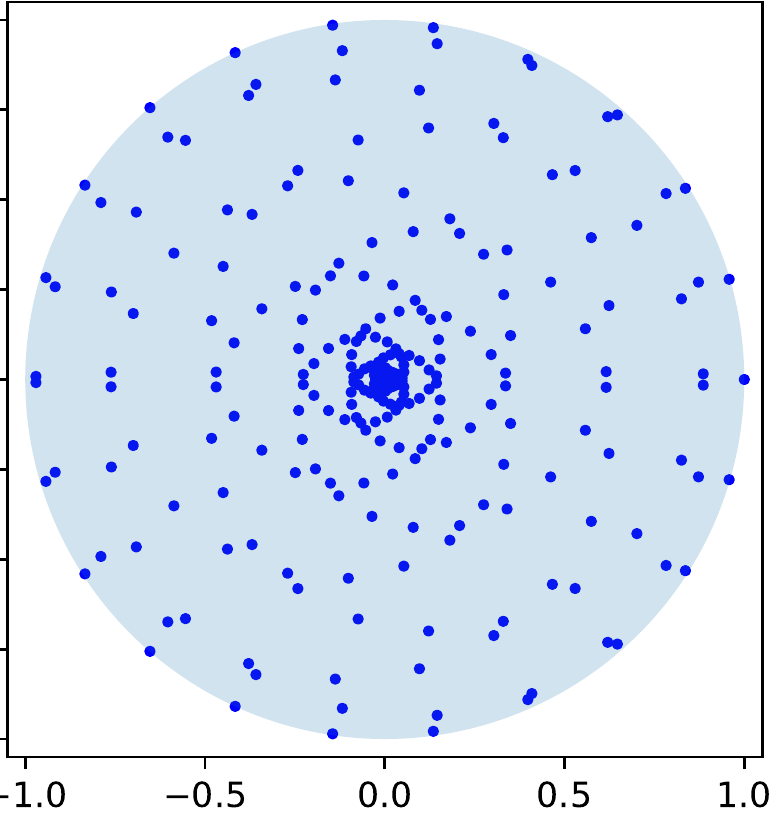}
    \includegraphics[trim=39mm 7mm 35mm 12mm, clip, width=0.185\textwidth]{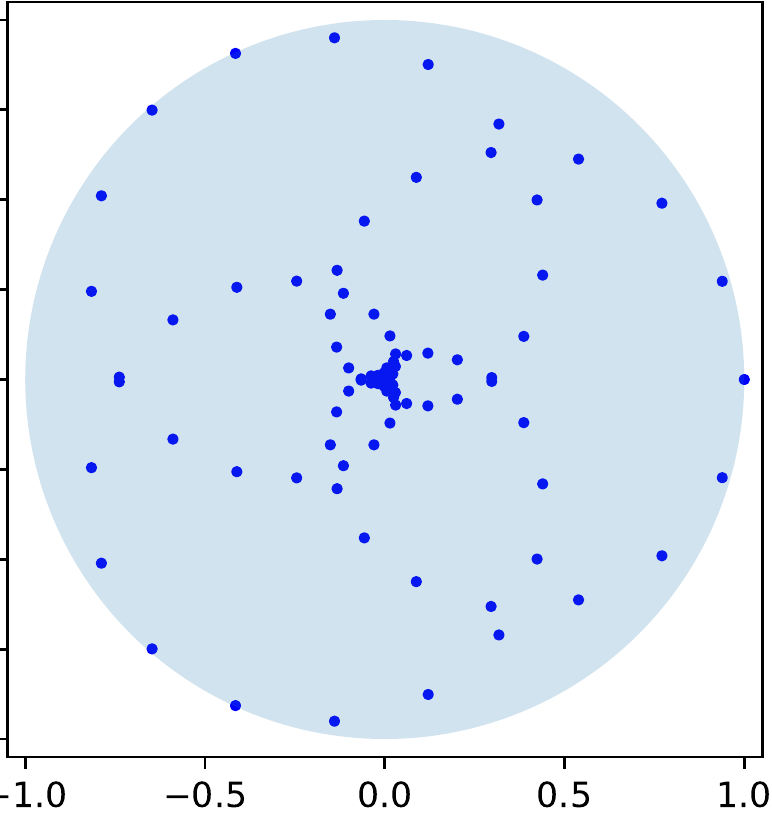}
    \includegraphics[trim=39mm 7mm 35mm 12mm, clip, width=0.185\textwidth]{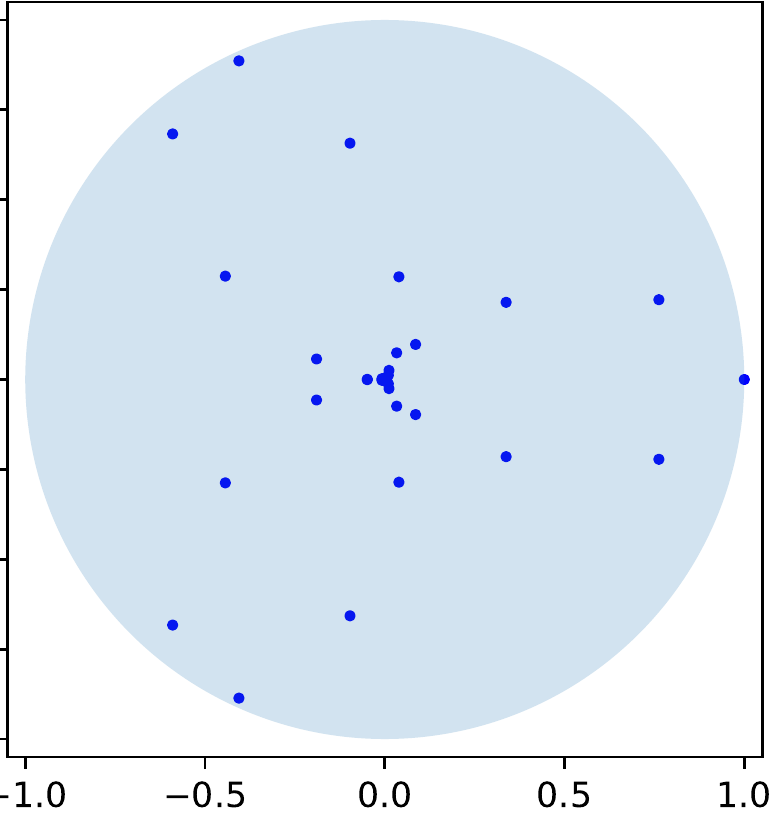}
\end{subfigure}

\begin{subfigure}[b]{\textwidth}
    \centering
    \includegraphics[trim=14mm 3mm 17mm 3mm, clip, width=0.21\textwidth]{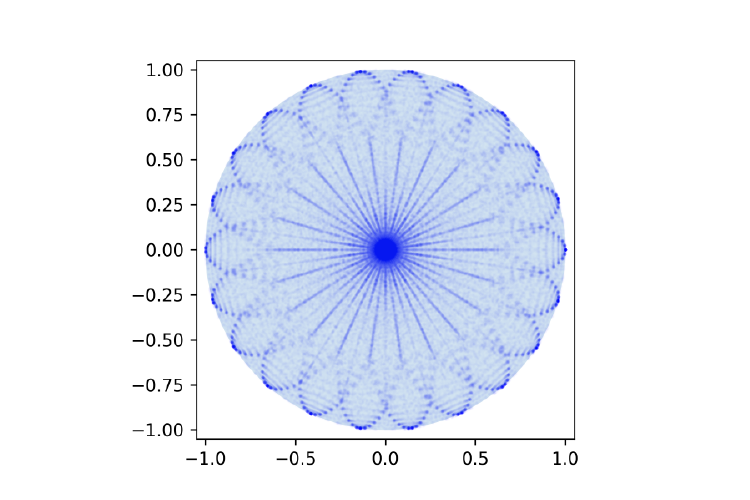}
    \includegraphics[trim=19.5mm 3mm 17mm 3mm, clip, width=0.185\textwidth]{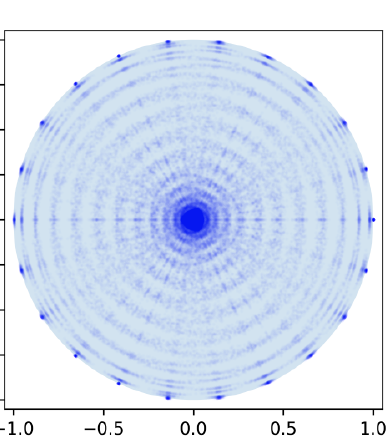}
    \includegraphics[trim=19.5mm 3mm 17mm 3mm, clip, width=0.185\textwidth]{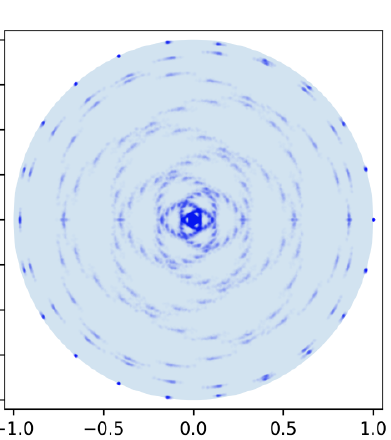}
    \includegraphics[trim=19.5mm 3mm 17mm 3mm, clip, width=0.185\textwidth]{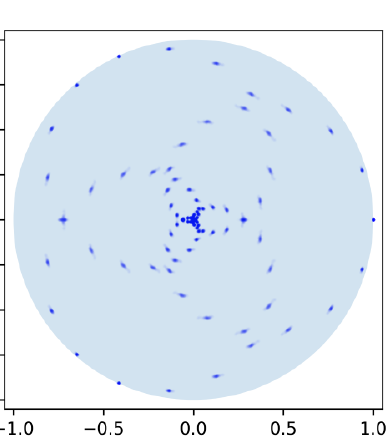}
    \includegraphics[trim=19.5mm 3mm 17mm 3mm, clip, width=0.185\textwidth]{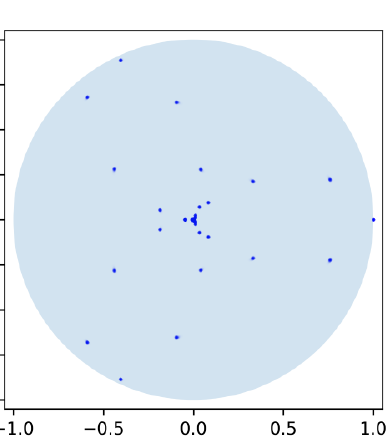}
    \label{fig:circle_map_irrational-random} 
\end{subfigure}

\begin{subfigure}[b]{\textwidth}
    \centering
    \includegraphics[trim=28mm 7mm 35mm 5mm, clip,  width=0.21\textwidth]{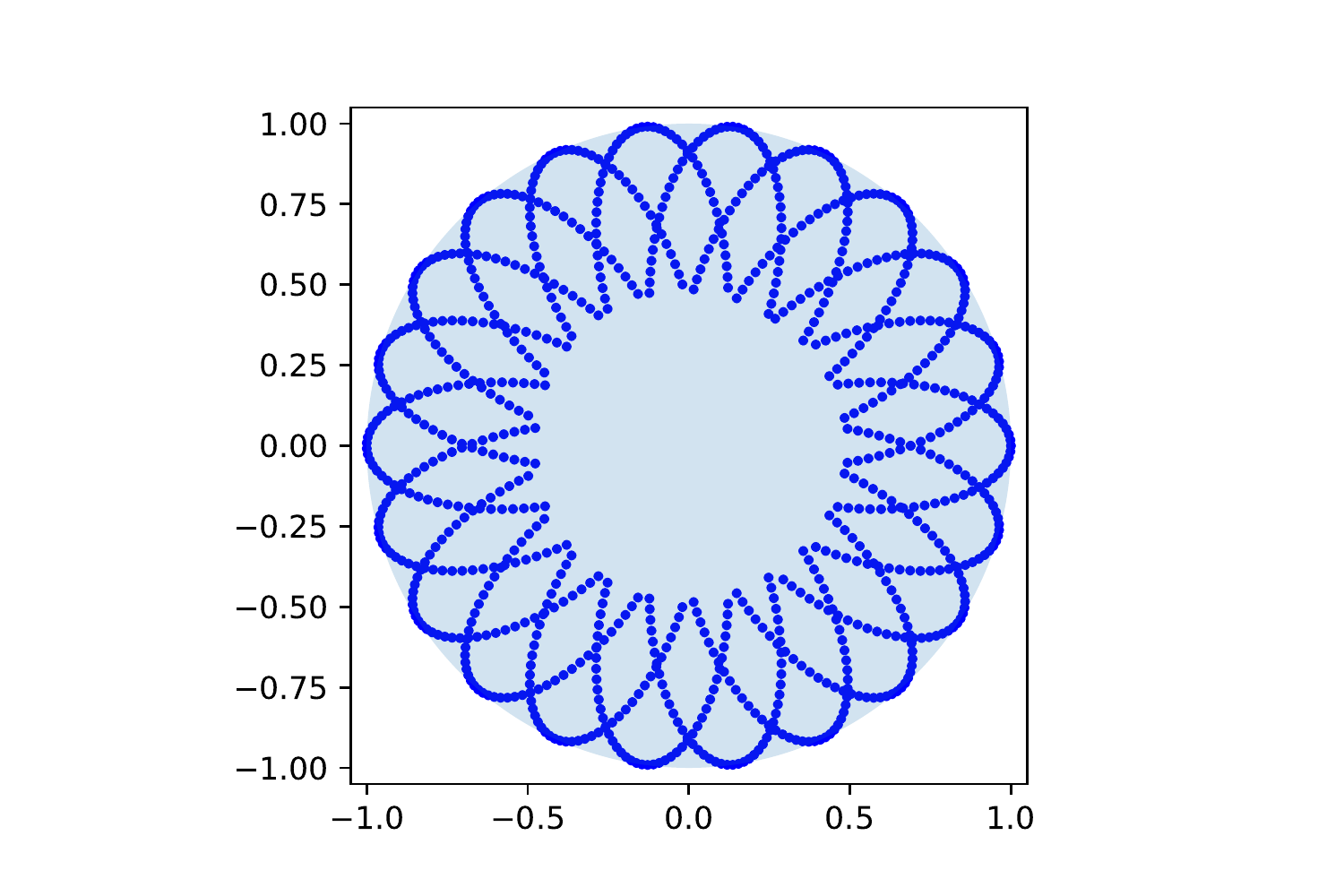}
    \includegraphics[trim=39mm 7mm 35mm 5mm, clip,  width=0.185\textwidth]{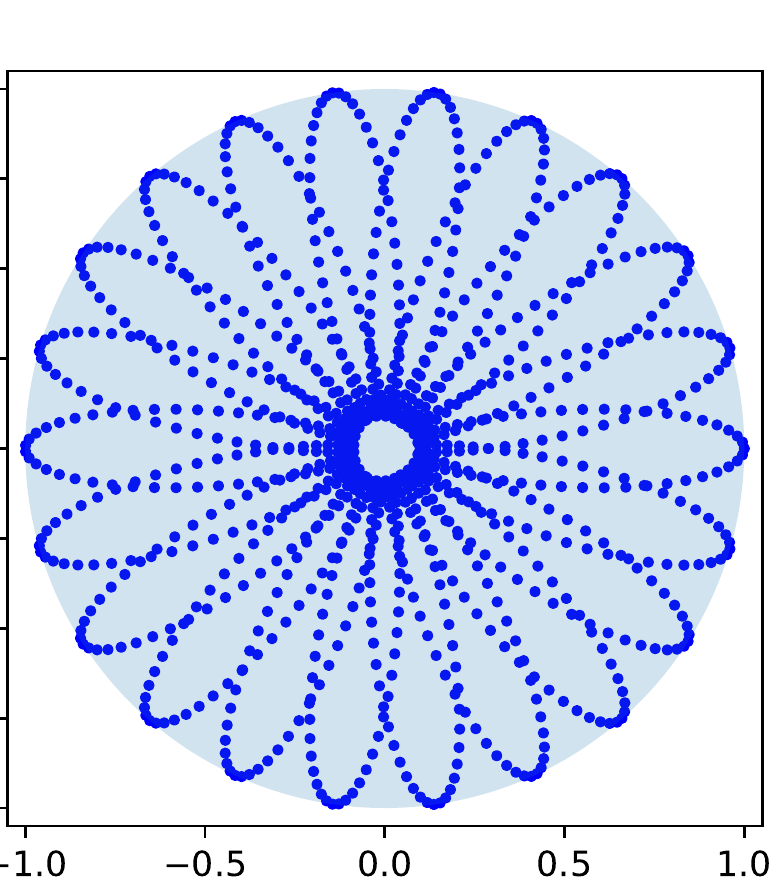}
    \includegraphics[trim=39mm 7mm 35mm 5mm, clip, width=0.185\textwidth]{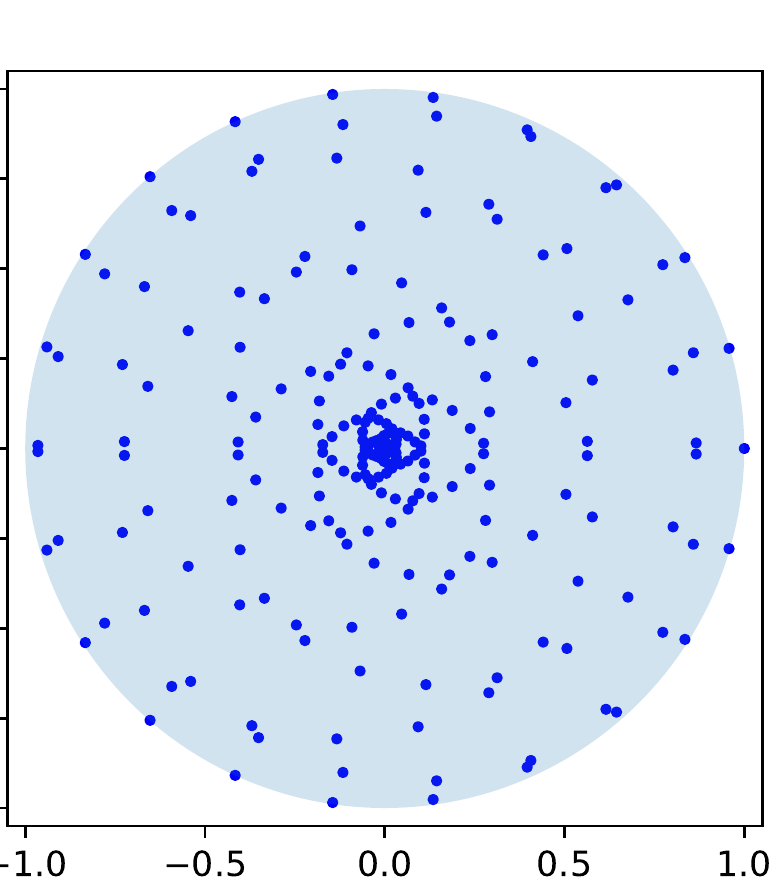}
    \includegraphics[trim=39mm 7mm 35mm 5mm, clip, width=0.185\textwidth]{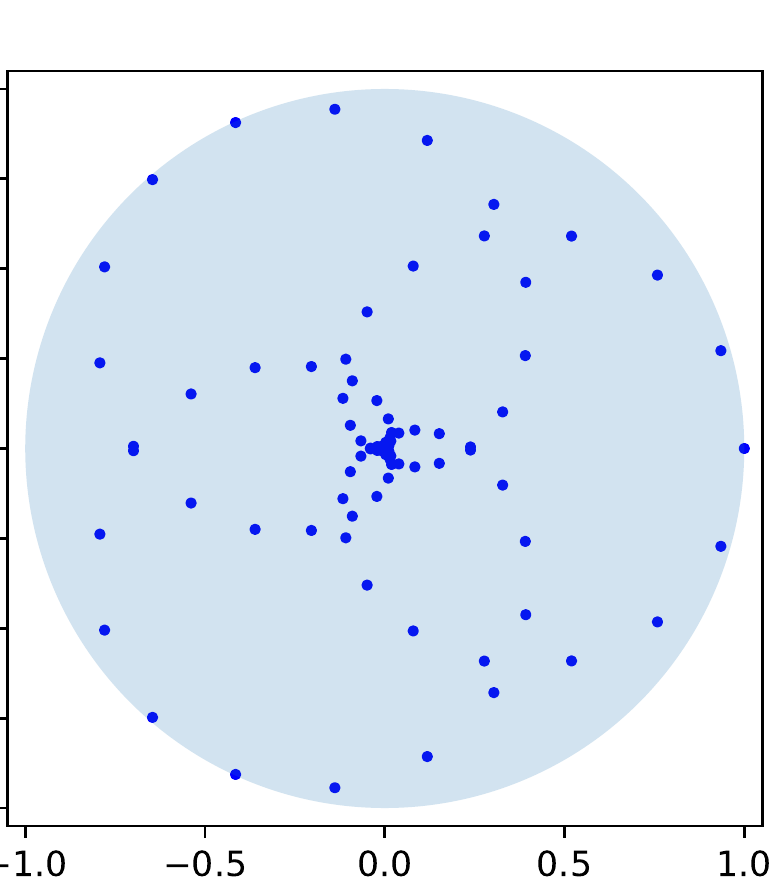}
    \includegraphics[trim=39mm 7mm 35mm 5mm, clip, width=0.185\textwidth]{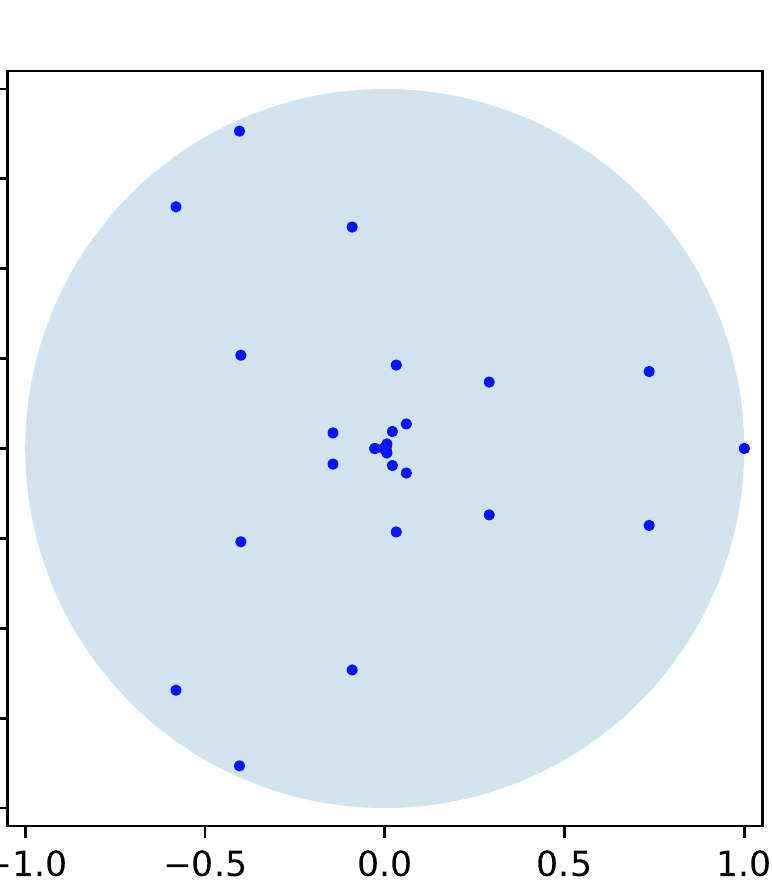}
\end{subfigure}

    \caption{Spectra of $\Gamma^{N,\eps}$ for the circle shift map with $\theta=\frac1\pi$ using $N=1000$ points, for $\eps\in \{10^{-5}, 10^{-4},10^{-3},10^{-2}, 10^{-1}\}$ (left to right) and two different choices of the discretization: Upper row: points on a regular lattice; center row: points chosen randomly from a uniform distribution. In the center row, each plot shows the spectra for 100 realisations of the discretization points. For comparison, we show the eigenvalue asymptotics from Proposition \ref{prop:eigenvalue asymptotics} for $\eps\in 3\cdot \{10^{-7}, 10^{-6},10^{-5}, 10^{-4},10^{-3}\}$ (left to right) in the lower row.}
    \label{fig:circle_map_irrational}
\end{figure}

\subsection{The Lorenz system}  
\label{sec:ExamplesLorenz}

For a second experiment, we consider the classical Lorenz system \cite{Lo63}
\begin{align*}
\dot u & = \sigma(v-u)\\
\dot v & = u(\rho-w) - v\\
\dot w & = uv-\beta w
\end{align*}
with parameter values $\sigma=10, \rho = 28, \beta = 8/3$.  The system possesses a robust strange attractor \cite{Tu99} which can be decomposed into several almost invariant sets \cite{froyland2003detecting}.  We approximate the attractor by a point cloud $x_1^N,\ldots,x^N_N$ resulting from an equidistant sampling of the trajectory  $X=(u,v,w):[200,2000]\to\R^3$ with initial value $X(0)=(1,1,1)$.  The map $F$ is the flow map $F^\tau$ with $\tau=0.1$.   

\begin{figure}[ht]
    \centering
    \includegraphics[width=0.47\textwidth]{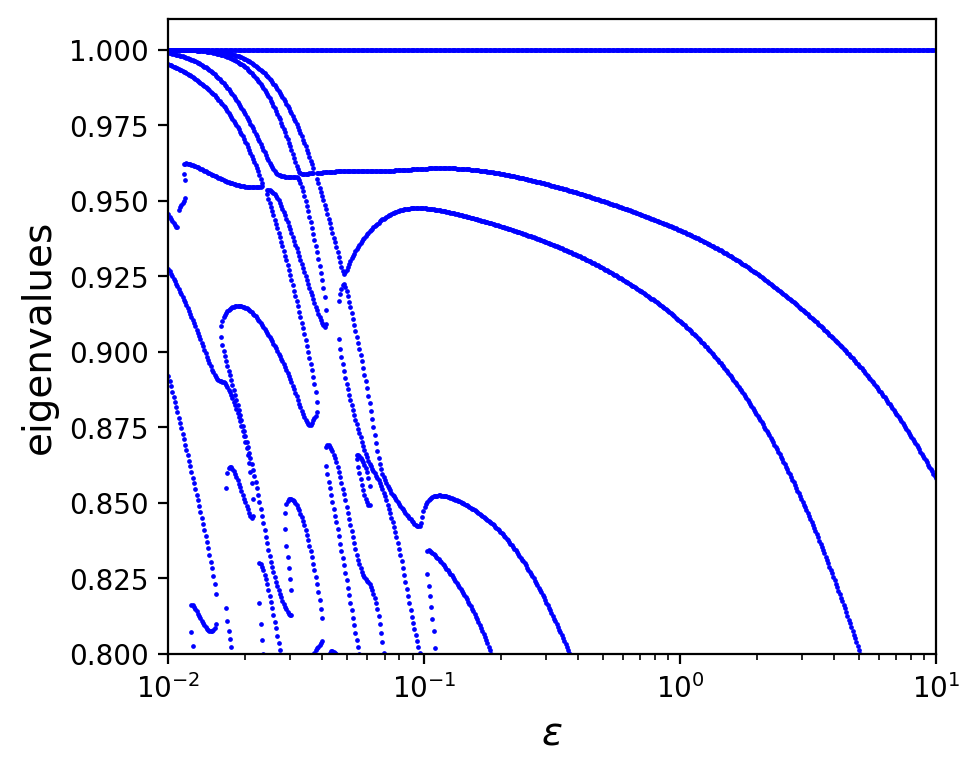}
    \quad
    \includegraphics[width=0.47\textwidth]{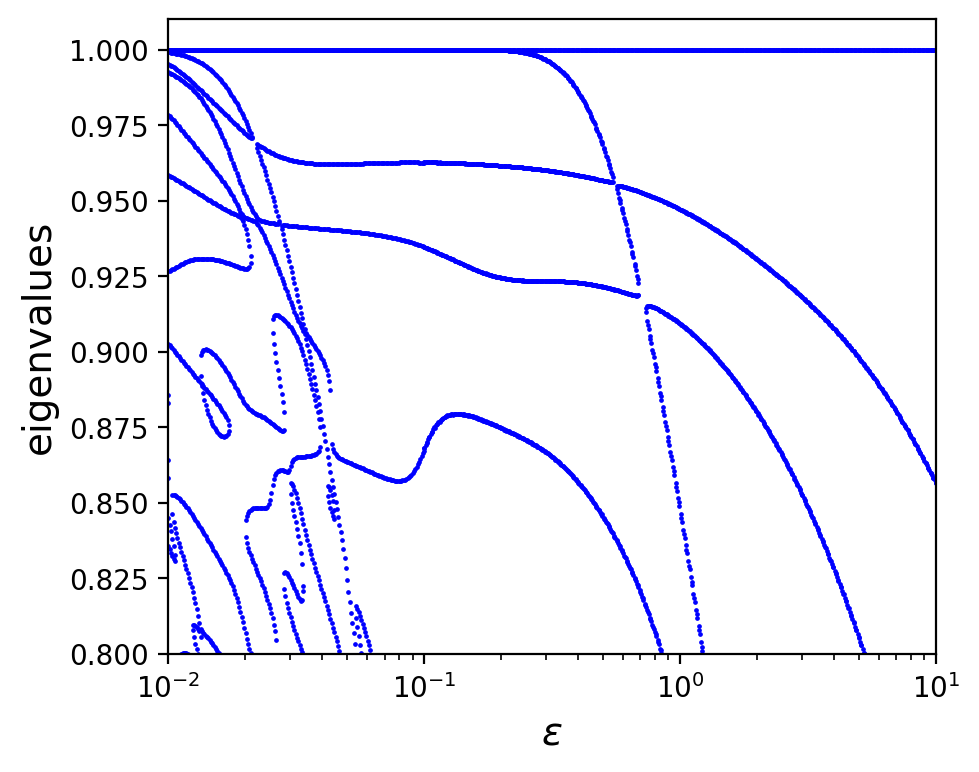}
    \caption{Largest real eigenvalues of $\Gamma^{N,\eps}$ with $N=1000$ (left) and $N=2000$ (right) as a function of $\eps$ for the Lorenz system.}
    \label{fig:Lorenz}
\end{figure}

\begin{figure}[ht]
    \centering
    \includegraphics[trim=0pt 0pt 50pt 0pt, clip, width=0.45\textwidth]{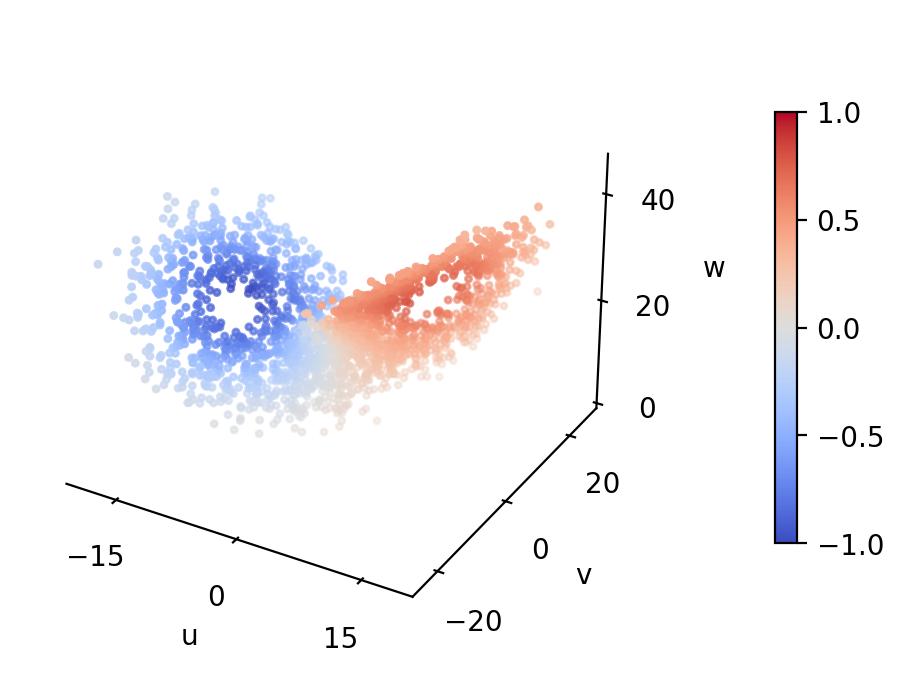}
    \includegraphics[width=0.53\textwidth]{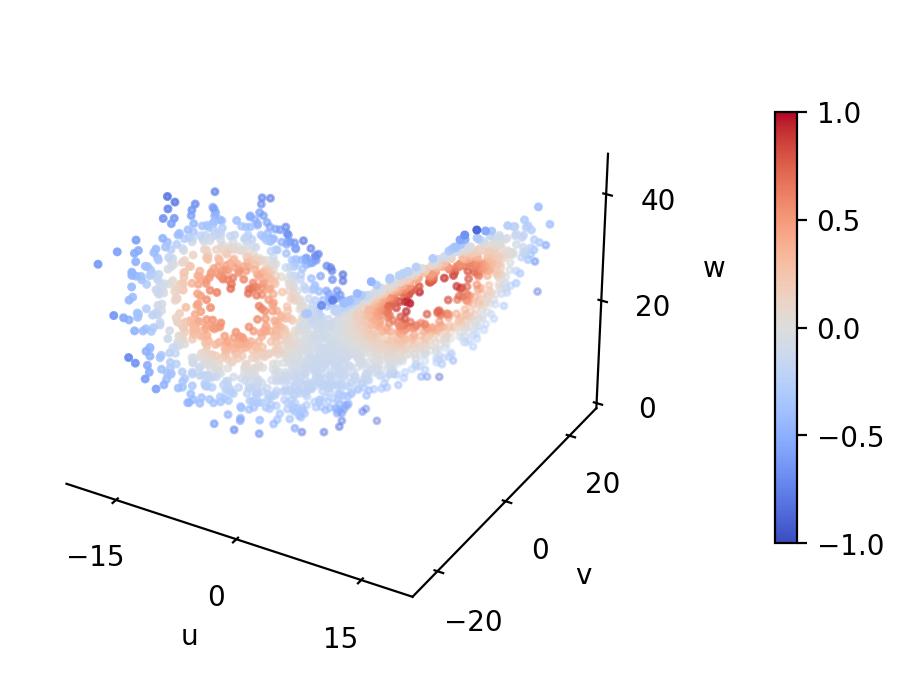}
    \caption{Eigenvectors at $\lambda_1$ (left) and  $\lambda_2$ (right) of $\Gamma^{N,\eps}$ with $N=2000$ and $\eps=10$ for the Lorenz system.}
    \label{fig:Lorenz2}
\end{figure}

The largest real eigenvalues of $\Gamma^{N,\eps}$ for $N=1000$ and $N=2000$ are shown in Figure~\ref{fig:Lorenz} as a function of $\eps$. For both values of $N$, there are two real eigenvalues $\lambda_1, \lambda_2 < 1$ that decay much slower with increasing $\eps$ than the other eigenvalues.  The corresponding eigenvectors each give rise to a decomposition of the point cloud into two almost invariant sets via their sign structure as shown in Figure~\ref{fig:Lorenz2}. This is in aggreement with previuous results on the Lorenz system in, e.g., \cite{DeJu99,froyland2003detecting}. For $N=2000$ points, there is also a third eigenvalue $\lambda_3$ very close to 1 which quickly decays at $\eps\approx 1$.
It corresponds to a point $x_i$ near the origin where the point density is low and $F(x_i) \approx x_i$, such that for sufficiently small $\eps$, mass at $F(x_i)$ gets transported almost exclusively back to $x_i$, thus forming a spurious invariant set that decays quickly as $\sqrt{\eps}$ reaches the scale of nearest neighbor distances around $x_i$. Such isolated points become less likely as $N$ is increased and their eigenvectors can be identified easily.

\subsection{Alanine dipeptide}
\label{sec:ExamplesAlanineDipeptide}

As a third experiment we analyse the dynamics of alanine dipeptide, a small biomolecule which is often used as a non-trivial model system for the analysis of macroscopic dynamics, cf.\ for example \cite{klus2018kernel}.  Like in the Lorenz system, we use trajectory data, here on the positions of the heavy atoms as provided by \cite{nuske2017markov,wehmeyer2018time}\footnote{The data is available for download at \textbf{\href{https://markovmodel.github.io/mdshare/ALA2/}{https://markovmodel.github.io/mdshare/ALA2/}}.}.  The original time series results from time integration of the molecule over 250\,ns with a time step of 1\,ps, yielding $M=250\,000$ points $\hat x_1,\ldots,\hat x_M$ in $\R^{30}$.  Here, we subsample this trajectory and use only every 50th (resp. 25th) point, yielding a sequence $x^N_1,\ldots,x^N_N$, $N=5\,000$ (resp. $10\,000$), in $\R^{30}$ which forms the data for our analysis.  We define the map $F$ by time-delay, that is we let $F(\hat x_j):=\hat x_{j+10}$ for $j=1,\ldots,M-10$.  This time lag of 10 ps has been chosen experimentally such that the real spectrum of $\Gamma^{N,\eps}$ is reasonably close to 1. We note, however, that the macroscopic structure of the eigenvectors, cf.\ Figure \ref{fig:AD2}, remains the same for time lags between 1\,ps and 30\,ps while, of course, the spectrum changes.

\begin{figure}[ht]
    \centering
    \includegraphics[width=0.48\textwidth]{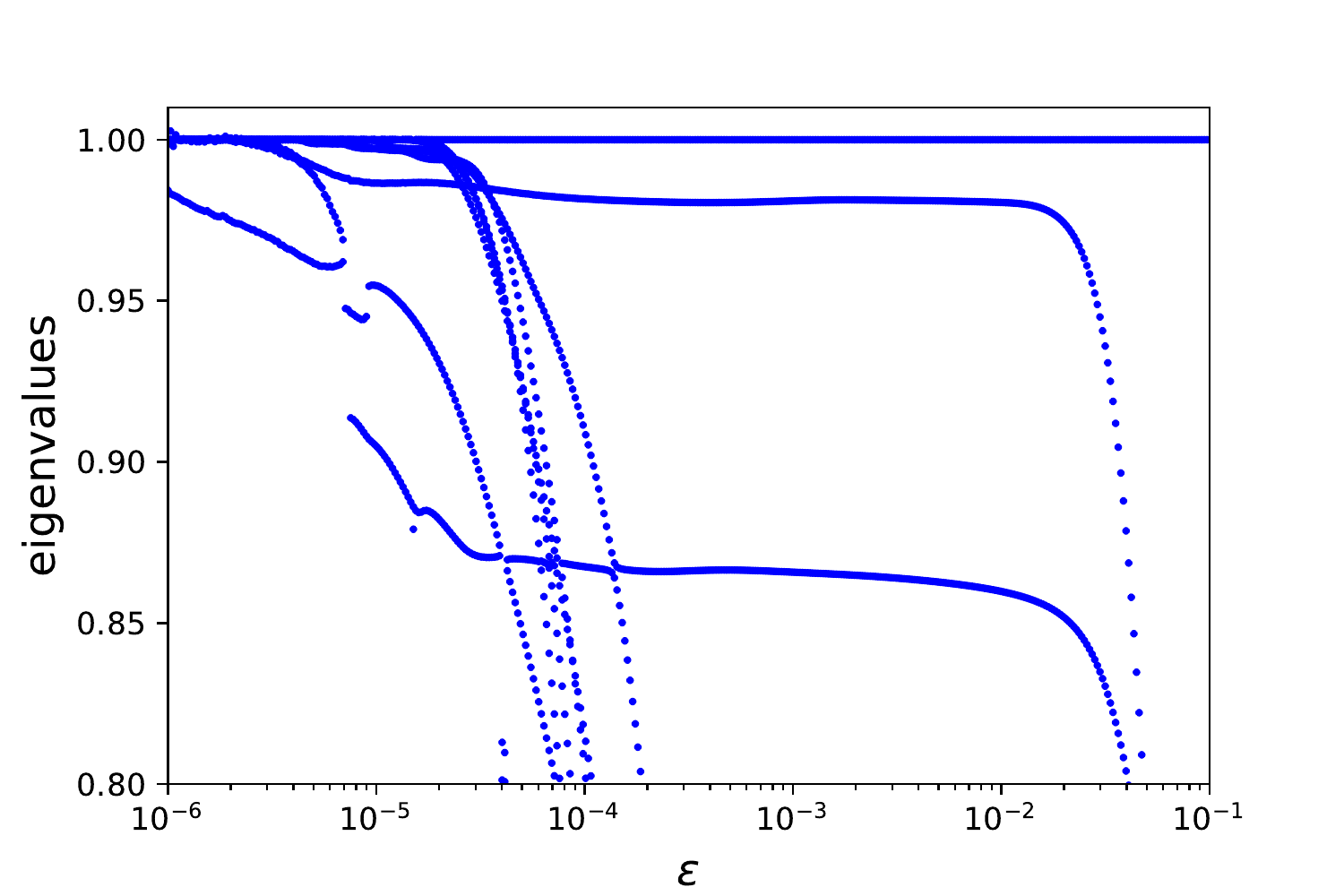}
    \includegraphics[width=0.48\textwidth]{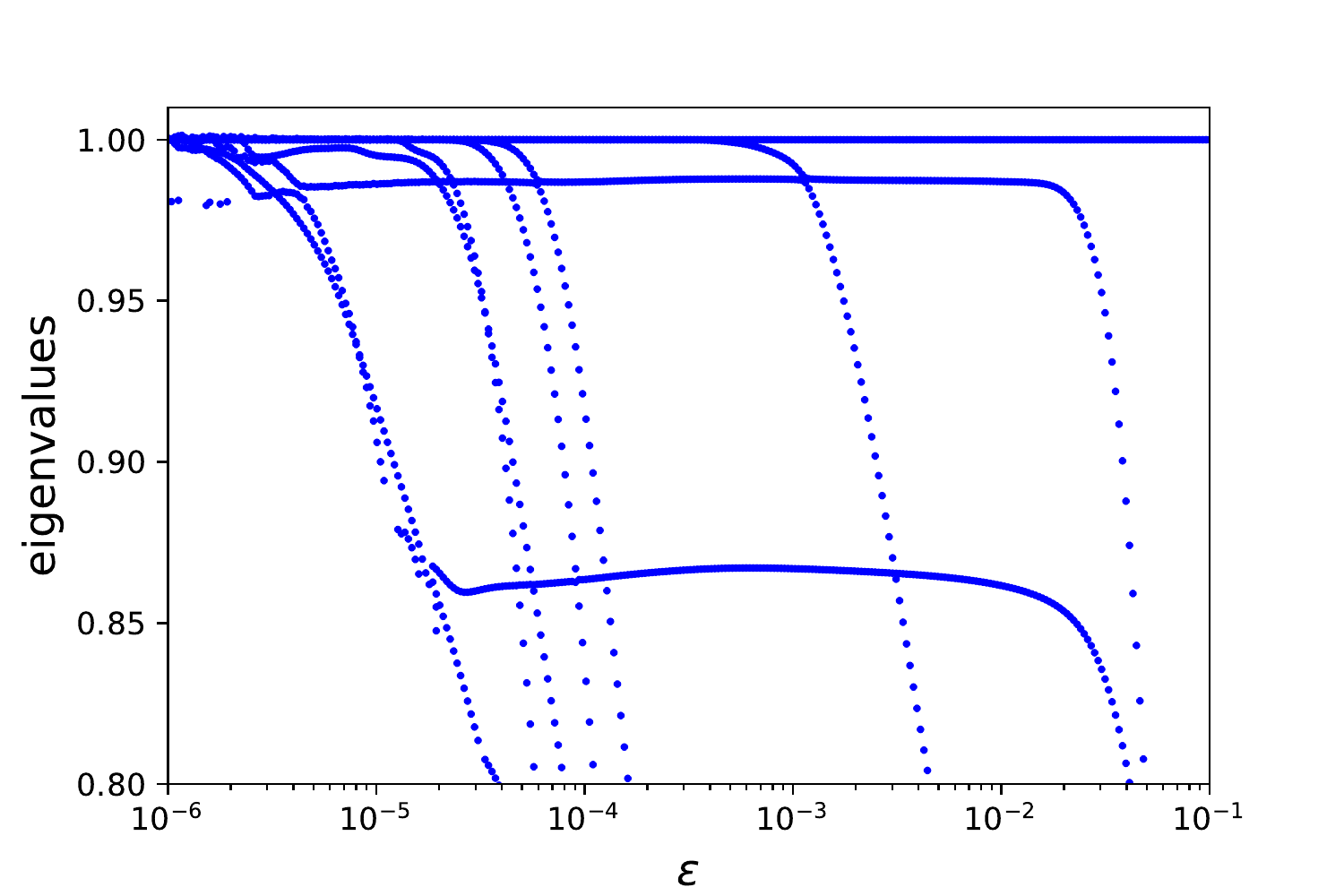} 
    \caption{The 10 largest real eigenvalues of $\Gamma^{N,\eps}$ as a function of $\eps$ for the alanine-dipeptide model for $N=5000$ (left) and $N=10000$ samples (right).}
    \label{fig:AD1}
\end{figure} 

The left panel in Figure \ref{fig:AD1} shows the dominant real spectrum of $\Gamma^{N,\eps}$ close to $1$ as a function of~$\eps$ for $N=5000$ and $N=10000$. 
There appear to be two eigenvalues $\lambda_1\approx 0.98$ and $\lambda_2 \approx 0.87$ which are essentially constant over several orders of magnitude of $\eps$. 
In fact, the corresponding eigenvectors give rise to the decomposition of the point cloud into three almost invariant sets as shown in Figure~\ref{fig:AD2}. 
There are several eigenvalues that are very close to one for $\eps\in [10^{-6},10^{-4.5}]$ and that decay rapidly  for $\eps\in [10^{-4.5},10^{-2}]$. In contrast, $\lambda_1$ and $\lambda_2$ do not decay for $\eps\in [10^{-4.5},10^{-2}]$. We conjecture that these correspond to spurious invariant subsets of the point cloud as described in the experiment on the Lorenz system. 

In order to better understand the behaviour of the eigenvalues with increasing $\eps$, we consider the simple Markov chain model with three states shown in Figure~\ref{fig:simple_MM}. We assume that the distance $d_1$ between the states 2 and 3 is much (i.e.\ an order of magnitude) smaller than the distance $d_2$ between 1 and 2 resp.\ 3. The transition matrix for this Markov chain is given by 
\[
T = \begin{bmatrix}
     1-2p_2 & p_2 & p_2 \\
     p_2 & 1-p_1-p_2 & p_1 \\
     p_2 & p_1 & 1-p_1-p_2
\end{bmatrix}.     
\]
Let $G^\eps\in\R^{3\times 3}$ denote the matrix of the entropically regularized transport plan for the associated distance matrix
\[
 \begin{bmatrix}
     0 & d_2 & d_2 \\
     d_2 & 0 & d_1 \\
     d_2 & d_1 & 0
\end{bmatrix}     ,
\]
so that the entropically smoothed transfer operator is represented by the matrix $\Gamma^\eps=G^\eps T$. There are two realms for $\eps$ for which the eigenvalues of $\Gamma^\eps$ are essentially constant: For $\eps \ll d_2^2$, the smoothing has essentially no effect and the spectrum of $\Gamma^\eps$ approximately coincides with that of $T$ itself.  For $\eps \approx d_2^2$, the smoothing starts to have an effect on the $2\leftrightarrow 3$ transition so that the associated eigenvalue starts to drop. At $\eps \approx d_1^2$, also the $1\leftrightarrow \{2,3\}$ transition is affected and the second real eigenvalue close to 1 starts to drop as well. 
 
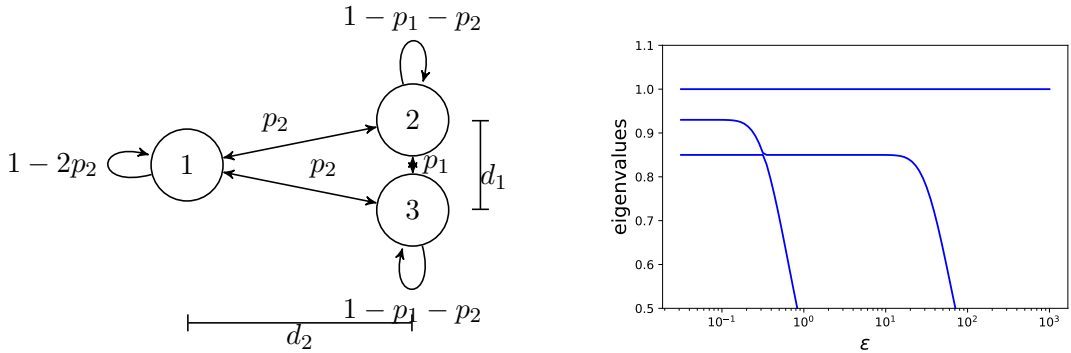
\begin{figure}[H]
    \centering

\begin{subfigure}[c]{0.48\textwidth}
    \centering
    \begin{tikzpicture}[scale=0.60, ->, >=stealth', auto, semithick]
    
        \node[state] at (0, 1) (1)  {$1$};
        \node[state] at (5, 2) (2)  {$2$};
        \node[state] at (5, 0) (3) {$3$};

        \path 
        (1) edge [loop left] node {$1-2p_2$} (1)
        (2) edge [loop above] node {$1-p_1-p_2$} (2)
        (3) edge [loop below] node {$1-p_1-p_2$} (3)
        (1) edge [<->] node {$p_2$} (2)
        (1) edge [<->] node {$p_2$} (3)
        (2) edge [<->] node {$p_1$} (3);

        \draw[|-|] (0,-2.5) -- (5,-2.5);
        \draw[-] (2.5,-2.8) node {$d_2$};

        \draw[|-|] (6.5,0) -- (6.5,2);
        \draw[-] (6.8,0.75) node {$d_1$};
     
        \end{tikzpicture}
	\end{subfigure}
	\qquad
	\begin{subfigure}[c]{0.46\textwidth}
    \centering
        \includegraphics[width=1.\textwidth]{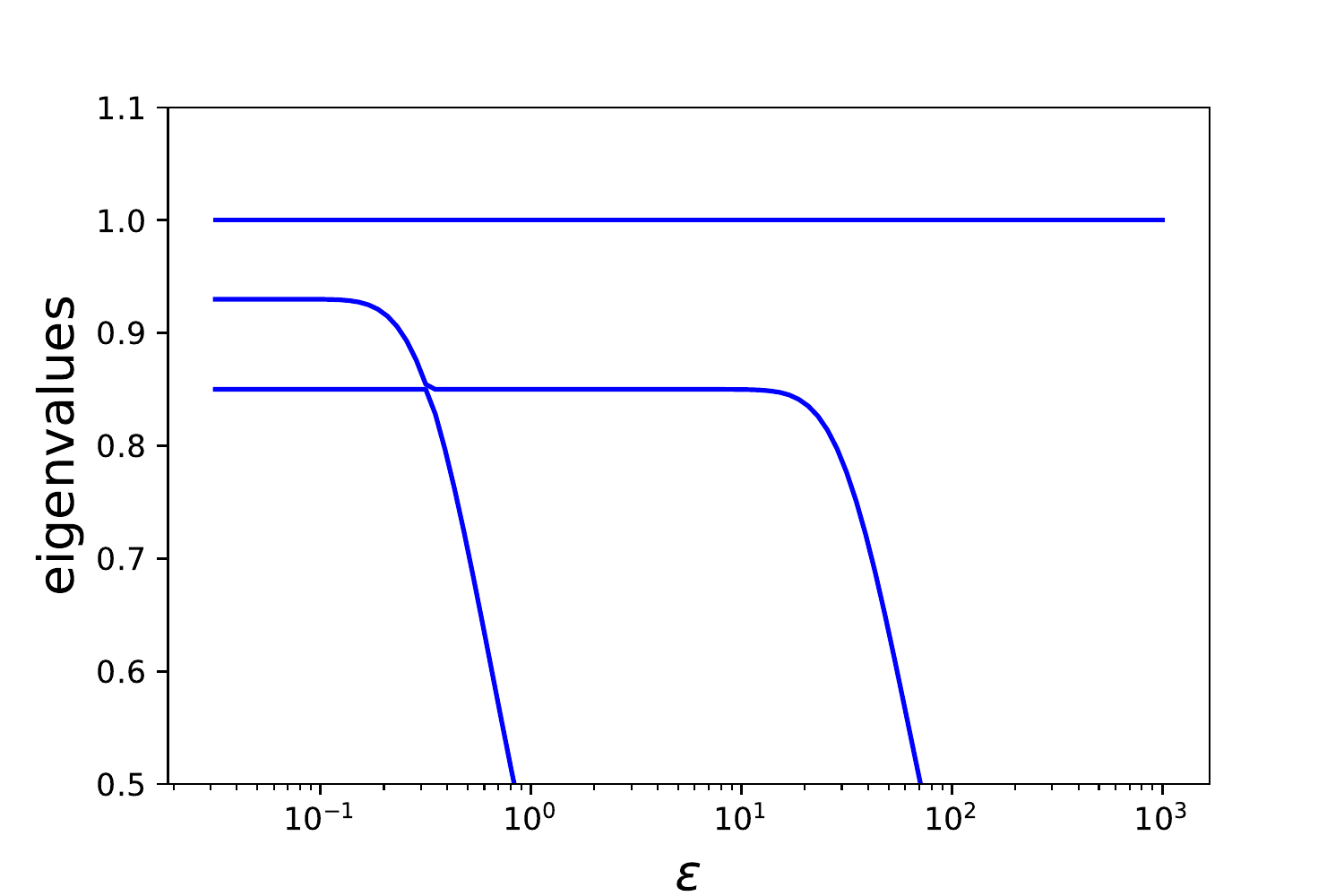}
	\end{subfigure}
    \caption{Simple Markov model showing the qualitative behaviour of the spectrum in Figure~\ref{fig:AD1} for $p_1=0.01$, $p_2=0.05$, $d_1=1$, $d_2=10$.}
    \label{fig:simple_MM}
\end{figure} 

Let us now return to the alanine dipeptide molecule. In Figure \ref{fig:AD2}, we show the eigenvectors at $\lambda_1\approx 0.98$ and $\lambda_2\approx 0.86$ for different values of $\eps$, projected onto the two relevant dihedral angles of the molecule.  A k-means clustering of these (for $\eps=0.01$) yields the three almost invariant (a.k.a.~metastable) sets shown in Figure~\ref{fig:AD3} which correspond to the well-known dominant conformations of the molecule, cf.\ \cite{klus2018kernel}. We stress the fact that we do not use any prior information on the dihedral angles in the computation but merely the raw trajectory data as described above. The angles are only used for the visualization of the eigenfunctions. 

\begin{figure}[ht]
    \centering
    \begin{subfigure}[b]{\textwidth}
        \includegraphics[width=0.24\textwidth]{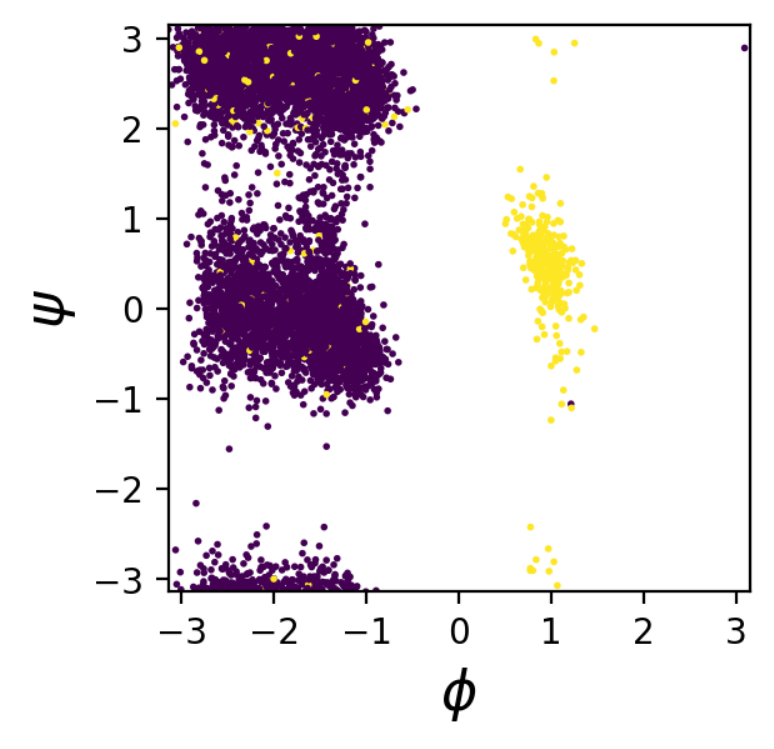}
        \includegraphics[width=0.24\textwidth]{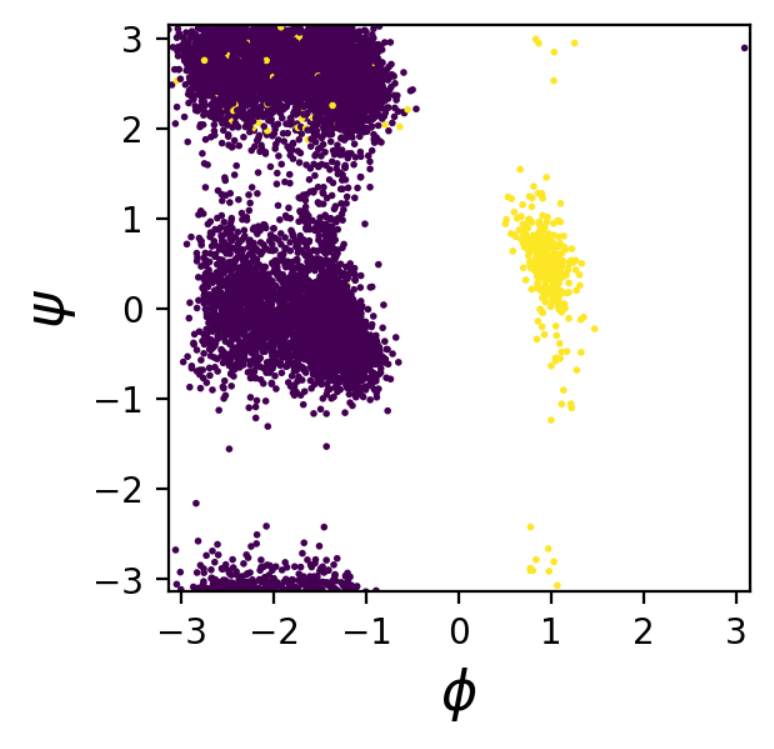}
        \includegraphics[width=0.24\textwidth]{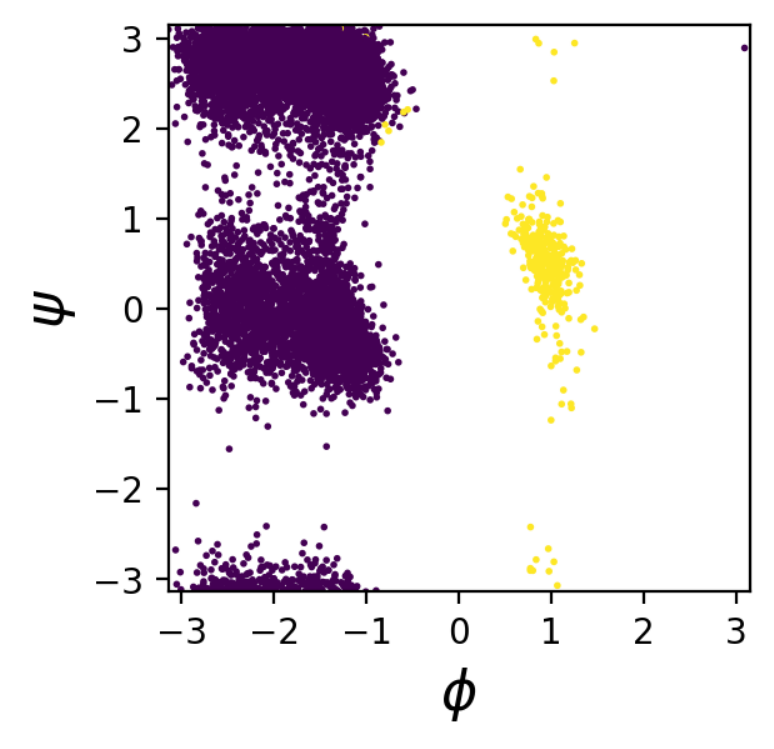}
        \includegraphics[width=0.24\textwidth]{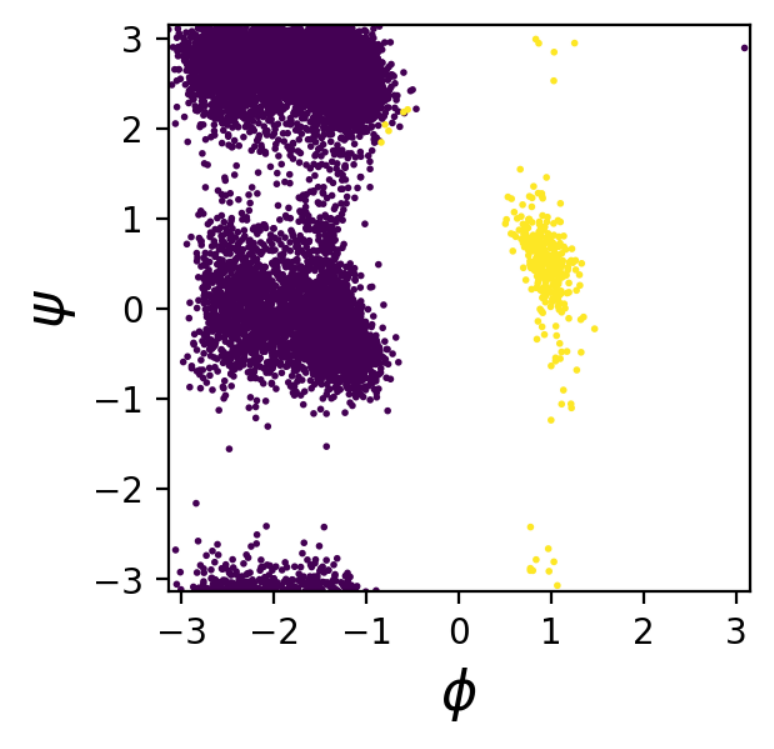}
    \end{subfigure}
    
    \begin{subfigure}[b]{\textwidth} 
        \includegraphics[width=0.24\textwidth]{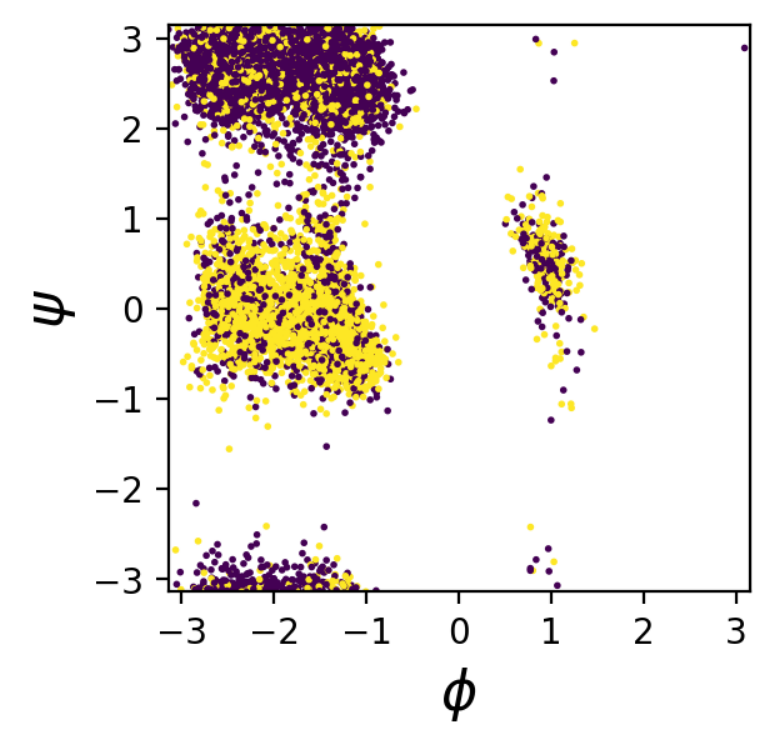}
        \includegraphics[width=0.24\textwidth]{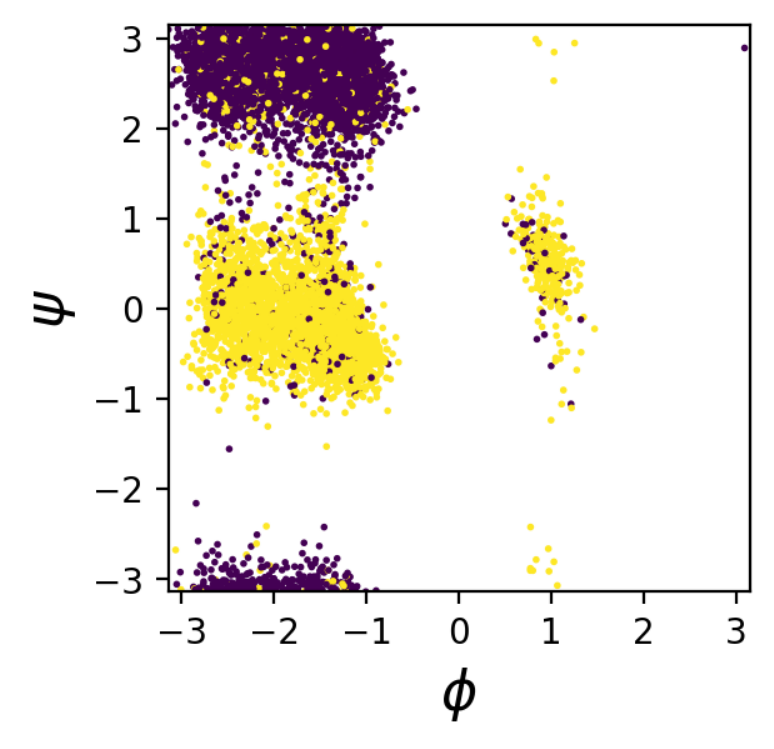}
        \includegraphics[width=0.24\textwidth]{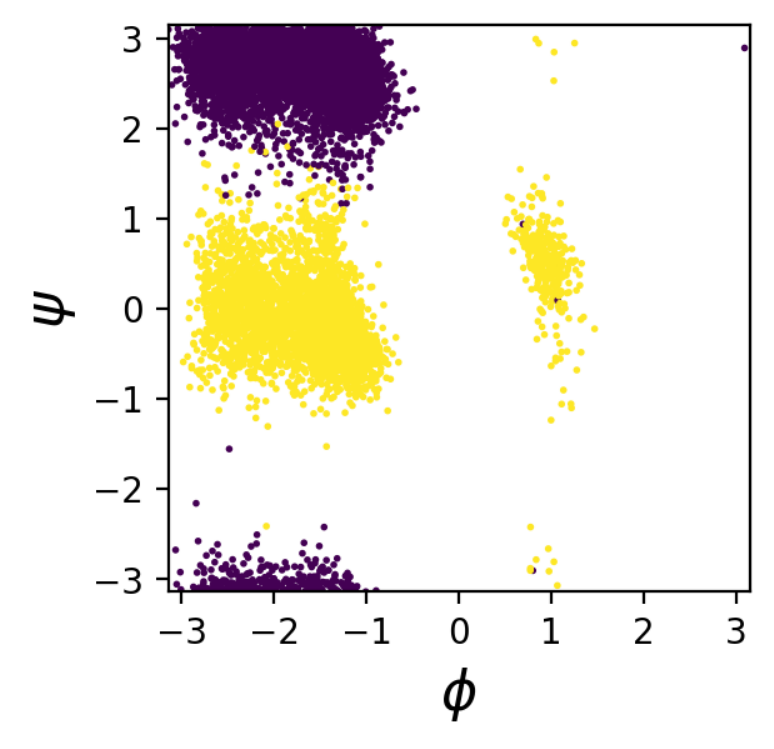}
        \includegraphics[width=0.24\textwidth]{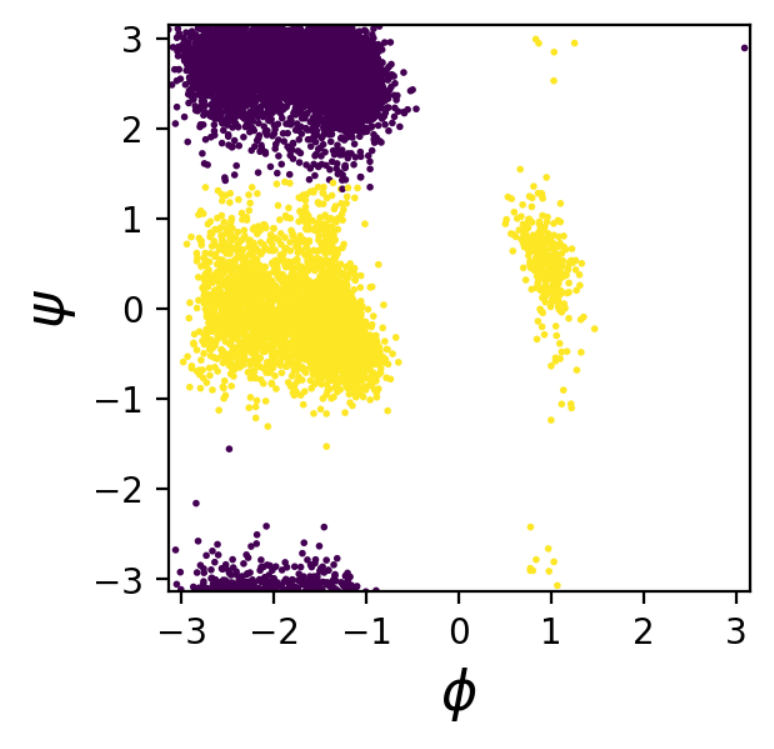}
    \end{subfigure}
    
    \caption{Projections of the eigenvectors of $\Gamma^{N,\eps}$ for the alanine dipeptide model at $\lambda_1\approx 0.98$ (upper row) and $\lambda_2\approx 0.86$ (lower row) for $N=5000$ and $\eps = 3\cdot 10^{-5}, 10^{-4}, 10^{-3}, 10^{-2}$ (from left to right).}  
    \label{fig:AD2}
\end{figure}

\begin{figure}[ht]
    \centering
    \includegraphics[width=0.24\textwidth]{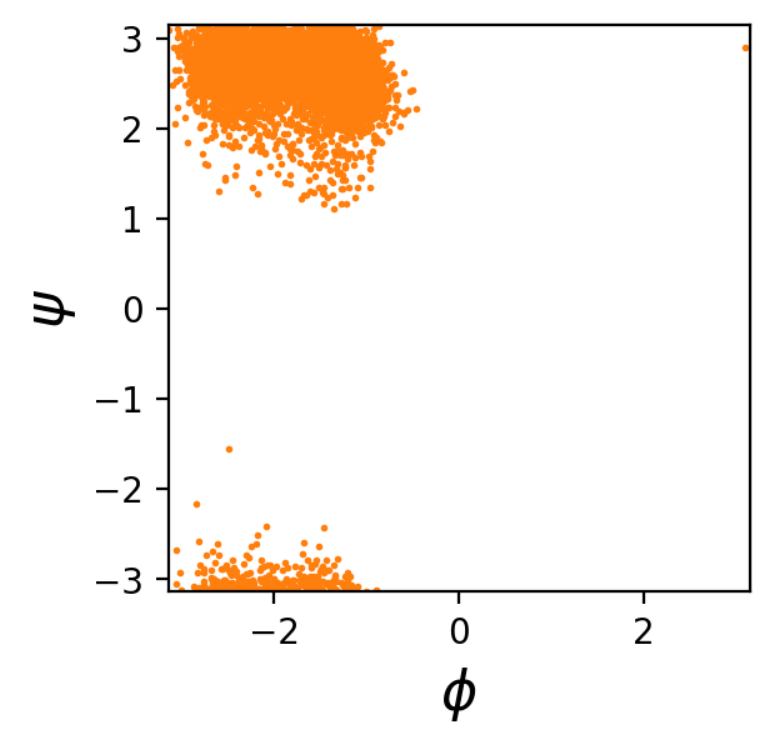}
    \includegraphics[width=0.24\textwidth]{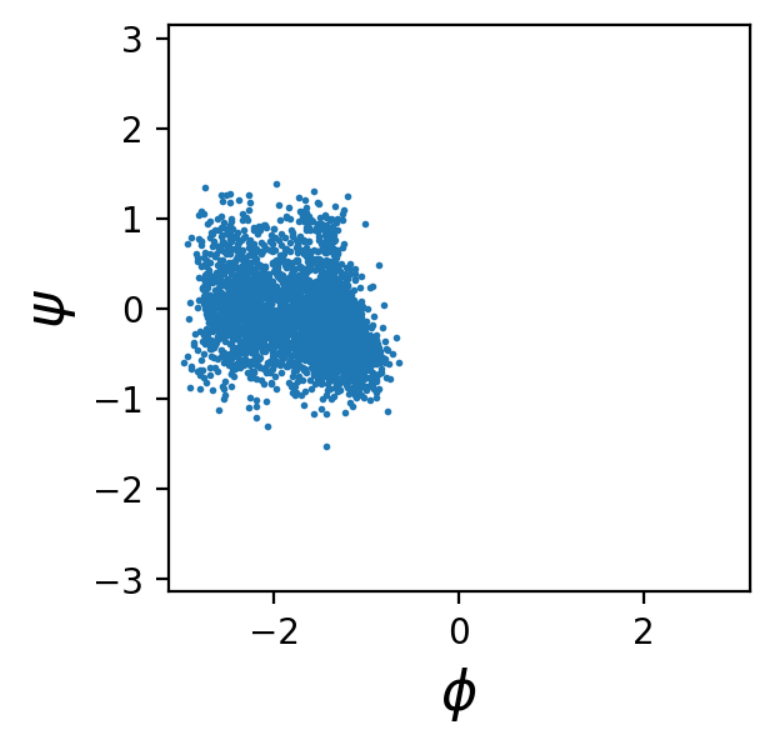}
    \includegraphics[width=0.24\textwidth]{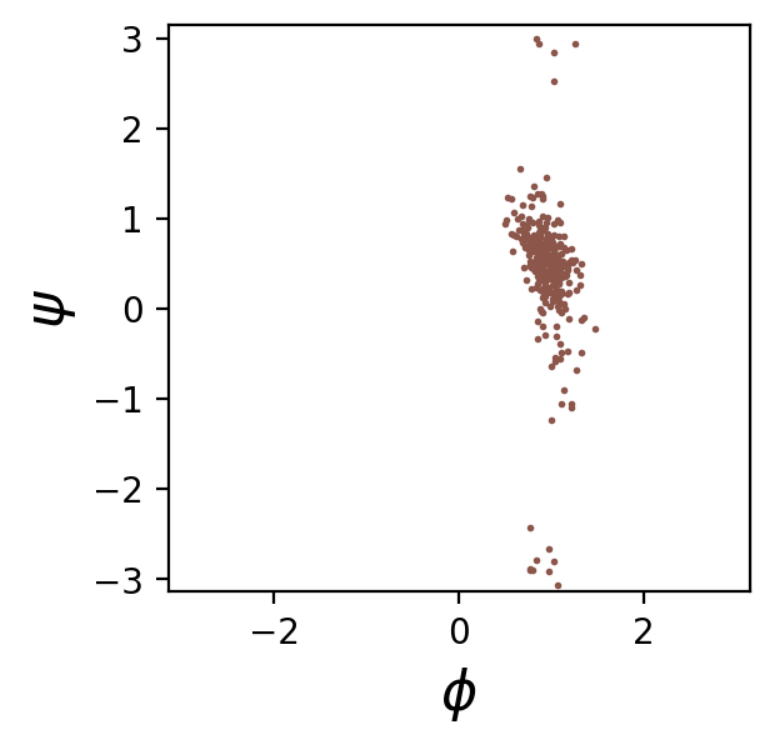}
    \caption{Projections of almost invariant sets as resulting from a k-means clustering of the two eigenvectors show in Figure~\ref{fig:AD2} for $\eps=0.01$.} 
    \label{fig:AD3}
\end{figure}

\section{Comparison to other methods}

We give a brief comparison to related data-based methods for approximating the transfer operator.  

\subsection{Normalized (conditional) Gaussian kernel}
\label{sec:NormalizedGaussians}

In \cite{froyland2021spectral}, a data-based approximation of $T$ was proposed which is related to our approach (the same normalization has also been used in \cite{KoltaiRenger2018}). 
There, a regularized transfer operator $T^\eps:L^2(\mu)\to L^2(\mu)$ also takes the form (cf.\ Section \ref{sec:TEps})
\begin{align}
	\label{eq:TFroyland}
	(T^{\eps} h)(y) & := \int_{\cX} t^{\eps}(x,y)\,h(x) \dd \mu(x), &
    t^{\eps}(x,y) & := \frac{1}{C(x)}\exp\left(\frac{-c(F(x),y)}{\eps}\right)
\end{align}
with $C(x) := \int \exp(-c(F(x),y)/\eps) \dd \mu(y)$.
The kernel $t^{\eps}(x,\cdot)$ is a conditional Gaussian kernel, normalized such that $\int t^\eps(x,y)\dd \mu(y)=1$ for $\mu$-a.e.\ $x$. It becomes equal to \eqref{eq:gFdens} and \eqref{eq:tdens} for the particular choice $\exp(\alpha(F(x))/\eps)=1/C(x)$ and $\beta(y)=0$, which corresponds to initializing $\alpha(F(x))=\beta(y)=0$ and running a single $\alpha$-half-iteration of the Sinkhorn algorithm \eqref{eq:Sinkhorn}.
Like in our approach, the discretized operator $T^{N,\eps}$ results from approximating $\mu$ by the empirical measure $\mu^N$ supported on the data points.

The dominating real spectrum of \eqref{eq:TFroyland} for the Lorenz system is shown in Figure~\ref{fig:spectrum_froyland_1} (left). For large $\eps$ it is qualitatively similar to that of the entropic transfer operator.
Only applying a single Sinkhorn half-iteration means that in general one has $$\int t^\eps(x,y)\dd \mu(x) \neq 1$$ and thus the constant density $h: x \mapsto 1 \in L^2(\mu)$ is generally not a fixed point of $T^\eps$. So this smoothing procedure perturbs the invariant measure of the operator. This also holds for the discrete approximation $T^{N,\eps}$, as shown in Figure \ref{fig:spectrum_froyland_2}. Indeed, the stationary density exhibits substantial fluctuations and several local spurious spikes.

\begin{figure}[ht]
    \centering
    \includegraphics[width=0.45\textwidth]{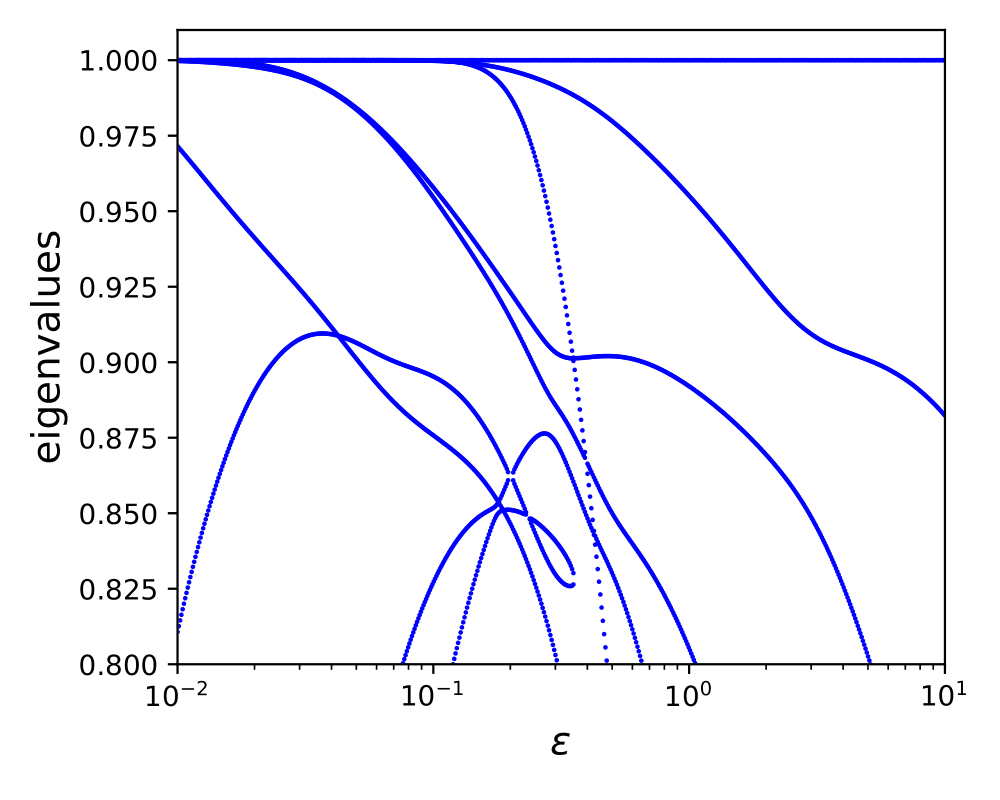}
    \quad
    \includegraphics[width=0.45\textwidth]{images/lorenz/spectrum_n1000_entropic}
    \caption{Real spectrum of $T^{N,\eps}$ with $N=1000$ as a function of $\eps$ for the Lorenz system  using the construction in \eqref{eq:TFroyland} from \cite{froyland2021spectral} (left), and using the entropic transfer operator (right, identical to the left panel in Figure \ref{fig:Lorenz}).} 
    \label{fig:spectrum_froyland_1}
\end{figure}

\begin{figure}[ht]
    \centering
    \includegraphics[width=0.45\textwidth]{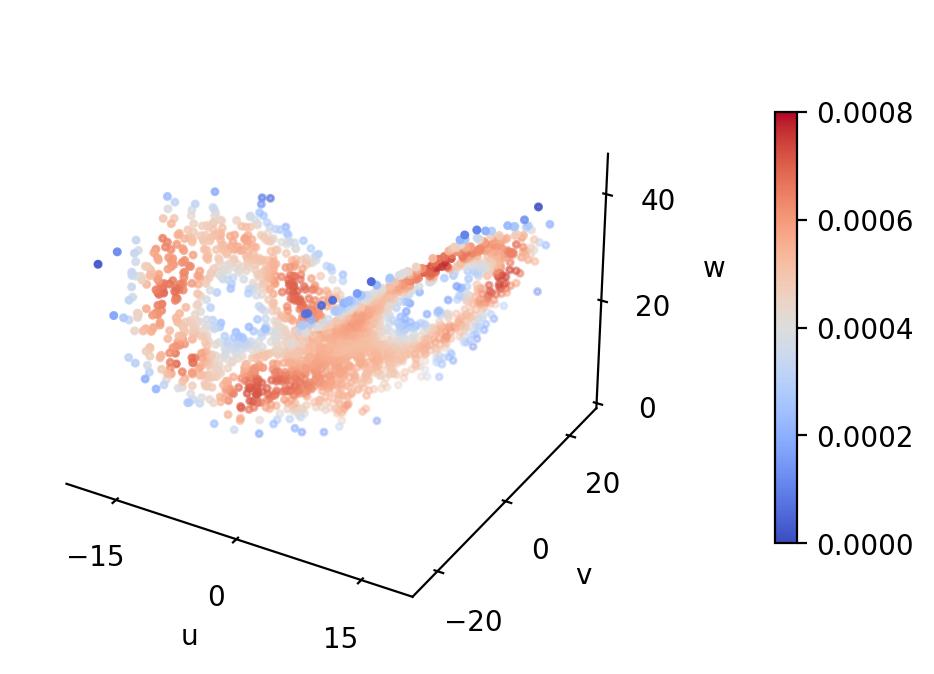}
    \quad
    \includegraphics[width=0.45\textwidth]{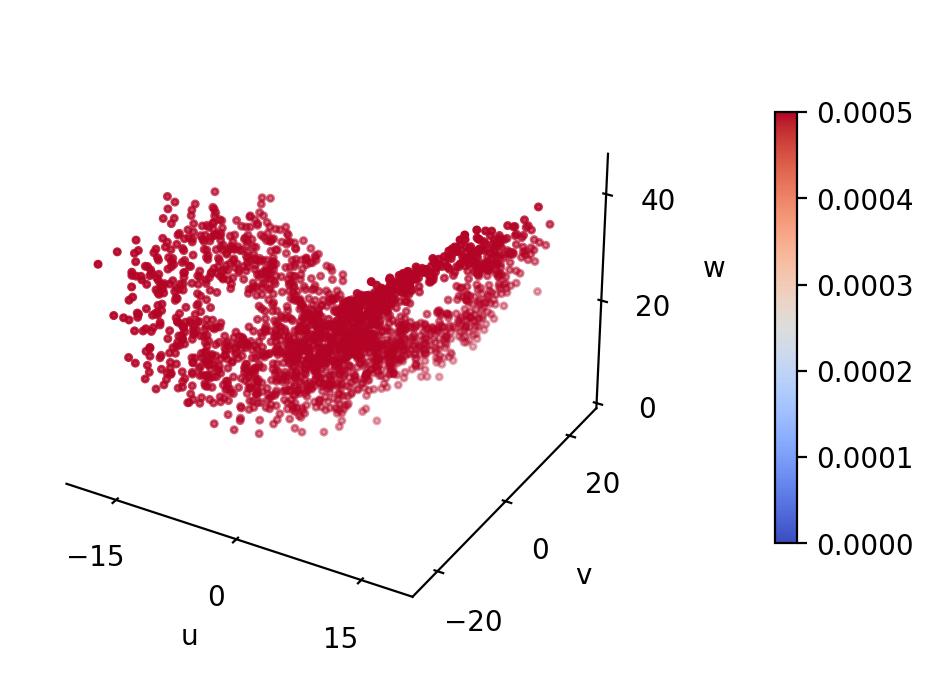}
    \caption{Stationary density of $T^{N,\eps}$ with $N=2000$, $\eps=10$ using the construction \eqref{eq:TFroyland} from \cite{froyland2021spectral} (left) and using the entropic transfer operator (right). Note that the density on the left is not constant while the one on the right is.} 
    \label{fig:spectrum_froyland_2}
\end{figure}

\subsection{Extended dynamic mode decomposition}

Secondly, we compare our approach to extended dynamic mode decomposition (EDMD) \cite{rowley2009spectral,williams2015data}, which is a prominent data based scheme for approximation of the Koopman operator.
Let $\psi_1,\ldots,\psi_J$ be a dictionary of functions in $C(\cX)$. In EDMD, one seeks an approximation $K^N$ of $K$ on $\mathrm{span}(\psi_j)_j$ by enforcing the conditions
\begin{align}
\label{eq:preEDMD}
\psi_j(F(x^N_i)) = (K \psi_j)(x^N_i), \quad i=1,\ldots,N,\ j=1,\ldots,J,
\end{align}
in a least squares sense: define rectangular matrices $A, B \in \R^{N \times J}$ by $A_{i,j}=\psi_j(x_i^N)$ and $B_{i,j}=\psi_j(F(x_i^N))$, and minimize the Frobenius norm
\begin{equation}
\label{eq:EDMD}
\| B - A K^N \|_F^2
\end{equation}
among all $K^N\in\R^{J\times J}$. A least squares loss is preferable since the system \eqref{eq:preEDMD} is often overdetermined ($N\gg J$), the evaluations of $F$ might be subject to noise, and invariance of $\mathrm{span}(\psi_j)_{j}$ under $K$ cannot be expected in general. The solution of \eqref{eq:EDMD} is given in terms of $A$'s pseudo-inverse $A^+$,
\[
K^N = A^+B.
\]

The choice of the dictionary $(\psi_j)_j$ determines how well the spectrum of $K^N$ approximates that of $K$. In the absence of prior information on the eigenvectors of $K$, a popular choice for the dictionary are radial functions centered at the sample points $(x_i^N)_i$. In this case, $J=N$, and $A$ is symmetric. Since $A$ is a square matrix, the minimizer in \eqref{eq:EDMD} can be approximated in a numerically robust way using the modified definition of $A$ by $A_{i,j}=\psi_j(x_i^N)+\sigma \delta_{i,j}$ for some (typically small) regularization $\sigma>0$.

In our numerical experiments, we have chosen radial Gaussians,
\begin{align}
    \label{eq:radial}
\psi_j(x) := \psi^\eps_j(x) := \exp\left(-\frac{c(x_j^N,x)}{\eps}\right).
\end{align}
 We denote the resulting approximation of $K$ by $K^{N,\eps}$. Then, $T^{N,\eps} := A^+ B^\top$ is a corresponding approximation of $T$ on $\mathrm{span} (\delta_{x_i^N})_i$.
This particular choice of dictionary can also be related to the famous kernel trick \cite{williams2016kernel}. For the regularization parameter $\sigma$, we had to choose the rather large value $\sigma=0.1$ in order to stabilize the computation: for smaller $\sigma$, some spectra were not contained in the unit circle anymore and their dependence on $\eps$ was not numerically stable.

\begin{figure}[h]
\centering
    \includegraphics[width=0.4\textwidth]{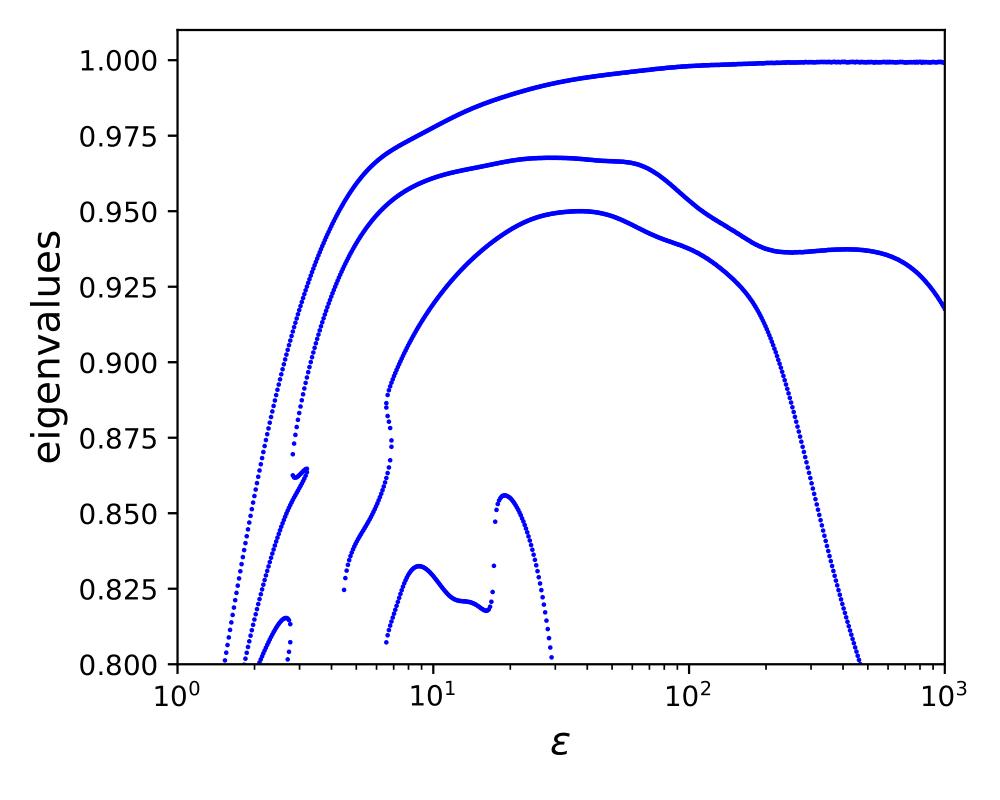}\quad
    \includegraphics[width=0.45\textwidth]{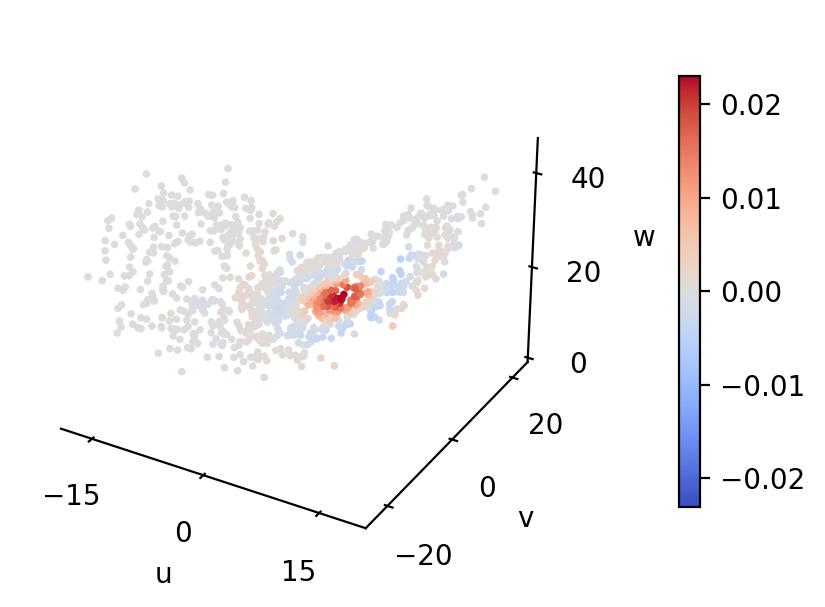}
\caption{Left: the largest real eigenvalues of $K^{N,\eps}$ constructed via EDMD for the Lorenz system with $N=1000$ points as a function of $\eps$. Right: the $i$-th column of the matrix $A^+B^\top$ (cf.~the main text), plotted in color over the point cloud $(x_j^N)_j$.  Note that this column contains negative entries and that thus $A^+B^\top$ is not a Markov matrix.}
\label{fig:LorenzKernelEDMD}
\end{figure}

The results of our numerical experiments are given in Figure~\ref{fig:LorenzKernelEDMD}. The left panel shows the largest real eigenvalues of $T^{N,\eps}$ for $N=1000$ as a function of $\eps$. In an intermediate range of $\eps$, the dominating real eigenvalues of $T^{N,\eps}$ (and the corresponding eigenvectors) encode qualitatively similar almost-invariant structures as the entropic transfer operator, cf.~Section \ref{sec:Method}, although for different values of $\eps$. The figure in the right panel illustrates a disadvantage of EDMD: the operators $K^{N,\eps}$ and $T^{N,\eps}$ are no longer Markov operators, i.e.~one loses the interpretation of jump probabilities between samples.

There is a significant formal similarity between EDMD and our approach: for our choice of Gaussian radial basis functions in \eqref{eq:radial}, the operator $G^{N,\eps}$ is obtained by a row- and column-wise rescaling of the corresponding matrix $B$ through the dual variables $\alpha$ and $\beta$ (see \eqref{eq:d-entropicsolution} and Section \ref{sec:Method}). However, the role played by $\eps$ in EDMD is a slightly different one. There is no blur on the length scale $\eps^{1/2}$ in $K^{N,\eps}$ as for entropic transfer operators. The damping of eigenvalues in $K^{N,\eps}$ depends on a more complex interplay between $\eps$, $N$ and the inversion regularization weight $\sigma$. For example, for $F(x)=x$ and $\sigma=0$, one finds that $A=B$ and thus $K^{N,\eps}$ is the identity operator on the span of the $\psi_j$, independently of the value of $\eps$. Also, the asymptotics in EDMD for small and large values of $\eps$ are different than for entropic transport: as $\eps \to \infty$, all basis functions become increasingly aligned, i.e.~$A_{i,j} \to 1$ and $B_{i,j}\to 1$ for all $i,j$, which implies that all eigenvalues of $A^+B$ approach zero, except for one, which tends to $1$ for the eigenvector with all entries one. For  $\eps\to 0$, we have $B_{i,j} \to 0$ in the generic case where $x^N_i \neq F(x^N_j)$ for all $i,j$, whereas $A^+$ remains bounded, and so all eigenvalues of $K^{N,\eps}$ approach zero.

\subsection{Diffusion maps}

Another conceptually related method are diffusion maps \cite{Coifman2006}. On a pair $(\cX,\mu)$ one defines a Markov operator $T^\eps$ via its kernel
$$t^{\eps}(x,y) := \frac{1}{C(x)}\exp\left(\frac{-c(x,y)}{\eps}\right)$$
with $C(x) := \int \exp(-c(x,y)/\eps) \dd \mu(y)$.
As above, $T^{N,\eps}$ can be approximated by using $\mu^N$ as substitute for $\mu$.
Note that this operator does not explicitely incorporate the map $F$, and thus its spectrum contains no direct information on the dynamics. Implicitly, the invariant measure $\mu$ may of course depend on $F$. Instead, the specturm of this $T^\eps$ contains information about the distribution of $\mu$ and its dominant eigenfunctions can be used to obtain a lower-dimensional embedding of $(\cX,\mu)$. On this embedding one may then apply Ulam's method, which may be unfeasible on the original high-dimensional dynamics, see for instance \cite{KoltaiWeiss2020}.
In comparison, the method of Section \ref{sec:NormalizedGaussians} and entropic transfer operators can be interpreted as mesh-free versions of Ulam's method that do not require a separate dimensionality reduction.

In \cite{SingerPNAS2009} it is proposed to estimate the metric $d$ (which induces the cost $c=d^2$) for the definition of $t^\eps$ from the local diffusivity of a stochastic dynamical system, thus assigning higher emphasis on `slow' coordinates in the subsequent low-dimensional embedding.
While stochastic systems are beyond the scope of the present paper, if such a relevant data-driven metric $d$ were available it could potentially also be used in the definition of the entropic transfer operator. 

\section{Conclusion} 

In this article we have introduced a new method for the regularization, discretization and numerical spectral analysis of transfer operators by using entropic optimal transport.
The analysis of the torus shift map and the numerical experiments indicate that this might be a promising avenue to follow.

Due to its Lagrangian nature, the method is readily applicable to problems with high-dimensional dynamics, in contrast to grid- or set-oriented methods, since no explicit covering of the support of the invariant measure $\mu$ needs to be constructed.
However, the performance of the numerical method depends on whether the discretization captures the features of interest.
The results in Section~\ref{sec:shift map} give a first impression on how ``feature size'', discretization scale and blur scale may be intertwined.
Clearly, features below the blur scale will not be uncovered.
On the other hand, the blur scale needs to be above the discretization scale since otherwise the discrete transfer operator approximation degenerates to a permutation matrix.

A better analytical understanding between the three scales beyond the shift map on the torus would therefore be a relevant open question.
For example, the invariant measure in the torus example has full support and thus the required relation between $N$ and $\eps$ was given by $N=n^d \gg \eps^{-d/2}$, i.e.~the required number of points to resolve a given blur scale increases exponentially with $d$, the so-called curse of dimensionality. In other cases, when the dimension of the support of $\mu$ is lower than the dimension of $\cX$, it seems plausible that $N$ scales based on the dimension of the former, not the latter, as suggested by the alanine dipeptide example in $d=30$. Can this be established rigorously? We refer to \cite{WeedBach19,SpikedWasserstein2019,MunkLCA22} for related results. When the number $N$ of points is small, can we still expect to reliably recover features at least on coarse scales? These questions may be related to the improved sample complexity of entropic optimal transport as compared to the unregularized version \cite{Genevay19a}.
Such results would confirm that our method can extract some spectral information even from coarse approximations of $\mu$, as indicated by some of the numerical experiments.

While here we formulated our approach for an underlying deterministic map, an application to a stochastic model is straightforward: The original transfer operator $T$ will then be defined directly in terms of a stochastic transition function, it can again be combined with entropic smoothing.  In fact, computationally a single realization of some random trajectory of the stochastic model might already be sufficient, if it samples the support of the invariant measure well enough.  

\section{Acknowledgments}

This research was supported by the DFG Collaborative Research Center TRR 109, ``Discretization in Geometry and Dynamics''. BS was supported by the Emmy Noether Programme of the DFG (project number 403056140).

\bibliographystyle{abbrv}
\bibliography{circleshift}

\end{document}